\documentclass[12pt,reqno]{amsart}

\newcommand{\sech}{\operatorname{sech}}

\usepackage{amssymb}
\usepackage{amsthm}
\usepackage{amsxtra}

\usepackage{fancyhdr}
\usepackage{graphicx}
\usepackage{color}
\usepackage{amsmath}
\usepackage{listings} 
\usepackage{float}
\usepackage{cases} 
\usepackage{threeparttable}
\usepackage{dcolumn}
\usepackage{multirow}
\usepackage{booktabs}

\theoremstyle{mtheorem}
\newtheorem{mtheorem}{Theorem}

\newtheorem{theorem}{Theorem}[section]

\newtheorem{proposition}[theorem]{Proposition}
\newtheorem{lemma}[theorem]{Lemma}

\numberwithin{equation}{section}

\title[Solutions to gKdV 
with low power nonlinearity]{Well-posedness and dynamics of solutions to the  generalized KdV with low power nonlinearity} 

\author[I. Friedman]{Isaac Friedman}
\address{Department of Mathematics \\Otterbein University,  Westerville, OH, USA }
\curraddr{}
\email{friedman2@otterbein.edu}

\author[O. Ria\~no]{Oscar Ria\~no}
\address{Department of Mathematics \& Statistics\\Florida International University,  Miami, FL, USA}
\curraddr{}
\email{orianoca@fiu.edu}

\author[S. Roudenko]{Svetlana Roudenko}
\address{Department of Mathematics \& Statistics\\Florida International University,  Miami, FL, USA}
\curraddr{}
\email{sroudenko@fiu.edu}

\author[D. Son]{Diana Son}
\address{Department of Mathematics \\University of Tennessee Knoxville, Knoxville,  TN, USA}
\curraddr{}
\email{dson2@vols.utk.edu}

\author[K. Yang]{Kai Yang}
\address{Department of Mathematics \& Statistics\\Florida International University,  Miami, FL, USA}
\curraddr{}
\email{yangk@fiu.edu}

\subjclass[2010]{35Q53, 35Q35, 35B40, 35B44, 65M70, 65N35} 

\keywords{generalized KdV, well-posedness, weighted spaces, low power nonlinearity, soliton resolution, soliton interaction}

\begin{document}

\begin{abstract}
We consider two types of the generalized Korteweg\,-\,de\,Vries equation, where the nonlinearity is given with or without absolute values, and, in particular, including the low powers of nonlinearity, an example of which is the Schamel equation. We first prove the local well-posedness of both equations in a weighted subspace of $H^1$ that includes functions with polynomial decay, extending the result of \cite{LinaresMiyaGus} to fractional weights. We then investigate solutions numerically, confirming the well-posedness and  extending it to a wider class of functions that includes exponential decay. 
We include a comparison of solutions to both types of equations, in particular, we investigate soliton resolution for the positive and negative data with different decay rates. Finally, we study the interaction of various solitary waves in both models, showing the formation of solitons, dispersive radiation and even breathers, all of which are easier to track in nonlinearities with lower power.
\end{abstract}

\maketitle

\tableofcontents

\section{Introduction}

In this paper we study solutions to the Cauchy problem for the two types of the generalized Korteweg-de Vries equation: one, gKdV, with the nonlinearity powers $\alpha$ extended to include the fractional exponents 
\begin{equation}\label{gKdV} 
    \begin{aligned}
      \left\{\begin{aligned}
      &\partial_t u+\partial_x^3 u\pm u^{\alpha}\partial_x u=0, \qquad x, t \in \mathbb{R},  \\
      &u(x,0)=u_0,
      \end{aligned}\right.
    \end{aligned}
\end{equation}
where the power $\alpha=\frac{m}{k}$ with $m,k\geq 1$ odd integers; and the second one, GKdV,  with the absolute value incorporated into the nonlinearity
\begin{equation}\label{GK} 
    \begin{aligned}
      \left\{\begin{aligned}
      &\partial_t u+\partial_x^3 u\pm |u|^{\alpha}\partial_x u=0, \qquad x, t \in \mathbb{R}, \, \, \\
      &u(x,0)=u_0,
      \end{aligned}\right.
    \end{aligned}
\end{equation}
where $\alpha>0$. 

The gKdV and GKdV equations can be regarded as extensions of the $k$-generalized KdV equation 
\begin{equation}\label{kgKdV}
    \begin{aligned}
 \partial_t u+\partial_x^3 u+ u^{k}\partial_x u=0, \qquad x, t \in \mathbb{R}, \, \, k\in \mathbb{Z}^{+}.  
    \end{aligned}
\end{equation}
Note that \eqref{kgKdV} is obtained by setting $\alpha=k \in \mathbb{Z}^{+}$ in the gKdV equation \eqref{gKdV}. The well-known KdV equation with $k = 1$ was first proposed as a model for unidirectional propagation of nonlinear dispersive long waves, see \cite{KdV}. The integer cases $k\geq 2$ have been used in several physical contexts, such as shallow-water waves with weakly non-linear restoring forces, or long internal waves in
a density-stratified ocean, or ion-acoustic waves in plasma, among many others (e.g., see \cite{AbdelouhabBonaFellandSaut1989, BenjaminBonaMahony1972, BonaColinLannes2005, BonaPritchardScott1981, Craig1985, JeffreyKakutani1972, Bona1981, ScottChuMcLau1973}). 
The modular power nonlinearity as in \eqref{GK} has also been used in physics; for example, when $\alpha \in (0,1)$ it is studied in models of non-Maxwellian trapped electrons and description of their dynamics in ion-acoustic solitary waves, \cite{MS2006}. A special case of $\alpha = \frac12$ is called the Schamel equation,
\begin{equation}\label{E:Schamel}
\partial_t u+\partial_x^3 u +|u|^{\frac12}\partial_x u=0, \qquad x, t \in \mathbb{R}, 
\end{equation}
which was derived by Schamel in early 70's to incorporate the presence of trapped electrons with flat-topped electron distribution function in weakly nonlinear ion-acoustic waves \cite{Sch1973}, other models are described in \cite{PDT2018}, see also \cite{PSKP2021, CNP2018}.  

Regarding the well-posedness of the initial value problem associated to the $k$-generalized KdV equation \eqref{kgKdV} in $H^s$-spaces, we recall the following results:
\begin{itemize}
    \item $k=1$ with $s\geq -1$, see Killip and Vi{\c s}an \cite{KillipVisan2019} (see also \cite{BourgainI1993,Guo2009,KenigPonceVega1993,KenigPonceVega1996,
    ColliandeKeelStaffiTakaTao2003,ChirstColliandeTao2003} for results when $s\geq -\frac{3}{4}$).
    \item $k=2$, Kenig-Ponce-Vega \cite{KenigPonceVega1993} obtained the local well-posedness when 
$s\geq \frac{1}{4}$; Colliander, Keel, Staffilani, Takaoka, and Tao \cite{ColliandeKeelStaffiTakaTao2003} showed the global well-posedness if $s > \frac{1}{4}$; the global well-posedness for the case $s=\frac{1}{4}$ was established by Guo \cite{Guo2009} and Kishimoto \cite{Kishimoto2009}.
    \item $k=3$, see Gr\"unrock \cite{Grunrock2005} for the local well-posedness when $s>-\frac{1}{6}$, and Tao \cite{Tao2007} for $s=-\frac{1}{6}$; for the global well-posedness when $s>-\frac{1}{42}$, see 
Gr\"unrock, Panthee and Drumond Silva \cite{GrunPantheSilva2007}.
\item $k\geq 4$, for the local well-posedness when $s \geq \frac{k-4}{2k}$, see Kenig, Ponce and Vega \cite{KenigPonceVega1993}.
\end{itemize}

In \cite{Kato1983, FonsecaLinaresPonce2015, Nahas2012}, the well-posedness for the initial value problem \eqref{kgKdV} was established in weighted spaces $H^s(\mathbb{R})\cap L^2(|x|^r\, dx)$ with $s\geq r>0$, where the regularity assumption is $s\geq \max\{s_k,0\}$ with $s_1=-\frac{3}{4}$, $s_2=\frac{1}{4}$ and $s_k=\frac{k-4}{2k}$ for $k\geq 4$. In \cite{Kato1983} with $k=1$, the persistence of solutions was shown in weighted spaces $L^2(e^{\beta x}\, dx)$, $\beta>0$ (for uniqueness results in the above weighted spaces, refer to \cite{EscaKenigPonVega2007,IsazaLinaresPonce2013} and references therein). 

For the GKdV equation the existence of solutions for low powers of nonlinearity, $0<\alpha<1$ in \eqref{GK}, was addressed in Linares, Miyazaki and Ponce \cite{LinaresMiyaGus} in a subspace of $H^1$. The approach to obtain the existence was based on the work of Cazenave and Naumkin \cite{CazNaum2016}, where authors developed a method to obtain local and global well-posedness for the NLS equation with low powers of nonlinearity.
We emphasize the work \cite{LinaresMiyaGus} as this paper extends and improves the results of that work. The well-posedness in Fourier-Lebesgue type spaces $\hat{L}^r$ for some range of $\alpha$ has been studied in \cite{MS2016}, the existence of non-scattering solutions in some range of subcritical setting was obtained in \cite{MS2017}.

Formally, solutions of the gKdV and GKdV equations satisfy the mass and $L^1$-type conservation laws:
\begin{align}
& M[t]=\int \big(u(x,t)\big)^2\, dx=M[0], \\
& \int u(x,t)\, dx=\int u(x,0)\, dx.  
\end{align}
The energy is also conserved: in the gKdV case
\begin{equation}
    \begin{aligned}\label{E:gKdV}
    E_{gKdV}[t]=\frac{1}{2}\int |\partial_x u(x,t)|^2\, dx \mp \frac{1}{(\alpha+1)(\alpha+2)}\int \big( u(x,t) \big)^{\alpha+2}\, dx=E_{gKdV}[0],
    \end{aligned}
\end{equation}
and in the GKdV case
\begin{equation}\label{E:GKdV}
    \begin{aligned}
    E_{GKdV}[t]=\frac{1}{2}\int |\partial_x u(x,t)|^2\, dx \mp \frac{1}{(\alpha+1)(\alpha+2)}\int |u(x,t)|^{\alpha+2}\, dx=E_{GKdV}[0].
    \end{aligned}
\end{equation}
In a few very specific cases the gKdV equation has infinitely many conserved quantities: in the KdV ($\alpha=1$) and the modified KdV ($\alpha =2$); 
these models are referred to as completely integrable. (No other cases of gKdV or GKdV are known to be completely integrable.) One of the 
advantages of the complete integrability is a possibility of solving equations (with a sufficiently fast decay of initial data\footnote{For example, satisfying the Fadeev condition $\int_{\mathbb R} (1+|x|)|u_0(x)| \, dx < \infty$.}) as well as studying the interaction of solitary waves (called solitons) via the inverse scattering method, for which the literature is enourmous (see for instance, \cite{Miura1976}); however, if the decay is too slow, then the inverse scattering is not applicable.  

Both \eqref{gKdV} and \eqref{GK} equations are invariant under the scaling: if $u$ solves one of them, then so does
\begin{equation}\label{Scalein}
u_{\lambda}(x,t)=\lambda^{\frac{2}{\alpha}}u(\lambda x, \lambda^3 t), \quad \lambda >0.
\end{equation}
Consequently, the (homogeneous) Sobolev space $\dot{H}^{s_c}$ is invariant under the scaling \eqref{Scalein}  when $s_c=\frac{1}{2}-\frac{2}{\alpha}$. 

The traveling (solitary) wave solutions for both equations are of the form
\begin{equation}\label{E:solitons}
u(x,t)=Q_{c}(x-ct-c_0),
\end{equation}
where $c>0$ denotes the speed of propagation, $c_0$ is an initial shift, and $Q_c$ is the rescaled ground state solution $Q$, 
\begin{equation}\label{E:Q_c}
Q_c(x)=c^{\frac{1}{\alpha}}Q(c^{\frac{1}{2}} x ),
\end{equation} 
where $Q$ is taken to be a smooth, positive, vanishing at infinity solution in the GKdV case of the equation:
\begin{equation}\label{eq:ground}
-Q+Q''+ \frac{1}{(\alpha+1)}|Q|^{\alpha}Q=0, 
\end{equation}
and in the gKdV case
\begin{equation}\label{eq:ground2}
-Q+Q''\pm \frac{1}{(\alpha+1)}Q^{\alpha+1}=0.
\end{equation}
Though technically the equations \eqref{eq:ground} and \eqref{eq:ground2} are different, the positive (ground state) solutions are the same in both cases (in \eqref{eq:ground2} with the focusing positive sign in front of the nonlinearity). We remark that for the defocusing GKdV equation (i.e., with the negative nonlinearity $-|u|^{\alpha}\partial_x u$), the non-existence of ground state solitary waves can be established by standard Pohozaev identities. The equation \eqref{eq:ground} has an explicit solution 
\begin{equation}\label{explground}
    Q(x)=\Big(\tfrac{(\alpha+1)(\alpha+2)}{2} \Big)^{\frac{1}{\alpha}}\sech^{\frac2{\alpha}} \big(\tfrac{\alpha x}{2}\big),
\end{equation}
see Figure \ref{F:profile QallP} for a comparison for different $\alpha$'s. 
Note that in the focusing case of \eqref{eq:ground2} (with a plus sign) the ground state $Q$ in \eqref{explground} is also a solution of \eqref{eq:ground2}. 

\begin{figure}[ht]
\includegraphics[width=0.6\textwidth,height=0.25\textheight]{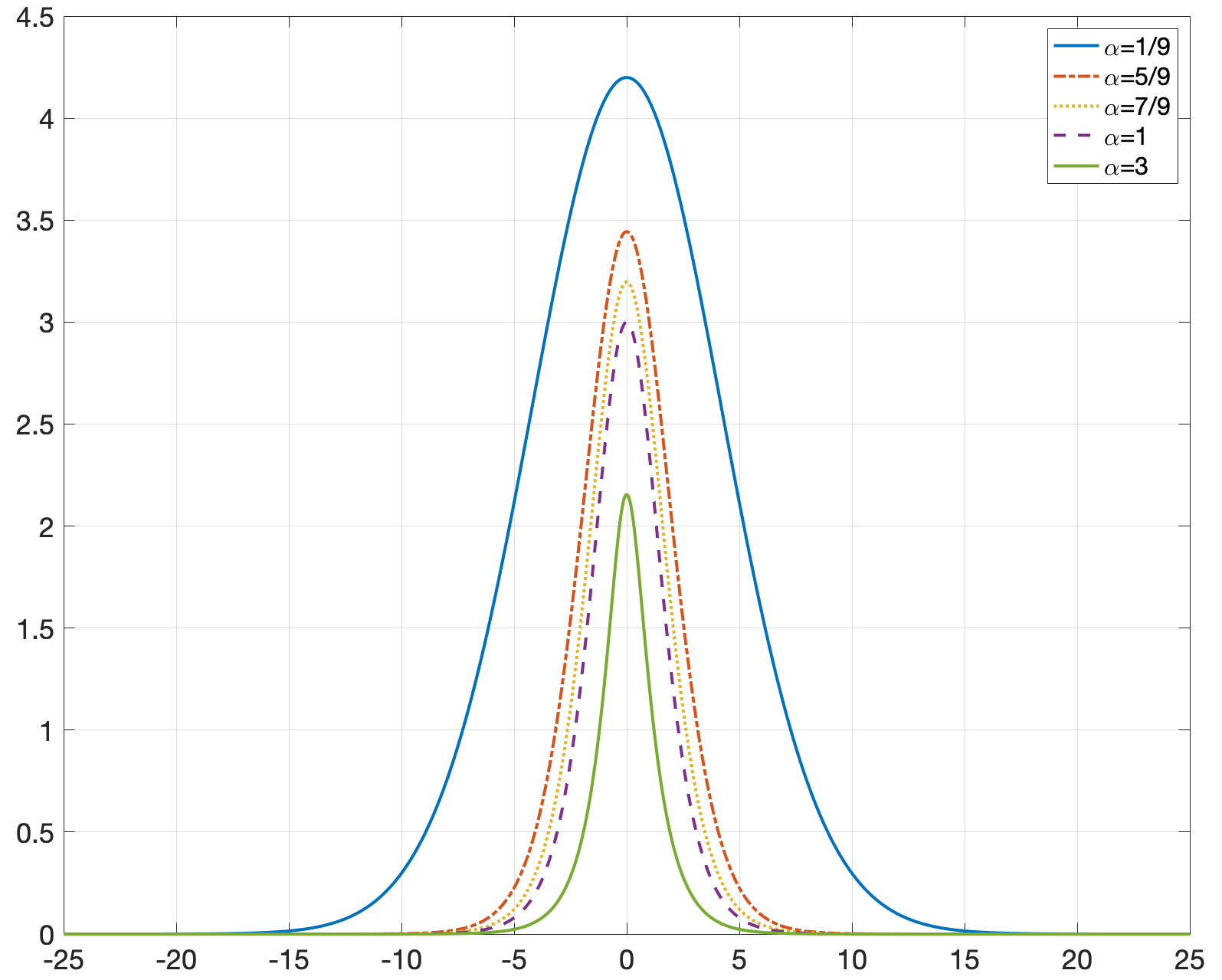}
\caption{\footnotesize The ground state profiles $Q$ from \eqref{explground} or \eqref{eq:ground}, \eqref{eq:ground2} for $\alpha=\frac19, \frac59, \frac79, 1,3$.}
\label{F:profile QallP}
\end{figure}

The traveling solitary waves \eqref{E:solitons} (often called solitons, though the terminology was defined for the integrable cases of KdV and mKdV) have been a subject of investigation for a long time as they play an important role not only as a fixed shape 
traveling for a sufficiently long time
but also as a threshold for global vs. finite time existence, stability vs. instability of solitary waves, and non-scattering solutions. Stability of solitary waves started in \cite{Benj1972}, with extensions and generalizations in \cite{Bona1975}, \cite{W-CPAM1986}, \cite{GSS1987}, \cite{BSS1987}, \cite{Albert1999}, \cite{PW1994}, \cite{MM2000}, \cite{MM2001} and many other results stemming from these. In most works on KdV-type equations, it is standard to consider positive {\it integer} power $k$ in the nonlinearity as defined in \eqref{kgKdV}. 

In this paper we study solutions to both the gKdV and GKdV equations with $\alpha>0$, specifically investigating the low power of nonlinearity $0<\alpha<1$, including the case of Schamel's equation \eqref{E:Schamel} ($\alpha=\frac12$). 
One of the main goals is to establish the local well-posedness
of solutions for the initial-value problems of the equations \eqref{gKdV} and \eqref{GK}. 
Since the nonlinearity is not necessarily smooth, inspired by the results in \cite{CazNaum2016} (see also \cite{AndyOscarSvet, CazHauNaum2020, CazNaum2018, LinaresMiyaGus, LinaresI, LinaresII, Miyazaki2020}), we introduce a class of initial data that guarantees the existence of local solutions for nonlinearities with such a low regularity. The idea is to consider sufficiently regular, with enough decay, initial conditions satisfying 
\begin{equation}\label{Incond}
    \inf_{x\in \mathbb{R}}\langle x \rangle^{m} |u_0(x)|>0,
\end{equation}
where $\langle x \rangle^m=(1+|x|^2)^{\frac{m}{2}}$, and construct local solutions of the gKdV and GKdV equations from such data, 
satisfying a similar property to \eqref{Incond} in both space and time variables. This in turn compensates the lack of regularity of the map $|u|^{\alpha}\partial_x u$. 

We remark that a similar approach was used before in \cite{LinaresMiyaGus} to obtain local solutions for the GKdV equation with $0<\alpha<1$, and $m$ being an integer in \eqref{Incond}, which was inspred by work of Cazenave-Naumkin for the NLS equation \cite{CazNaum2016}. In this work, we prove the well-posedness with fractional power weights and extend the range of admissible powers $m$.

\begin{mtheorem}\label{mainTHM}
Let $\alpha>0$, $m \in \mathbb{R}^{+}$, $m>\max\{\frac{1}{2\alpha},\frac{1}{2}\}$. Let $s\in \mathbb{Z}$ with $s \geq 2m+4$, and assume that $u_0$ is a complex-valued function such that 
\begin{equation}\label{eqmainTH1}
    u_0 \in H^s(\mathbb{R}), \, \, \langle x\rangle^mu_0 \in L^{\infty}(\mathbb{R}), \, \, \langle x \rangle^m\partial_x^{j}u_0 \in L^2(\mathbb{R}), \, \, j=1,2,3,4,
\end{equation}
\begin{equation}\label{eqmainTH2}
    \|u_0\|_{H^s}+\|\langle x \rangle^m u_0\|_{L^{\infty}}+\sum_{j=1}^4\|\langle x \rangle^m\partial_x^{j}u_0 \|_{L^2}<\delta
\end{equation}
for some $\delta>0$ and
\begin{equation}\label{condi1}
\inf_{x\in \mathbb{R}}\langle x \rangle^m|u_0(x)|=:\lambda>0.
\end{equation}
Then there exist $T=T(\alpha,\delta,s,\lambda)>0$ and a unique solution $u$ of the gKdV equation (or a unique solution $u$ of the GKdV equation with $\alpha=\frac{m}{k}>0$, where $m$, $k$  are odd integers)  in the class
\begin{equation}\label{eqmainTH3}
    u \in C([0,T];H^s(\mathbb{R})), \qquad \langle x \rangle^m \partial_x^{j}u \in C([0,T];L^2(\mathbb{R})), \, \, j=1,2,3,4
\end{equation}
with 
\begin{equation}\label{eqmainTH4}
    \langle x \rangle^m u \in C([0,T];L^{\infty}(\mathbb{R})), \qquad \partial_x^{s+1}u\in L^{\infty}(\mathbb{R}; L^2([0,T])),
\end{equation}
and
\begin{equation}\label{eqmainTH5}
    \sup_{0\leq t \leq T}\|\langle x\rangle^m(u(t)-u_0)\|_{L^{\infty}}\leq \frac{\lambda}{2}.
\end{equation}
Moreover, the map $u_0 \mapsto u(\cdot,t)$ is continuous in the following sense: for any compact $I\subset [0,T]$, there exists a neighborhood $V$ of $u_0$ satisfying \eqref{eqmainTH1} and \eqref{condi1} such that the map is Lipschitz continuous from $V$ into the class defined
by \eqref{eqmainTH3} and \eqref{eqmainTH4}.
\end{mtheorem}

Our proof of Theorem \ref{mainTHM} adapts the arguments in \cite{LinaresMiyaGus} developed for the gKdV initial value problem \eqref{GK} with $0<\alpha<1$. We are able to extend  
the results in \cite{LinaresMiyaGus} to {\it arbitrary} powers $\alpha>0$ and {\it fractional} weights $m$ (not only integers), thus, enlarging the range of possible decay on the initial data (allowing, for example, very slow decay, such as $1/|x|^{1/2+}$ when $\alpha \geq 1$). A key argument is the deduction of Lemma \ref{derivexp2}, which relates the action of fractional weights on solutions of the linear KdV equation.

After we obtain the local well-posedness as described in Theorem \ref{mainTHM}, we study the solutions numerically, in particular, we investigate solutions with initial data decay not given by the theorem. Furthermore, we study the difference between solutions of the gKdV and GKdV equations, comparing them in various settings, including different powers $\alpha$ and initial data.

{\bf Remarks.}
(i) An example of initial data that satisfies the conditions of Theorem \ref{mainTHM} is the following functions (here, $m$ is as in the statement of the theorem)
\begin{equation}\label{E:example}
u_0(x)=\frac{2\lambda e^{i\theta}}{\langle x \rangle^m}+\varphi(x), \quad  \lambda \in \mathbb R, \quad \theta \in \mathbb{R},
\end{equation}
with $\varphi \in \mathcal S(\mathbb{R})$ (the Schwartz class of functions), i.e., the decay at infinity is $1/|x|^m$. 

(ii) Numerically we study solutions to the Cauchy problems \eqref{gKdV} and \eqref{GK} with initial data decaying at infinity as slow as $1/|x|$ (i.e., $|x|^{-\beta}$
with $\beta \geq 1$), see Section \ref{S:Numerical results}.
Note that our local theory in Theorem \ref{mainTHM} assures the existence of solutions, in fact, for a wider class of conditions with $\beta>\max\{\frac{1}{2\alpha},\frac{1}{2}\}$ (for $\alpha > \frac12$). (This is also an improvement of the results in \cite{LinaresMiyaGus}.)

(iii) On the other hand, observe that the class of initial data in Theorem \ref{mainTHM} does not include any exponentially decaying data, for example, the ground state \eqref{explground}. This is due to the fact that the analytical approach in the proof of the theorem can only deal with the polynomial decay (polynomial weights in \eqref{condi1}). Nevertheless, using the numerical approach, we are able to investigate the behavior of solutions which decay exponentially, in particular, we study the ground state initial condition and its various perturbations. It would be interesting to obtain local well-posedness for the exponentially decaying data, which would include 
the ground state, for the above equations with low power nonlinearity.

The remark above shows that our analytical and numerical results complement each other in terms of studying the local behavior of solutions. Beyond the local well-posedness, we study numerically the dynamics of solutions (the formation of solitons and dispersive radiation) as well as the interaction of solitary waves in both types of equations: we start with the ground state initial data and investigate the stability in both settings,  
then we exhibit the difference of solutions between the gKdV and GKdV equations for various types of initial data. We find that starting with positive initial data, solutions converge to a rescaled soliton, or a combination of several solitons traveling to the right, plus the radiation (fast oscillatory dispersing part) traveling to the left, which justifies the soliton resolution conjecture. 
Furthermore, the solitons which are formed in the gKdV model are slightly higher (and thus, faster) than the ones generated by the same data in the GKdV model. The negative initial data generates similar behavior (up to the sign) in the GKdV model, and a very different behavior in the gKdV model, where the simulations on a relatively short time interval show that   
the gKdV solutions are more likely to disperse into the radiation. However, we observe a new phenomenon in gKdV solutions evolving from the negative data (which is easier to observe for small powers $\alpha \ll 1$): after the initial bump disperses left into the radiation, the first negative peak decreases in its height and eventually disappears, which allows the second positive peak to start forming a soliton, followed later by emergence of more and more smaller solitons (see Figures \ref{F:profile Ngau 19}, \ref{F:profile Ngau 59}).
Lastly, we study the interaction of two bump profiles. 
For a more complete picture, we study interactions of different types of multi-bumps, and for example, in the interaction of two Gaussian bumps of opposite sign we also observe breathers in the GKdV equation.

This paper is organized as follows: in Section \ref{S:lwp}, we prove the local well-posedness establishing Lemma \ref{derivexp2} and using the key ingredients from \cite{LinaresMiyaGus}. In Section \ref{S:Numerical results}, we start our numerical investigations of soliton dynamics and stability, and give positive confirmation to soliton resolution conjecture. In Section \ref{S:4} we consider interaction of two solitary waves, first between two soliton profiles (possibly with different signs) and then a combination of solitons and Gaussian profiles. We briefly describe our numerical method in Appendix.

{\bf Acknowledgments.} 
The research of this project started 
during the Summer 2021 REU program ``AMRPU @ FIU" that took place at the Department of Mathematics and Statistics, Florida International University, and was organized under the NSF (REU Site) grant DMS-2050971. In particular,  support of I.F. and D.S. came from that grant. O.R., S.R. and K.Y. were partially supported by DMS-1927258 (PI: S. Roudenko). 
\smallskip

{\bf Notation.}
Given $x\in \mathbb{R}$, $\lfloor x \rfloor$ denotes the greatest integer less than or equal to $x$. The Fourier transform and the inverse Fourier transform of a function $f$ are denoted by $\widehat{f}$ and $f^{\vee}$, respectively.  For $s\in \mathbb{R}$, the Bessel potential of order $-s$ is denoted by $J^s=(1-\Delta)^{s/2}$, equivalently, $J^s$ is defined by the Fourier multiplier with symbol $\langle \xi \rangle^{s}=(1+|\xi|^2)^{s/2}$.  The Riesz potential of order $-s$ is denoted $D^s=(-\Delta)^{s/2}$, that is, $D^s$ is the Fourier multiplier operator determined by the function $|\xi|^s$. Given $\beta\in (0,1)$, we use one of the Stein's derivatives $\mathcal{D}^{\beta}$, see \eqref{Steinsquare} below.

If $A$ denotes a functional space, we define the spaces $L^p_TA$ and $L^{p}_t A$ by the norms
\begin{equation*}
\|f\|_{L^{p}_T A}=\big\| \|f(\cdot,t)\|_{A} \big\|_{L^{p}([0,T])}\,  \text{ and } \, \|f\|_{L^{p}_t A}=\big\| \|f(\cdot,t)\|_{A} \big\|_{L^{p}(\mathbb{R})},
\end{equation*}
respectively, for all $1\leq p\leq \infty$. Similarly, for all $1\leq p,q\leq \infty$ we define
$$
\|f\|_{L^{p}_x L^q_t}=\big\| \|f(x,\cdot)\|_{L^q(\mathbb{R})} \big\|_{L^{p}_x(\mathbb{R})}  \quad \mbox{and} \quad \|f\|_{L^{p}_x L^{q}_T}=\big\| \|f(x,\cdot)\|_{L^q([0,T])} 
\big\|_{L^{p}_x(\mathbb{R})}.
$$ 

We denote by $U(t)=e^{-t \partial_x^3}$, $t\in \mathbb{R}$, the unitary group describing the solution of the linear equation associated to the gKdV equation, in other words,
\begin{equation}\label{Uoperdef}
    \begin{aligned}
    U(t)f(x)=\int e^{it \xi^3+ix\xi} \, \widehat{f}(\xi)\, d\xi,
    \end{aligned}
\end{equation}
provided that $f$ is sufficiently regular with enough decay.


\section{Well-posedness}
\label{S:lwp}

In this part we prove the local well-posedness result stated in Theorem \ref{mainTHM}. We first deduce some key linear estimates between weights and solutions of the linear KdV equation. Then, we prove existence of solutions for the gKdV and GKdV equations.


\subsection{Linear estimates}

Let $\beta\in (0,1)$. We define one of Stein's fractional derivatives (see \cite{Stein1961})
\begin{equation}\label{Steinsquare}
\mathcal{D}^{\beta} f(x)=\left(\int_{\mathbb{R}^N}\frac{|f(x)-f(y)|^2}{|x-y|^{N+2\beta}}\, dy\right)^{1/2}, \, \,  x\in \mathbb{R}^N.
\end{equation}
When $p=2$ and $\beta\in (0,1)$, it follows
\begin{equation} \label{prelimneq}  
\left\|\mathcal{D}^{\beta}(fg)\right\|_{L^2(\mathbb{R}^N)} \leq C( \left\|f\mathcal{D}^{\beta} g\right\|_{L^2(\mathbb{R}^N)}+\left\|g\mathcal{D}^{\beta}f \right\|_{L^2(\mathbb{R}^N)}),
\end{equation}
and also
\begin{equation} \label{prelimneq1}
\left\|\mathcal{D}^{\beta} f\right\|_{L^{\infty}(\mathbb{R}^N)} \leq C( \left\|f\right\|_{L^{\infty}(\mathbb{R}^N)}+\left\|\nabla f\right\|_{L^{\infty}(\mathbb{R}^N)}).
\end{equation}

We recall the following characterization of the spaces $L^p_s(\mathbb{R}^N)=J^{-\beta}L^p(\mathbb{R}^N)$.
\begin{theorem}(Stein \cite{Stein1961})\label{TheoSteDer} 
Let $\beta \in (0,1)$ and $\frac{2N}{N+2\beta}<p<\infty$. Then $f\in L_{\beta}^p(\mathbb{R}^N)$ if and only if
\begin{itemize}
\item[(i)]  $f\in L^p(\mathbb{R}^N)$ and
\item[(ii)]$\mathcal{D}^{\beta}f \in L^{p}(\mathbb{R}^N)$
\end{itemize}
with 
\begin{equation*}
\|J^b f\|_{L^p(\mathbb{R}^N)}=\|(1-\Delta)^{\frac{\beta}{2}} f \|_{L^p(\mathbb{R}^N)} \simeq \|f\|_{L^p(\mathbb{R}^N)}+\|\mathcal{D}^{\beta} f\|_{L^p(\mathbb{R}^N)} \simeq \|f\|_{L^p(\mathbb{R}^N)}+\|D^{\beta} f\|_{L^p(\mathbb{R}^N)}.
\end{equation*} 
\end{theorem}

\begin{proposition}\label{PropDbyJ}
Let $\beta \in [0,1)$, $m \in \mathbb{R}^{+}\cup \{0\}$ and $k \in \mathbb{Z}^{+}\cup \{0\}$, then
\begin{equation}\label{Deriineq}
\begin{aligned}
\|D^{\beta} \big(\langle x \rangle^{m}\partial_{x}^{k}f \, \big)\|_{L^2(\mathbb{R})} \leq C\big( \sum_{\substack{0\leq l\leq k \\ l<m}} \|J^{\beta+k-l}\big(\langle x \rangle^{m-l}f \, \big)\|_{L^2(\mathbb{R})}+\|J^{\beta+k-m} f \,  \|_{L^2(\mathbb{R})}\big),
\end{aligned}
\end{equation} 
where we define by zero the sum over the empty summation.
\end{proposition}
\begin{proof}
The proof is a consequence of Theorem \ref{TheoSteDer} together with properties \eqref{prelimneq} and \eqref{prelimneq1}. For a deduction, see (2.14) in \cite{AndyOscarSvet}.
\end{proof}

We require the following interpolation inequalities.
\begin{proposition}\label{propweight2} For any $a,b>0$, $\theta\in (0,1)$,
\begin{equation}\label{eq1propweight2}
\|\langle x \rangle^{\theta a}\big(J^{(1-\theta)b}f\big)\|_{L^2(\mathbb{R}^N)}\leq C \|J^{b}f\|^{1-\theta}_{L^2(\mathbb{R}^N)}\|\langle x \rangle^{a} f\|_{L^2(\mathbb{R}^N)}^{\theta},
\end{equation}
\begin{equation}\label{eq2propweight2}
\|J^{\theta a}\big(\langle x \rangle^{(1-\theta)b}f\big)\|_{L^2(\mathbb{R}^N)}\leq C \|\langle x \rangle^{b}f\|^{1-\theta}_{L^2(\mathbb{R}^N)}\|J^{a} f\|_{L^2(\mathbb{R}^N)}^{\theta}.
\end{equation}
\end{proposition}

\begin{proof}
We refer to \cite[Lemma 4]{NahasPonce2009}, see also \cite[Lemma 2.7]{LinaresPastorSilva2020}.
\end{proof}

Next, we deduce the following interpolation inequality, which is an extension of \cite[Lemma 2]{LinaresMiyaGus}.

\begin{lemma}\label{Interplemma}
Let $a, b>0$ and $x\in \mathbb{R}$. Then for any $\theta \in [0,1]$ with $(1-\theta)b \in \mathbb{Z}^{+}$,
\begin{equation}\label{Interpoeq1}
    \begin{aligned}
    \|\langle x \rangle^{\theta a}\partial_x^{(1-\theta)b}f\|_{L^2(\mathbb{R})}\leq C(\|\langle x \rangle^{a}f\|_{L^{2}(\mathbb{R})}+\|J^{b} f \|_{L^2(\mathbb{R})}).
    \end{aligned}
\end{equation}
\end{lemma}

\begin{proof}
Let $\mu>0$ and an integer $r\in \mathbb{Z}^{+}\cup\{0\}$. We claim 
\begin{equation}\label{eqInterpoeq2}
    \begin{aligned}
    \|\langle x \rangle^{\mu}\partial_x^{r}f\|_{L^2}\leq C\sum_{\substack{0\leq l \leq r\\ \mu-l>0 }} \|J^{r-l}\big(\langle x \rangle^{\mu-l} f \big)\|_{L^2}+\|J^r f\|_{L^2}.
    \end{aligned}
\end{equation}
Once we establish \eqref{eqInterpoeq2}, Proposition \ref{propweight2} and \eqref{eqInterpoeq2} yield \eqref{Interpoeq1}. To prove \eqref{eqInterpoeq2} we split the expression on the left-hand side as follows
\begin{equation*}
\begin{aligned}
 \langle x \rangle^{\mu}\partial_x^{r}f=\sum_{l=0}^r c_l\partial_x^{r-l}\big(\partial_x^{l}(\langle x \rangle^{\mu})f\big)=&\sum_{\substack{0\leq l \leq r\\ \mu-l\leq 0}} (\cdots) +\sum_{\substack{0\leq l \leq r\\ \mu-l> 0}} (\cdots)=: \, A_1+A_2.
\end{aligned}
\end{equation*}
We rewrite $A_1$ as
\begin{equation*}
    \begin{aligned}
    A_1=\sum_{\substack{0\leq l \leq r\\ \mu-l\leq 0}}c_l[\partial_x^{r-l},\partial_x^{l}(\langle x \rangle^{\mu})]f+\sum_{\substack{0\leq l \leq r\\ \mu-l\leq 0}}c_l\partial_x^{l}(\langle x \rangle^{\mu}) \partial_x^{r-l}f,
    \end{aligned}
\end{equation*}
where $[A,B]=AB-BA$ denotes the commutator between the operators $A$ and $B$. Then, since $|\partial^{l+\beta}_x(\langle x \rangle^{\mu})|\leq c<\infty$, whenever $\mu-l \leq 0$ and $\beta \geq 0$ integer, it follows
\begin{equation*}
    \|A_1\|_{L^2}\leq C \|J^{r}f\|_{L^2}.
\end{equation*}
To estimate $A_2$, we write
\begin{equation*}
    \begin{aligned}
    A_2=&\sum_{\substack{0\leq l \leq r\\ \mu-l> 0}} c_l \partial_x^{r-l}\Big(\frac{\partial_x^{l}(\langle x \rangle^{\mu})}{\langle x \rangle^{\mu-l}} \langle x \rangle^{\mu-l} f\Big)\\
    =&\sum_{\substack{0\leq l \leq r\\ \mu-l>0 }}c_l \Big[\partial_x^{r-l},\frac{\partial_x^{l}(\langle x \rangle^{\mu})}{\langle x \rangle^{\mu-l}}\Big] \big( \langle x \rangle^{\mu-l}f\big)+\sum_{\substack{0\leq l \leq r\\ \mu-l> 0}}c_l \frac{\partial_x^{l}(\langle x \rangle^{\mu})}{\langle x \rangle^{\mu-l}} \partial_x^{r-l}\big(\langle x \rangle^{\mu-l} f \big).
    \end{aligned}
\end{equation*}
Thus, we get
\begin{equation*}
    \begin{aligned}
    \|A_2\|_{L^2}\leq C \sum_{\substack{0\leq l \leq r\\ \mu-l>0 }} \|J^{r-l}\big(\langle x \rangle^{\mu-l} f \big)\|_{L^2}.
    \end{aligned}
\end{equation*}
Gathering the estimates for $A_1$ and $A_2$, we deduce \eqref{eqInterpoeq2}.
\end{proof}

We recall the following result deduced in \cite[Theorem 1]{FonsecaLinaresPonce2015}.

\begin{lemma}\label{fractderLiner}
Let $\beta \in (0,1)$ and $U(t)$, $t \in \mathbb{R}$, be the unitary group of operators defined in \eqref{Uoperdef}. If
\begin{equation*}
    f \in H^{2\beta}(\mathbb{R})\cap L^{2}(|x|^{2\beta}\, dx),
\end{equation*}
then for all $t\in \mathbb{R}$ and almost every $x\in \mathbb{R}$
\begin{equation*}
    \begin{aligned}
|x|^{\beta}U(t)f=U(t)\big(|x|^{\beta}f\big)(x)+U(t)\big\{\Phi_{t,\beta}(\widehat{f})(\xi)\big\}^{\vee}(x)
    \end{aligned}
\end{equation*}
where
\begin{equation*}
    \Phi_{t,\beta}(\widehat{f})(\xi)=\lim_{\epsilon \to 0}\frac{1}{c_{\beta}}\int_{|y|\geq \epsilon}\frac{e^{it((\xi+y)^3-\xi^3)}-1}{|y|^{1+\beta}} \widehat{f}(\xi+y)\, dy, 
\end{equation*}
 $c_{\beta}=\pi 2^{-\beta-1}\Gamma(-\frac{\beta}{2})$, and for some $C>0$ independent of $f$ and $t$,
\begin{equation*}
    \|\big\{\Phi_{t,\beta}(\widehat{f})(\xi)\big\}^{\vee}\|_{L^2}\leq C\langle t \rangle\big(\|f\|_{L^2}+\|D^{2\beta} f\|_{L^2} \big).
\end{equation*}
\end{lemma}

We are now in a position to establish the following estimate relating solutions of the linear KdV equation and fractional weights.

\begin{lemma}\label{derivexp2}
Let $m \in \mathbb{R}^{+}$. Then for any $t\in \mathbb{R}$, there exists $C>0$ such that
\begin{equation*}
\|\langle x \rangle^{m}U(t)f\|_{L^2} \leq C \langle t \rangle^{\lfloor m \rfloor+1}\big( \|J^{2m}f\|_{L^2}+\|\langle x\rangle^m f\|_{L^2} \big).
\end{equation*}
\end{lemma}

\begin{proof}
We write $m=m_1+m_2$, where $m_1\in \mathbb{Z}^{+}\cup \{0\}$, $m_2\in [0,1)$. Then by Plancherel's identity and Leibniz's rule, we deduce
\begin{equation*}
\begin{aligned}
\|\langle x\rangle^{m}U(t)f\|_{L^2}\leq & C \|U(t)f\|_{L^2}+C\||x|^{m_2}|x|^{m_1}U(t)f\|_{L^2} \\
\leq & C \|f\|_{L^2}+C\|D^{m_2}\big(\frac{d^{m_1}}{d\xi^{m_1}}(e^{it\xi^3}\widehat{f} \, )\big)\|_{L^2}\\
\leq & C \|f\|_{L^2}+C\sum_{k=0}^{m_1}\|D^{m_2}\big(\frac{d^{k}}{d\xi^k}(e^{it\xi^3})\frac{d^{m_1-k}}{d\xi^{m_1-k}}\widehat{f} \, )\big)\|_{L^2},
\end{aligned}
\end{equation*}
when $m_2=0$, we will assume that $D^{m_2}$ is the identity operator. To estimate the right-hand side of the above expression, we use a simple inductive argument to obtain the following identity
\begin{equation*}
\begin{aligned}
\frac{d^{k}}{d\xi^k}(e^{it\xi^3})=e^{it\xi^3}\sum_{l=0}^{\lfloor \frac{2k}{3} \rfloor}c_l t^{k-l}\xi^{2k-3l},
\end{aligned}
\end{equation*}
for some constants $c_l$, $l=0,1,\dots,\lfloor \frac{2 k}{3} \rfloor$ and any $k\geq 0$. Then, an application of Lemma \ref{fractderLiner} in the frequency domain yields
\begin{equation}\label{eqWLEQ0}
    \begin{aligned}
    \sum_{k=0}^{m_1}\|&D^{m_2}\big(\frac{d^{k}}{d\xi^k}(e^{it\xi^3})\frac{d^{m_1-k}}{d\xi^{m_1-k}}\widehat{f} \, )\big)\|_{L^2} \\
    \leq \,  C \, & \langle t \rangle^{m_1+1} \sum_{k=0}^{m_1}\sum_{l=0}^{\lfloor \frac{2k}{3} \rfloor} \Big(\|D^{m_2}\big(\xi^{2k-3l}\frac{d^{m_1-k}}{d\xi^{m_1-k}}\widehat{f}\,\big)\|_{L^2}
     + \|\langle \xi \rangle^{2m_2}\big(\xi^{2k-3l}\frac{d^{m_1-k}}{d\xi^{m_1-k}}\widehat{f}\, \big)\|_{L^2}\Big).
     \end{aligned}
     \end{equation}
Setting $0\leq k \leq m_1$, $0\leq l \leq \lfloor \frac{2k}{3} \rfloor$, when $m_2>0$, we use Theorem \ref{TheoSteDer}, together with \eqref{prelimneq}  and \eqref{prelimneq1} to deduce
\begin{equation}\label{eqWLEQ1}
\begin{aligned}
\|D^{m_2}&\big(\xi^{2k-3l}\frac{d^{m_1-k}}{d\xi^{m_1-k}}\widehat{f\,}\big)\|_{L^2}\\
\leq & C \|\langle \xi \rangle^{2m_2+2k-3l}\frac{d^{m_1-k}}{d\xi^{m_1-k}}\widehat{f}\|_{L^2}+C \|\mathcal{D}^{m_2}\big(\frac{\xi^{2k-3l}}{\langle \xi \rangle^{2k-3l}}\langle \xi \rangle^{2k-3l}\frac{d^{m_1-k}}{d\xi^{m_1-k}}\widehat{f}\,\big)\|_{L^2}\\
\leq & C \|\langle \xi \rangle^{2m_2+2k-3l}\frac{d^{m_1-k}}{d\xi^{m_1-k}}\widehat{f}\|_{L^2}+C \|\mathcal{D}^{m_2}\big(\frac{\xi^{2k-3l}}{\langle \xi \rangle^{2k-3l}}\big)\langle \xi \rangle^{2k-3l}\frac{d^{m_1-k}}{d\xi^{m_1-k}}\widehat{f\,}\|_{L^2}\\
&+C \|\frac{\xi^{2k-3l}}{\langle \xi \rangle^{2k-3l}}\mathcal{D}^{m_2}\big(\langle \xi \rangle^{2k-3l}\frac{d^{m_1-k}}{d\xi^{m_1-k}}\widehat{f\,}\big)\|_{L^2}\\
\leq & C \|\langle \xi \rangle^{2m_2+2k-3l}\frac{d^{m_1-k}}{d\xi^{m_1-k}}\widehat{f}\|_{L^2}+C\|D^{m_2}\big(\langle \xi \rangle^{2k-3l}\frac{d^{m_1-k}}{d\xi^{m_1-k}}\widehat{f\,}\big)\|_{L^2},
    \end{aligned}
\end{equation}
where we have used the fact that the function $\xi^{2k-3l}\langle \xi \rangle^{-2k+3l}$ and its derivative belong to $L^{\infty}(\mathbb{R})$. Notice that the same bound in \eqref{eqWLEQ1} is valid when $m_2=0$. Thus, to complete the estimate of \eqref{eqWLEQ0}, it is enough to bound the right-hand side of \eqref{eqWLEQ1}. When $0\leq k <m_1$, by Proposition \ref{PropDbyJ} and the interpolation inequality in Proposition \ref{propweight2}, we have
\begin{equation}\label{eqWLEQ2}
    \begin{aligned}
    \|&\langle \xi \rangle^{2m_2+2k-3l}\frac{d^{m_1-k}}{d\xi^{m_1-k}}\widehat{f}\|_{L^2} \\
    &\leq C \sum_{\substack{0\leq r \leq k \\ r<2m_2+2k-3l}}\|J^{m_1-k-r}\big(\langle \xi \rangle^{2m_2+2k-3l-r}\widehat{f\,}\big)\|_{L^2}+\|J^{m_1-k-2m_2-2k+3l}\widehat{f}\|_{L^2}\\
    &\leq C \sum_{\substack{0\leq r \leq k \\ r<2m_2+2k-3l}} \|\langle \xi \rangle^{2m_1+2m_2}\widehat{f}\|_{L^2}^{\frac{2m_2+2k-3l-r}{2m_1+2m_2}}\|J^{\frac{(2m_1+2m_2)(m_1-k-r)}{2m_1-2k+3l+r}}\widehat{f}\|_{L^2}^{\frac{2m_1-2k+3l+r}{2m_1+2m_2}} \\
    &\hspace{2cm}+\|J^{m_1-k-2m_2-2k+3l}\widehat{f}\|_{L^2} \\
     &\leq  C\|\langle \xi \rangle^{2m_1+2m_2} \widehat{f}\|_{L^2}+C\|J^{m_1+m_2} \widehat{f}\|_{L^2}, 
    \end{aligned}
\end{equation}
where in the last line above we used Young's inequality and the fact that $m_1-k-2m_2-2k+3l, \, \frac{2m_1-2k+3l+r}{2m_1+2m_2}\leq m_1+m_2$. Notice that the same upper bound in \eqref{eqWLEQ2} is still valid when $k=m_1$. This remark completes the estimate for the first term on the right-hand side of \eqref{eqWLEQ1}. The estimate for the remaining factor on the right-hand side of \eqref{eqWLEQ1} follows by similar arguments as above. Using Propositions \ref{PropDbyJ} and  \ref{propweight2}, we conclude
\begin{equation}\label{eqWLEQ3}
    \begin{aligned}
    \|D^{m_2}\big(\langle \xi \rangle^{2k-3l}\frac{d^{m_1-k}}{d\xi^{m_1-k}}\widehat{f\,}\big)\|_{L^2} &\leq  C\|\langle \xi \rangle^{2m_1+2m_2} \widehat{f}\|_{L^2}+C\|J^{m_1+m_2} \widehat{f}\|_{L^2}, 
    \end{aligned}
\end{equation}
for all $0\leq k\leq m_1$, $0\leq l \leq \lfloor \frac{2k}{3} \rfloor$. Gathering \eqref{eqWLEQ0}, \eqref{eqWLEQ2}, \eqref{eqWLEQ3}, and using Plancherel's identity, we obtain the desired result.
\end{proof}

We also require the following sharp (homogeneous) version of Kato smoothing effect deduced in \cite{KenigPonceVega1993}. 

\begin{lemma}[Kenig-Ponce-Vega \cite{KenigPonceVega1993}]\label{KatoS}
Let $U(t)$, $t\in\mathbb{R}$, denote the unitary group describing the solution of the linear KdV equation defined in \eqref{Uoperdef}. Then, for all $ f\in{L^2(\mathbb{R})}$ complex or real valued, 
$$\|U(t)f\|_{L_t^\infty L_x^2}+\|\partial_xU(t)f\|_{L_x^\infty L_t^2}=(1+\frac{1}{\sqrt{3}})\|f\|_{L^2}.$$
\end{lemma}


\subsection{Proof of Theorem \ref{mainTHM}}

We only prove Theorem \ref{mainTHM} for the Cauchy problem \eqref{GK} as the same proof developed below is valid for the Cauchy problem \eqref{gKdV}. Besides, by simple modifications to our argument, we will assume that $u$ is a real-valued function. 

We will use the same functional space $\mathcal{X}_T$ introduced in \cite{LinaresMiyaGus} (except we allow the power $m$ of the weight $\langle x\rangle^m$ to be fractional). Setting $m>\max\{\frac{1}{2\alpha},\frac{1}{2}\}$, $s\in\mathbb{Z^+}$ with $s\geq 2m+4$, we consider the space $\mathcal{X}_T$ defined as
\begin{equation*}
\begin{aligned}
\mathcal{X}_T = \big\{u &\in C([0,T];H^s(\mathbb{R})): \\
\|u\|_{\mathcal{X}_T}:= & \|u\|_{L_T^\infty H_x^s}  +\|\langle x \rangle^m u \|_{L^\infty_T L^\infty_x}+\sum_{l=1}^4{\|\langle x \rangle^m \partial^l_x u\|_{L^\infty_T L^2_x}} + \|\partial^{s+1}_x u \|_{L^\infty_x L^2_T} \leq 2C_1\delta,\\
& \qquad \qquad \sup_{0\leq t \leq T}{\| \langle x\rangle^m (u(t)-u(0))\|_{L^{\infty}}}\leq{\frac{\lambda}{2}}\big\}.
\end{aligned}
\end{equation*}
We equip this space with a distance function
\begin{equation*}
   d_{\mathcal{X}_T}(u,v)=\|u-v\|_{\mathcal{X}_T}.
\end{equation*}
Then, by the definition of the space $\mathcal{X}_T$ and \eqref{condi1}, we have 
\begin{equation}\label{NonNullCond}
\frac{\lambda}{2}\leq \langle x \rangle^m |u(x,t)| \leq \langle x \rangle^m |u_0(x)| +\frac{\lambda}{2},
\end{equation}
for any $(x,t) \in \mathbb{R} \times [0,T]$ as long as $u\in X_T$. We use the integral representation of the GKdV to define the map
\begin{equation*}
\Phi (u(t)) = U(t)u_0 \mp \int_0^t U(t-s)(|u|^\alpha \partial_x u )(s)ds,
\end{equation*}
where $U(t)$, $t\in \mathbb{R}$, is given by \eqref{Uoperdef}, $u_0$ as in \eqref{eqmainTH1}-\eqref{condi1}, and $u\in \mathcal{X}_T$. In what follows, we will find $C_1>0$, $T>0$ such that $\Phi$ defines a contraction on the complete metric space $\mathcal{X}_T$. \smallskip

We begin by proving that for small $T>0$, $\Phi(u) \in \mathcal{X}_T$, for all $u\in \mathcal{X}_T$.  \smallskip

$\bullet$ {\bf Estimate for $\|\Phi(u)\|_{L_T^\infty H^s_x}$ and $\|\partial_x^{s+1} \Phi(u)\|_{L^\infty_x L^2_T}$}. Using Lemma \ref{KatoS}, we get
\begin{equation} \label{eqWLEQ4}
\begin{aligned}
\|\partial_x^s \Phi(u)\|_{L_T^\infty L_x^2}+\|\partial_x^{s+1}\Phi(u)\|_{L^\infty_x L^2_T}&\leq C_0\|u_0\|_{H^s}+C\|\partial_x^s(|u|^{\alpha}\partial_x u)\|_{L^1_TL^2_x}\\
&\leq C_0\|u_0\|_{H^s}+C \sum_{j=0}^{s}{\|\partial_x^j(|u|^\alpha)\partial^{s+1-j}_x u \|_{L_T^1 L_x^2}}.
\end{aligned}
\end{equation}
To control the sum on the right-hand side of \eqref{eqWLEQ4}, by an interpolation argument, it is enough to estimate the cases $j=0$ and $j=s$. When $j=0$, we use H\"older's inequality and the definition of the space $\mathcal{X}_T$ to deduce
\begin{equation}\label{eqWLEQ4.1}
\begin{aligned}
\||u|^{\alpha}\partial_x^{s+1}u\|_{L^1_TL^2_x} &\leq T^{\frac{1}{2}}\| |u|^{\alpha} \partial_x^{s+1} u \|_{L^2_T L^2_x}\\
&\leq CT^{\frac{1}{2}} \|\partial_x^{s+1}u\|_{L^\infty_x L^2_T} \|\langle x \rangle^{m\alpha} |u|^\alpha\|_{L_T^\infty L_x^\infty} \|\langle x \rangle^{-m\alpha} \|_{L^2}\\
&\leq CT^{\frac{1}{2}} \|\partial_x^{s+1} u\|_{L_x^\infty L_T^2} \|\langle x \rangle^m u\|^{\alpha}_{L_T^{\infty} L_x^{\infty}}\\
&\leq CT^{\frac{1}{2}} \delta^{\alpha +1},
\end{aligned}
\end{equation}
where we have used that $m>\frac{1}{2\alpha}$. Next, we deal with the case $j=s$. We have
\begin{equation}\label{eqWLEQ5}
\begin{aligned}
\|\partial^s_x(|u|^\alpha)\partial_x u\|_{L^1_T L^2_x}\leq & C T\left( \||u|^{\alpha-s}|\partial_x u|^s\partial_x u\|_{L^\infty_T L^2_x }+\cdots+\||u|^{\alpha-1}|\partial_x^s u|\partial_x u\|_{L^\infty_T L^2_x } \right)\\
=:& CT\big(\mathcal{I}_{s,s}+\dots+\mathcal{I}_{s,1} \big).
\end{aligned}
\end{equation}
Again, an interpolation argument reduces to the proof of the control of $\mathcal{I}_{s,1}$ and $\mathcal{I}_{s,s}$ only. First, using the definition of the space $\mathcal{X}_T$ and property \eqref{NonNullCond}, it is not difficult to see 
\begin{equation}\label{betaesti}
    |u|^{\beta}\leq C(\lambda^{-|\beta|}+\delta^{|\beta|}) \langle x \rangle^{-m\beta},
\end{equation}
for all $\beta \in \mathbb{R}$ and $C>0$ independent of $u\in \mathcal{X}_T$ and depending on $C_1$. The previous inequality and the Sobolev embedding $H^1(\mathbb{R})\hookrightarrow L^{\infty}(\mathbb{R})$ yield 
\begin{equation}\label{eqWLEQ6}
\begin{aligned}
\mathcal{I}_{s,1} \leq & C\||u|^{\alpha-1}|\partial_xu|\|_{L^\infty_T L^\infty_x} \|\partial_x^su\|_{L_T^\infty L^2_x}\\
\leq & C(\lambda^{-|\alpha-1|}+\delta^{|\alpha-1|})\|\langle x \rangle^{-m(\alpha-1)}\partial_xu\|_{L^\infty_T L^\infty_x} \|\partial_x^su\|_{L_T^\infty L^2_x} \\
\leq & C(\lambda^{-|\alpha-1|}+\delta^{|\alpha-1|})\big(\|\langle x \rangle^{m}\partial_xu\|_{L^\infty_T L^2_x}+\|\langle x \rangle^{m}\partial_x^2u\|_{L^\infty_T L^2_x}\big) \|\partial_x^su\|_{L_T^\infty L^2_x} \\
\leq & C(\lambda^{-|\alpha-1|}+\delta^{|\alpha-1|})\delta^2,
\end{aligned}
\end{equation}
where we have also used that $\langle x \rangle^{-m(\alpha-1)}\leq \langle x \rangle^{m}$. Using again \eqref{betaesti} and the Sobolev embedding gives
\begin{equation}\label{eqWLEQ7}
\begin{aligned}
\mathcal{I}_{s,s}\leq &C(\lambda^{-|\alpha-s|}+\delta^{|\alpha-s|})\|\langle x \rangle^{-m(\alpha-s)}|\partial_x u|^s \partial_x u\|_{L^{\infty}_T L^2_x}\\
\leq &C(\lambda^{-|\alpha-s|}+\delta^{|\alpha-s|})\|\langle x \rangle^{m}\partial_x u\|_{L^{\infty}_T L^{\infty}_x}^s \|\langle x \rangle^{-m\alpha}\partial_x u\|_{L^{\infty}_T L^2_x}\\
\leq &C(\lambda^{-|\alpha-s|}+\delta^{|\alpha-s|})(\|\langle x \rangle^{m}\partial_x u\|_{L^{\infty}_T L^{2}_x}+\|\langle x \rangle^{m}\partial_x^2 u\|_{L^{\infty}_T L^{2}_x})^s \|\langle x \rangle^m\partial_x u\|_{L^{\infty}_T L^2_x}\\
\leq &C(\lambda^{-|\alpha-s|}+\delta^{|\alpha-s|})\delta^{s+1}.
\end{aligned}
\end{equation}
Plugging \eqref{eqWLEQ4.1}, \eqref{eqWLEQ6} and \eqref{eqWLEQ7} into \eqref{eqWLEQ5}, and going back to \eqref{eqWLEQ4}, we deduce
\begin{equation}\label{eqWLEQ7.1}
    \begin{aligned}
    \|\partial_x^s \Phi(u)\|_{L^{\infty}_T L^2_x}+\|\partial_x^{s+1} \Phi(u)\|_{L^{\infty}_x L^2_T} &\leq C_0 \|u_0\|_{H^s}+CT^{\frac{1}{2}}\langle T \rangle^{\frac{1}{2}}\langle \lambda^{-1}+\delta \rangle^{\alpha+2s+1}.
    \end{aligned}
\end{equation}
Similarly, we deduce
\begin{equation}\label{eqWLEQ7.1.1}
    \begin{aligned}
    \| \Phi(u)\|_{L^{\infty}_T L^2_x}\leq \|u_0\|_{H^s}+CT^{\frac{1}{2}}\langle T \rangle^{\frac{1}{2}}\langle \lambda^{-1}+\delta \rangle^{\alpha+2s+1}.
    \end{aligned}
\end{equation}
Thus, \eqref{eqWLEQ7.1} and \eqref{eqWLEQ7.1.1} complete the estimate of the $H^{s}$-norm  of $\Phi(u)$ and $\|\partial_x^{s+1}\Phi(u)\|_{L^{\infty}_x L^2_T}$.

$\bullet$ {\bf Estimate for $\|\langle x \rangle^m \partial_x^l \Phi(u)\|_{L^\infty_T L^2_x}, 1 \leq l \leq 4$}. An application of Lemma \ref{derivexp2} yields
\begin{equation*}
\begin{aligned}
 \|\langle x \rangle^m \partial_x^l \Phi(u)\|_{L^\infty_T L^2_x} \leq & \|\langle x \rangle^m U(t) \partial^l_x u_0 \|_{L^\infty_T L^2_x} + \|\int_0^t \langle x \rangle^m U(t-s) \partial^l_x (|u|^\alpha \partial_x u )(s)ds\|_{L^\infty_T L^2_x}\\
\leq & C\langle T\rangle^{\lfloor m \rfloor +1}\|\langle x \rangle^m \partial^l_x u_0\|_{L^2_x}+C\langle T\rangle^{\lfloor m \rfloor +1} \|D^{2m}\partial^{l}_x u_0\|_{L^2_x}\\
&+C\langle T\rangle^{\lfloor m \rfloor +1} \|\langle x \rangle^{m} \partial_x^l (|u|^\alpha \partial_x u) \|_{L^1_T L^2_x} \\
&+ C\langle T\rangle^{\lfloor m \rfloor +1} \|D^{2m}\partial^{l}_x (|u|^{\alpha} \partial_x u)\|_{L^1_T L^2_x}.
\end{aligned}
\end{equation*}
Since $s\geq 2m+4$, we have $\|D^{2m}\partial^l_x u_0\|_{L^2}\leq \|u_0\|_{H^s}$. Thus, we proceed with the  estimate for the term $\|D^{2m}\partial_x^{l} (|u|^{\alpha} \partial_x u)\|_{L^1_T L^2_x}$. Gathering \eqref{eqWLEQ4.1}, \eqref{eqWLEQ6} and \eqref{eqWLEQ7}, we get
\begin{equation*} 
\begin{aligned}
\|D^{2m}\partial^{l}_x (|u|^{\alpha} \partial_x u)\|_{L^1_T L^2_x} &\leq C \||u|^\alpha \partial_x u\|_{L^1_T L^2_x} + C\|\partial^s_x (|u|^\alpha \partial_x u)\|_{L^1_T L^2_x} \\
&\leq C T\|u\|^{\alpha + 1}_{L^\infty_T H^1_x} +CT^{\frac{1}{2}}\langle T \rangle^{\frac{1}{2}}\langle \lambda^{-1}+\delta\rangle^{\alpha+2s+1}\\
&\leq CT^{\frac{1}{2}}\langle T \rangle^{\frac{1}{2}}\langle \lambda^{-1}+\delta\rangle^{\alpha+2s+1}.
\end{aligned}
\end{equation*}
Next, we deal with $\|\langle x \rangle^m \partial_x^l(|u|^\alpha \partial_x u)\|_{L^1_T L^2_x}$. Using interpolation again, we obtain
\begin{equation*} 
\begin{aligned}
\|\langle x \rangle^m \partial_x^l(|u|^\alpha \partial_x u)\|_{L^1_T L^2_x}\leq & CT\|\langle x \rangle^m |u|^\alpha \partial^{l+1}_x u\|_{L^\infty_T L^2_x} + CT\|\langle x \rangle^m \partial^l_x(|u|^\alpha)\partial_x u\|_{L^\infty_T L^2_x}\\
\leq & CT\|\langle x \rangle^m |u|^\alpha \partial_x^{l+1} u\|_{L^\infty_T L^2_x} + CT\| \langle x \rangle^m |u|^{\alpha -l} |\partial_x u|^{l+1}\|_{L^\infty_T L^2_x}\\
&+CT \|\langle x \rangle^m |u|^{\alpha - 1} |\partial_x^l u|\partial_x u\|_{L^\infty_T L^2_x}\\
=:& \, CT\big(\mathcal{II}_1+\mathcal{II}_2+\mathcal{II}_3\big).
\end{aligned}
\end{equation*}
The estimate for $\mathcal{II}_1$ with $1\leq l \leq 3$ is obtained as follows
\begin{equation*}
\begin{aligned}
\mathcal{II}_1 &\leq C \|\langle x \rangle^m u \|_{L^\infty_T L^\infty_x}^{\alpha} \|\langle x\rangle^{m-m\alpha} \partial_x^{l+1} u\|_{L^\infty_T L^2_x}\\
&\leq C \|\langle x \rangle^m u \|_{L^\infty_T L^\infty_x}^{\alpha} \|\langle x\rangle^{m} \partial_x^{l+1} u\|_{L^\infty_T L^2_x} \\
&\leq C\delta^{\alpha+1}.
\end{aligned}
\end{equation*}
If $l=4$ and $\alpha \geq 1$, since $s>5$, we get
\begin{equation*}
\begin{aligned}
\mathcal{II}_1 &= \|\langle x \rangle^m (\langle x \rangle^m |u|)^\alpha \langle x \rangle^{-m\alpha} \partial^{5} u \|_{L^\infty_T L^2_x}\\
&\leq C\|\langle x \rangle^m u\|^\alpha_{L^\infty_T L^\infty_x}\|\langle x \rangle^{m(1-\alpha)}\partial^5_x u\|_{L^\infty_T L^2_x}\\
&\leq  C\|\langle x \rangle^m u\|^\alpha_{L^\infty_T L^\infty_x} \|\partial^5_x u\|_{L^\infty_T L^2_x}\\
&\leq C\delta^{\alpha + 1}.
\end{aligned}
\end{equation*}
If $l=4$ and $0<\alpha<1$, we use Lemma \ref{Interplemma}, the fact that $s\geq \frac{1}{\alpha}+4$ and the Sobolev embedding to get
\begin{equation*}
\begin{aligned}
\mathcal{II}_1  &\leq C\|\langle x \rangle^m u\|^\alpha_{L^\infty_T L^\infty_x}\|\langle x \rangle^{m(1-\alpha)}\partial^5_x u\|_{L^\infty_T L^2_x}\\
&\leq C\|\langle x \rangle^{m}u\|_{L^{\infty}_T L^{\infty}_x}^{\alpha}\big(\|\langle x \rangle^m \partial_x^{4}u\|_{L^{\infty}_T L^2_x}+\|J^{\frac{1}{\alpha}+4}u\|_{L^{\infty}_T L^2_x}\big)\\
&\leq C\delta^{\alpha+1}.
\end{aligned}
\end{equation*}
This completes the estimate for $\mathcal{II}_1$. Next, by \eqref{betaesti} and the Sobolev embedding, we get
\begin{equation*} 
\begin{aligned}
\mathcal{II}_2 \leq & C(\lambda^{-|\alpha-l|}+\delta^{|\alpha-l|})\|\langle x \rangle^{m-m(\alpha-l)}|\partial_x u|^{l+1}\|_{L^{\infty}_T L^{2}_x}\\
\leq & C(\lambda^{-|\alpha-l|}+\delta^{|\alpha-l|})\|\langle x\rangle^m \partial_x u\|_{L^{\infty}_T L^{\infty}_x}^{l+1}\|\langle x\rangle^{-m\alpha}\|_{L^{2}_x}\\
\leq & C \langle \lambda^{-1}+\delta \rangle^{\alpha+2s+1}.
\end{aligned}
\end{equation*}
Finally, by previous arguments one can deduce
\begin{equation*}
\begin{aligned}
\mathcal{II}_3 \leq & C(\lambda^{-|\alpha-1|}+\delta^{|\alpha-1|})\|\langle x \rangle^{m}\partial_x u\|_{L^{\infty}_T L^{\infty}_x}\|\langle x \rangle^{m(1-\alpha)}\partial_x^l u\|_{L^{\infty}_T L^{2}_x}\\
\leq &C \langle \lambda^{-1}+\delta \rangle^{\alpha+2s+1}.
\end{aligned}
\end{equation*}
Consequently, we gather the above estimates for $\mathcal{II}_j$, $j=1,2,3$, to get
\begin{equation*}
\begin{aligned}
\|\langle x \rangle^m \partial_x^l(|u|^\alpha \partial_x u)\|_{L^{\infty}_T L^2_x}\leq  CT^{\frac{1}{2}}\langle T\rangle^{\frac{1}{2}} \langle \lambda^{-1}+\delta \rangle^{\alpha+2s+1},
\end{aligned}
\end{equation*}
for any $l=1,\dots,4$. Thus, we conclude
\begin{equation}\label{eqWLEQ7.2} 
    \begin{aligned}
    \sup_{1\leq l \leq 4}\|\langle x \rangle^m &\partial_x^l \Phi(u)\|_{L^\infty_T L^2_x} \leq C_0\langle T \rangle^{\lfloor m \rfloor+1}\delta +C T^{\frac{1}{2}}\langle T\rangle^{\lfloor m \rfloor+\frac{3}{2}} \langle \lambda^{-1}+\delta \rangle^{\alpha+2s+1}.
    \end{aligned}
\end{equation}

$\bullet$ {\bf Estimate for $\|\langle x \rangle^m \Phi (u)\|_{L^\infty_T L^\infty_x}$}.  The fact that
\begin{equation*}
  \frac{d}{dt}U(t)u_0 = -\partial_x^3 U(t)u_0  
\end{equation*}
allows us to write 
\begin{equation*}
    \begin{aligned}
    \langle x \rangle^m U(t) u_0=\langle x \rangle^m u_0-\int_0^t \, \langle x \rangle^m U(s)\partial_x^3 u_0(s)\, ds.
    \end{aligned}
\end{equation*}
Then, we apply the Sobolev embedding, the fact that $s\geq 2m+4$ and Lemma \ref{derivexp2} to get
\begin{equation}\label{eqWLEQ8}
\begin{aligned}
\|\langle x \rangle^m  (U(t)u_0 - u_0)\|_{L^{\infty}_TL^\infty_x} \leq & \int_0^T \| \langle x \rangle^m U(s) \partial^3_x u_0(s) \|_{L^\infty_x}\, ds\\
\leq &C T\big(\|\langle x \rangle^m U(t) \partial_x^3 u_0\|_{L^{\infty}_T L^2_x} +\|\langle x \rangle^m U(t) \partial^4_x u_0\|_{L^{\infty}_T L^2_x}\big)\\
\leq &C T\langle T \rangle^{\lfloor m \rfloor +1} \big(\|\langle x \rangle^m \partial^3_x u_0\|_{L^2_x} +\|\langle x \rangle^m \partial^4_x u_0\|_{L^2_x} \\
&\hspace{2cm}
+\|J^{2m+4} u_0\|_{L^2_x}\big)\\
\leq &CT^{\frac{1}{2}}\langle T\rangle^{\lfloor m\rfloor+\frac{3}{2}}\delta.
\end{aligned} 
\end{equation}
By a similar argument used above, we also deduce
\begin{equation*}
\begin{aligned}
\|\langle x \rangle^m & \int_0^t{U(t-s)(|u|^\alpha\partial_xu)(s)ds}\|_{L^{\infty}_T L^\infty_x}\\
\leq & C\langle T \rangle^{\lfloor m \rfloor+1} \big( \|\langle x \rangle^m ( |u|^\alpha \partial_x u)\|_{L_T^{1} L_x^2}+\|\langle x \rangle^m \partial_x( |u|^\alpha \partial_x u)\|_{L_T^{1} L_x^2} \\
&+ \|J^{2m+1}(|u|^\alpha \partial_x u)\|_{L_T^{1} L^2_x}\big).
\end{aligned}
\end{equation*}
Following the argument in \eqref{eqWLEQ6}, we obtain
\begin{equation*}
   \|\langle x \rangle^m |u|^\alpha \partial_x u\|_{L_T^1 L^2_x}+\|\langle x \rangle^m \partial_x (|u|^\alpha \partial_x u)\|_{L_T^{1} L_x^2} \leq CT^{\frac{1}{2}}\langle T \rangle^{\frac{1}{2}}\langle \lambda^{-1}+\delta \rangle^{\alpha+2s+1}.
\end{equation*}
Now, setting $0\leq \beta \leq s$ an integer, as done previously, we can use interpolation and the estimates \eqref{eqWLEQ4.1} and \eqref{eqWLEQ5} to get
\begin{equation*}
    \|\partial^\beta_x (|u|^\alpha \partial_x u)\|_{L^1_T L^2_x} \leq CT^{\frac{1}{2}}\langle T \rangle^{\frac{1}{2}}\langle \lambda^{-1}+\delta \rangle^{\alpha+2s+1}.
\end{equation*}
Summarizing, we have deduced
\begin{equation}\label{eqWLEQ9}
\begin{aligned}
\|\langle x \rangle^m & \int_0^t{U(t-s)(|u|^\alpha\partial_xu)(s)ds}\|_{L^{\infty}_T L^\infty_x}\leq CT^{\frac{1}{2}}\langle T \rangle^{\lfloor m\rfloor +\frac{3}{2}}\langle \lambda^{-1}+\delta \rangle^{\alpha+2s+1}.
\end{aligned}
\end{equation}
Finally, \eqref{eqWLEQ8} and \eqref{eqWLEQ9} produce the following estimate
\begin{equation}\label{eqWLEQ9.1} 
\begin{aligned}
\|\langle x \rangle^m \Phi(u)\|_{L^\infty_T L^\infty_x} \leq &\|\langle x \rangle^m u_0\|_{L^\infty_x}+\|\langle x \rangle^m \left(U(t)u_0- u_0\right)\|_{L^\infty_T L^\infty_x}\\
&+\|\langle x \rangle^m \int_0^tU(t-s)(|u|^\alpha \partial_x u)(s)ds\|_{L_T^\infty L_x^\infty}\\
\leq & \delta+CT^{\frac{1}{2}}\langle T \rangle^{\lfloor m \rfloor+\frac{3}{2}}\langle \lambda^{-1}+\delta\rangle^{\alpha+2s+1}.
\end{aligned}
\end{equation}
Therefore, \eqref{eqWLEQ7.1}, \eqref{eqWLEQ7.1.1}, \eqref{eqWLEQ7.2} and \eqref{eqWLEQ9.1} show that there exist some constants $C_1\geq 6\max\{C_0,1\}$ and $C>0$ such that
\begin{equation}\label{eqWLEQ9.2}
\begin{aligned}
\|\Phi(u)\|_{X_T} \leq & C_1\langle T \rangle^{\lfloor m\rfloor+1}\delta+CT^{\frac{1}{2}}\langle T \rangle^{\lfloor m \rfloor+\frac{3}{2}}\langle \lambda^{-1}+\delta\rangle^{\alpha+2s+1},
\end{aligned}
\end{equation}
and using \eqref{eqWLEQ8} and \eqref{eqWLEQ9}, we get
\begin{equation}\label{eqWLEQ9.3}
\begin{aligned}
 \|&\langle x \rangle^m (\Phi(u(t))-u_0)\|_{L^{\infty}_T L^\infty_x}\\
 &\leq  \|\langle x \rangle^m (U(t)u_0 - u_0)\|_{L^{\infty}_TL^\infty_x}+ \|\langle x \rangle^m \int_0^t{U(t-s)(|u|^\alpha \partial_x u)(s)ds}\|_{L^{\infty}_TL^\infty_x}\\
& \leq CT^{\frac{1}{2}}\langle T \rangle^{\lfloor m\rfloor +\frac{3}{2}}\langle \lambda^{-1}+\delta \rangle^{\alpha+2s+1}.
\end{aligned}
\end{equation}
Then, let $0<T=T(\alpha, \delta, s)\leq 1$ small enough such that
\begin{equation}\label{eqWLEQ9.4}
    \begin{aligned}
    &C_1\langle T \rangle^{\lfloor m\rfloor+1}\delta+CT^{\frac{1}{2}}\langle T \rangle^{\lfloor m \rfloor+\frac{3}{2}}\langle \lambda^{-1}+\delta\rangle^{\alpha+2s+1} \leq 2C_1 \delta, \\
    &CT^{\frac{1}{2}}\langle T \rangle^{\lfloor m\rfloor +\frac{3}{2}}\langle \lambda^{-1}+\delta \rangle^{\alpha+2s+1} \leq \frac{\lambda}{2}.
    \end{aligned}
\end{equation}
Consequently, \eqref{eqWLEQ9.2}, \eqref{eqWLEQ9.3} and \eqref{eqWLEQ9.4} imply that $\Phi(u)\in \mathcal{X}_T$.
\medskip

Next, we estimate the difference $\Phi(u)-\Phi(v)$ with $u,v \in \mathcal{X}_T$. We divide our arguments according to the terms in the definition of the norm in the space $\mathcal{X}_T$. \smallskip

$\bullet$ {\bf Estimate for $\|\Phi(u)-\Phi(v)\|_{L_T^\infty H^s_x}$ and $\|\partial_x^{s+1}(\Phi(u)-\Phi(v))\|_{L^\infty_x L^2_T}$.} 
By using Lemma \ref{KatoS} and Leibniz rule, we have 
\begin{equation}\label{eqWLEQ9.5}
\begin{aligned}
\|\partial^s_x & (\Phi(u)-\Phi(v))\|_{L^\infty_T L^2_x} +\|\partial^{s+1}_x (\Phi(u)-\Phi(v))\|_{L^\infty_x L^2_T}\\
&\leq C \|\partial^s_x (|u|^\alpha\partial_x u) - \partial_x^s (|v|^\alpha \partial_x v)\|_{L^1_T L^2_x}\\
&\leq C \sum_{j=0}^{s} \|\partial_x^j (|u|^\alpha) \partial^{s+1-j}_x u - \partial_x^j (|v|^\alpha) \partial_x^{s+1-j} v\|_{L^1_T L^2_x}=: \sum_{j=0}^s\mathcal{III}_j.
\end{aligned}
\end{equation}
By using an interpolation argument, it is suffices to treat the cases $\mathcal{III}_0$ and $\mathcal{III}_s$. The condition \eqref{NonNullCond} and the mean value theorem yield 
\begin{equation}\label{eqWLEQ10}
\begin{aligned}
|&|u|^\beta-|v|^\beta|\\ 
&\leq C\left\{ \begin{aligned}
&\lambda^{\beta-1}\langle x \rangle^{-m(\beta-1)}|u-v|, \hspace{5.1cm} \text{ if }  \beta<1, \\
&\big(\|\langle x \rangle^m u\|_{L^{\infty}_TL^{\infty}_x}^{\beta-1}+\|\langle x \rangle^m v\|_{L^{\infty}_TL^{\infty}_x}^{\beta-1}\big)\langle x \rangle^{-m(\beta-1)}|u-v|, \qquad \text{ if } \beta\geq 1,
\end{aligned}\right.\\
& \leq C(\lambda^{-|\beta-1|}+\delta^{|\beta-1|})\langle x \rangle^{-m(\beta-1)}|u-v|,
\end{aligned}
\end{equation}
for all $\beta \in \mathbb{R}$. We first estimate $\mathcal{III}_0$. Indeed, by \eqref{eqWLEQ10}, we have
\begin{equation} \label{eqWLEQ11}
\begin{aligned}
\mathcal{III}_0 \leq & C\|(|u|^\alpha-|v|^{\alpha}) \partial_x^{s+1} u\|_{L^1_T L^2_x} + \||v|^{\alpha}(\partial_x^{s+1}u-\partial_x^{s+1}v))\|_{L^1_T L^2_x}\\
\leq &  C T^{\frac{1}{2}}(\lambda^{-|\alpha-1|}+\delta^{|\alpha-1|})\|\langle x \rangle^{-m(\alpha - 1)}|u-v|\partial_x^{s+1} u\|_{L^2_T L^2_x} \\
&+  C T^{\frac{1}{2}} \||v|^{\alpha}(\partial_x^{s+1}u-\partial_x^{s+1}v)\|_{L_T^2 L^2_x}\\
\leq & C T^{\frac{1}{2}}(\lambda^{-|\alpha-1|}+\delta^{|\alpha-1|})\|\langle x \rangle^{-\alpha m}\|_{L^{2}_x}\|\langle x\rangle^{m}(u-v)\|_{L^{\infty}_T L^{\infty}_x}\|\partial_x^{s+1}u\|_{L^{\infty}_x L^2_T} \\
&+C T^{\frac{1}{2}}\|\langle x \rangle^{-\alpha m}\|_{L^{2}_x}\|\langle x\rangle^{m}v\|_{L^{\infty}_T L^{\infty}_x}^{\alpha}\|\partial_x^{s+1}(u-v)\|_{L^{\infty}_x L^2_T}\\
\leq & C T^{\frac{1}{2}}\langle \lambda^{-1}+\delta \rangle^{\alpha+2s+1} d_{\mathcal{X}_T}(u,v).
\end{aligned}
\end{equation}
Next, we decompose $\mathcal{III}_s$ as follows
\begin{equation} \label{eqWLEQ11.1}
\begin{aligned}
\mathcal{III}_s \leq & CT \||u|^{\alpha -s}(\partial_x u)^{s+1} - |v|^{\alpha -s}(\partial_x v)^{s+1}\|_{L^\infty_T L^2_x}\\
&+ \cdots\\
&+CT\||u|^{\alpha -1} \partial_x^s u \partial_x u - |v|^{\alpha -1} \partial_x^s v \partial_x v \|_{L^\infty_T L^2_x},\\
=:& CT(\mathcal{IV}_s +\dots+\mathcal{IV}_1). 
\end{aligned}
\end{equation}
Again, the terms $\mathcal{IV}_j$, $2\leq j \leq s-1$, can be estimated by the interpolation between the estimates for $\mathcal{IV}_s$ and $\mathcal{IV}_1$. Using the Sobolev embedding $ H^1(\mathbb{R}) \hookrightarrow L^\infty(\mathbb{R})$, \eqref{betaesti} and \eqref{eqWLEQ10}, we deduce
\begin{equation*} 
\begin{aligned}
\mathcal{IV}_s \leq &C\||u|^{\alpha-s} ((\partial_x u)^{s+1} - (\partial_x v)^{s+1})\|_{L_T^\infty L^2_x}\\
&+ \|(|u|^{\alpha -s}-|v|^{\alpha -s})(\partial_x v)^{s+1}\|_{L^\infty_T L^2_x} \\
\leq &C(\lambda^{-|\alpha-s|}+\delta^{|\alpha-s|})\|\langle  x \rangle^{-m(\alpha-s)} (|\partial_x u|^s + |\partial_x v|^s)(\partial_x u - \partial_x v)\|_{L^\infty_T L^2_x}\\
&+C(\lambda^{-|\alpha-s-1|}+\delta^{|\alpha-s-1|}) \|\langle x \rangle^{-m(\alpha -s -1)} |u-v| (\partial_x v)^{s+1}\|_{L^\infty_T L^2_x}\\
\leq &C(\lambda^{-|\alpha-s|}+\delta^{|\alpha-s|})\|\langle  x \rangle^{m}\partial_x u\|_{L^{\infty}_T L^{\infty}_x}^{s}\|\langle  x \rangle^{m}(\partial_x u - \partial_x v)\|_{L^\infty_T L^2_x}\\
&+C(\lambda^{-|\alpha-s-1|}+\delta^{|\alpha-s-1|})\\
&\qquad\times\|\langle x\rangle^{-m(\alpha+1)}\|_{L^2}\|\langle x \rangle^m \partial_x v\|^{s+1}_{L^{\infty}_{T}L^{\infty}_x}\|\langle x \rangle^{m}(u-v)\|_{L^{\infty}_T L^{\infty}_x}\\
\leq &C\langle \lambda^{-1}+\delta \rangle^{\alpha+2s+1} d_{\mathcal{X}_T}(u,v).
\end{aligned}
\end{equation*}
Next we deal with $\mathcal{IV}_1$. By using \eqref{betaesti} and \eqref{eqWLEQ10}, we obtain 
\begin{equation*}
\begin{aligned}
\mathcal{IV}_1 \leq & \||u|^{\alpha - 1} \partial_x^s u \partial_x (u - v)\|_{L^\infty_T L^2_x} + \||u|^{\alpha - 1} \partial^s_x(u-v)\partial_x v\|_{L^\infty_T L^2_x}\\
&+\|(|u|^{\alpha - 1} - |v|^{\alpha - 1})\partial_x^s v \partial_x v \|_{L_T^\infty L^2_x}\\
\leq & C (\lambda^{-|\alpha-1|}+\delta^{|\alpha-1|})(\|\partial_x^s u\|_{L^{\infty}_T L^{2}_x}\|\langle x \rangle^m \partial_x(u-v)\|_{L^{\infty}_T L^{\infty}_x}\\
&+\|\langle x \rangle^m\partial_x v\|_{L^{\infty}_TL^{\infty}_x}\|\partial_x^s(u-v) \|_{L^{\infty}_TL^{2}_x})\\
&+C(\lambda^{-|\alpha-2|}+\delta^{|\alpha-2|})\|\partial_x^s v\|_{L^{\infty}_T L^{2}_x}\|\langle x\rangle^{m} \partial_x v\|_{L^{\infty}_T L^{\infty}_x}\|\langle x\rangle^m(u- v)\|_{L^{\infty}_T L^{\infty}_x}\\
\leq & C\langle \lambda^{-1}+\delta \rangle^{\alpha+2s+1}d_{\mathcal{X}_T}(u,v).
\end{aligned}
\end{equation*}
Therefore, plugging the previous estimates in \eqref{eqWLEQ11.1}, we have
\begin{equation}\label{eqWLEQ12}
    \begin{aligned}
    \mathcal{III}_s \leq & C T \langle \lambda^{-1}+\delta \rangle^{\alpha+2s+1}d_{\mathcal{X}_T}(u,v).
    \end{aligned}
\end{equation}
We combine \eqref{eqWLEQ11} and \eqref{eqWLEQ12} to get 
\begin{equation}\label{eqWLEQ12.1}
 \begin{aligned}
\|\partial_x^s  (\Phi(u)-\Phi(v))\|_{L^\infty_T L^2_x}& +\|\partial_x^{s+1} (\Phi(u)-\Phi(v))\|_{L^\infty_x L^2_T}\\
&\leq  C T^{\frac{1}{2}}\langle T\rangle^{\frac{1}{2}} \langle \lambda^{-1}+\delta \rangle^{\alpha+2s+1}d_{\mathcal{X}_T}(u,v).
\end{aligned}   
\end{equation}
A similar reasoning yields
\begin{equation}\label{eqWLEQ12.1.1}
\begin{aligned}
\|\Phi(u) - \Phi(v)\|_{L^\infty_T L^2_x}\leq CT^{\frac{1}{2}}\langle T \rangle^{\frac{1}{2}}\langle \lambda^{-1}+\delta \rangle^{\alpha+2s+1} d_{\mathcal{X}_T}(u,v). 
\end{aligned}
\end{equation}

$\bullet$ {\bf Estimate for $\|\langle x \rangle^{m}(\Phi(u)-\Phi(v))\|_{L^\infty_T L^{\infty}_x}$}. We use the Sobolev embedding and Lemma \ref{derivexp2} to deduce
\begin{equation}\label{eqWLEQ12.2}
\begin{aligned}
\|\langle x \rangle^m  (\Phi(u) - \Phi(v))\|_{L^\infty_T L^\infty_x} \leq & C T\langle T \rangle^{\lfloor m \rfloor +1} \|\langle x \rangle^m (|u|^\alpha \partial_x u - |v|^\alpha \partial_x v)\|_{L_T^{\infty} L^2_x}\\
&+ C T\langle T \rangle^{\lfloor m \rfloor +1}\|\langle x \rangle^m \partial_x(|u|^\alpha \partial_x u - |v|^\alpha \partial_x v)\|_{L_T^{\infty} L^2_x}\\
&+ C\langle T\rangle^{\lfloor m \rfloor+1} \|\partial_x^s(|u|^\alpha \partial_x u - |v|^\alpha \partial_x v)\|_{L_T^{1} L^2_x}.
\end{aligned}
\end{equation}
The estimate for the last term on the right-hand side of \eqref{eqWLEQ12.2} follows exactly as in the argument to control \eqref{eqWLEQ9.5}. On the other hand, the estimates for the first two terms can be obtained by repeating 
the arguments using \eqref{betaesti} and \eqref{eqWLEQ10}. We summarize these estimates as follows 
\begin{equation}\label{eqWLEQ13.1}  
    \begin{aligned}
\|\langle x \rangle^m (\Phi(u) - \Phi(v))\|_{L^\infty_T L^\infty_x}\leq & CT^{\frac{1}{2}}\langle T\rangle^{\lfloor m\rfloor+\frac{3}{2}}\langle \lambda^{-1}+ \delta \rangle^{\alpha+2s+1}d_{\mathcal{X}_T}(u,v).
    \end{aligned}
\end{equation}

$\bullet$ {\bf Estimate for $\|\langle x \rangle^m \partial_x^l (\Phi(u) -\Phi(v))\|_{L^\infty_T L^2_x}$, $0\leq l\leq 4$.} An application of Lemma \ref{derivexp2} yields
\begin{equation}\label{eqWLEQ14}
\begin{aligned}
\|\langle x \rangle^m & \partial_x^l (\Phi(u) -\Phi(v))\|_{L^\infty_T L^2_x}\\
\leq & C \langle T\rangle^{\lfloor m \rfloor+1} \|\langle x \rangle^m \partial_x^l (|u|^\alpha \partial_x u-|v|^\alpha \partial_x v )\|_{L^1_T L^2_x}\\
&+C\langle T\rangle^{\lfloor m \rfloor+1} \|\partial_x^s(|u|^\alpha \partial_x u - |v|^\alpha \partial_x v)\|_{L^1_T L^2_x}.
\end{aligned}
\end{equation}
The estimate for the second term on the right-hand side of the above inequality follows from \eqref{eqWLEQ9.5}. By interpolation, we divide the estimate for the first term on the right-hand side of \eqref{eqWLEQ14} as follows 
\begin{equation*} 
\begin{aligned}
\|\langle x \rangle^m  \partial_x^l  (|u|^\alpha \partial_x u-|v|^\alpha \partial_x v )\|_{L^1_T L^2_x}\leq & C\|\langle x \rangle^m (|u|^{\alpha - l}(\partial_x u)^{l+1} - |v|^{\alpha - l}(\partial_x v)^{l+1})\|_{L^1_T L^2_x}\\
&+ C\|\langle x \rangle^m (|u|^{\alpha - 1} \partial_x^l u \partial_x u -  |v|^{\alpha - 1} \partial_x^l v \partial_x v)\|_{L^1_T L^2_x}\\
&+ C\|\langle x \rangle^m (|u|^\alpha \partial_x^{l+1} u - |v|^\alpha \partial_x^{l+1} v)\|_{L^1_T L^2_x}\\
=:& \mathcal{V}_{1} + \mathcal{V}_{2} + \mathcal{V}_{3}.
\end{aligned}
\end{equation*}
The estimates for $\mathcal{V}_{1}$ and $\mathcal{V}_{2}$ can be obtained via previously derived arguments. Thus, we have
\begin{equation*}
\begin{aligned}
\mathcal{V}_{1}+\mathcal{V}_{2} \leq CT\langle \lambda^{-1}+\delta \rangle^{\alpha+2s+1} d_{\mathcal{X}_T}(u,v). 
\end{aligned}
\end{equation*}
Finally, we estimate $\mathcal{V}_{3}$. The cases $1 \leq l < 4$ can be obtained similar to the above. We only consider the case $l=4$. We apply \eqref{eqWLEQ10} to get
\begin{equation*}
\begin{aligned}
\mathcal{V}_{3} \leq & CT(\lambda^{-|\alpha-1|}+\delta^{|\alpha-1|}) \|\langle x \rangle^{m(1-\alpha)}\partial_x^5 v\|_{L^{\infty}_T L^{2}_x}\|\langle x\rangle^{m}(u-v)\|_{L^{\infty}_T L^{\infty}_x}\\
&+CT\|\langle x \rangle^{m} u\|_{L^{\infty}_T L^{\infty}_x}^{\alpha}\|\langle x\rangle^{m(1-\alpha)}\partial_x^5(u-v)\|_{L^{\infty}_T L^2_x}\\
\leq & CT\langle \lambda^{-1}+\delta \rangle^{\alpha+2s+1} d_{\mathcal{X}_T}(u,v),
\end{aligned}
\end{equation*}
where the previous estimate is easily checked for $\alpha \geq 1$.When $0<\alpha<1$, by Lemma \ref{Interplemma} and the ideas in  \eqref{eqWLEQ7.1.1}, the same inequality holds. Therefore, we deduce
\begin{equation}\label{eqWLEQ15} 
\begin{aligned}
\sup_{0\leq l \leq 4}  \|\langle x \rangle^m  \partial_x^l (\Phi(u) -\Phi(v))\|_{L^\infty_T L^2_x} \leq & CT^{\frac{1}{2}}\langle T \rangle^{\lfloor m\rfloor+\frac{3}{2}}\langle \lambda^{-1}+\delta \rangle^{\alpha+2s+1} d_{\mathcal{X}_T}(u,v).
\end{aligned}
\end{equation}
Collecting \eqref{eqWLEQ12.1}, \eqref{eqWLEQ12.1.1}, \eqref{eqWLEQ13.1} and \eqref{eqWLEQ15}, we have
\begin{align*}
d_{\mathcal{X}_T}(\Phi(u),\Phi(v)) &\leq CT^{\frac{1}{2}}\langle T \rangle^{\lfloor m\rfloor+\frac{3}{2}}\langle \lambda^{-1}+\delta \rangle^{\alpha+2s+1} d_{\mathcal{X}_T}(u,v).
\end{align*}
By taking $T=T(\delta, \alpha)>0$ sufficiently small such that \eqref{eqWLEQ9.4} and
\begin{equation*}
    CT^{\frac{1}{2}}\langle T \rangle^{\lfloor m\rfloor+\frac{3}{2}}\langle \lambda^{-1}+\delta \rangle^{\alpha+2s+1} \leq \frac{1}{2}
\end{equation*}
are valid, yields that $\Phi$ is a contraction on $\mathcal{X}_T$. Hence, by the Banach fixed point theorem the integral equation associated to the GKdV equation has a unique solution in $\mathcal{X}_T$. This completes the existence part. The remaining properties stated in Theorem \ref{mainTHM} are deduced by standard arguments (see for example the arguments in \cite{LinaresI}), for the sake of brevity we omit their deduction. The proof of Theorem \ref{mainTHM} is therefore complete. \smallskip

Now that we have obtained the local well-posedness of solutions to gKdV and GKdV in $\mathcal{X}_T \subset H^1$, one of the natural questions would be to investigate if solutions can be extended to globally existing ones or they would only exist for a finite time. Asymptotic behavior of global solutions would be another question to investigate. We mention that Theorem \ref{mainTHM} can address solutions with initial data that decay rather slowly, for example, with the decay rate $\frac1{|x|^{1/2 +}}$ as $|x| \to \infty$. 
However, Theorem \ref{mainTHM} can not handle the data, which decay faster than a polynomial (for example, exponential decay such as in the ground state). To address that we do numerical investigations, results of which we describe in the next sections. 


\section{Numerical study: Single-bump data}\label{S:Numerical results}

We now study solutions to gKdV and GKdV numerically (a brief review of the numerical approach is described in Appendix). We first consider a physically relevant version of the modular KdV, Schamel's equation \eqref{E:Schamel} (i.e., GKdV with $\alpha = \frac12$) and compare the evolution of positive and negative generic data in that case. After that,  we consider powers that are either very close to zero or close to 1 in order to better understand the influence of the power $\alpha$ and of the modulus in the nonlinearity (gKdV vs. GKdV); thus, we take $\alpha=\frac19$ and $\alpha=\frac79$, and later compare those cases with the integer powers $\alpha=1$ and $\alpha = 3$ in both equations. 
  
\subsection{Schamel's equation}
To start with, we consider Gaussian initial data 
$$
u_0(x)=A \, e^{-x^2},
$$ 
and show the time evolution of the solution $u(x,t)$ to the Cauchy problem \eqref{GK} for the positive and negative amplitudes $A=6$ and $A=-6$ in Figure \ref{F:profile-Schamel}. 
\begin{figure}[ht]
\includegraphics[width=0.32\textwidth]{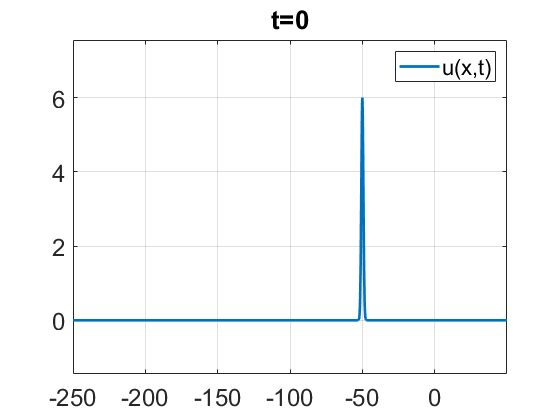}
\includegraphics[width=0.32\textwidth]{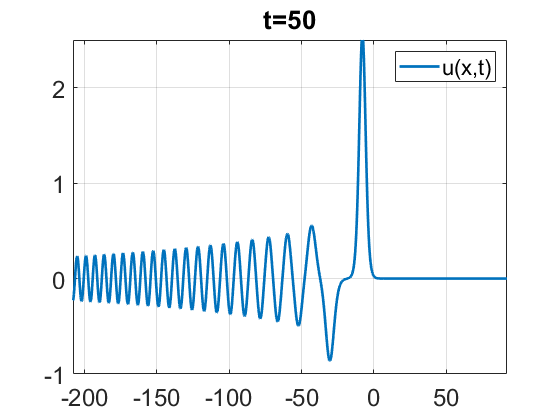}
\includegraphics[width=0.32\textwidth]{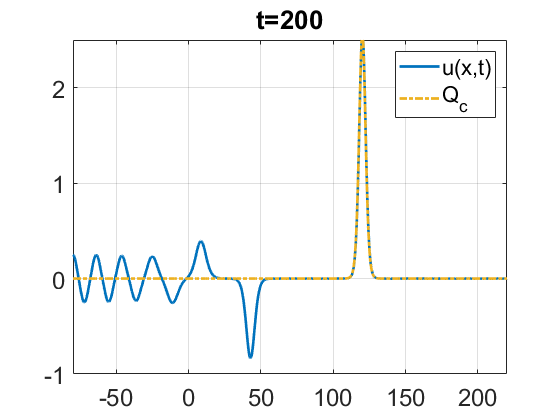}
\includegraphics[width=0.32\textwidth]{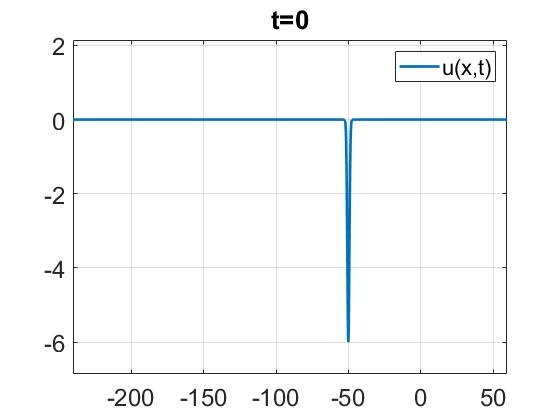}
\includegraphics[width=0.32\textwidth]{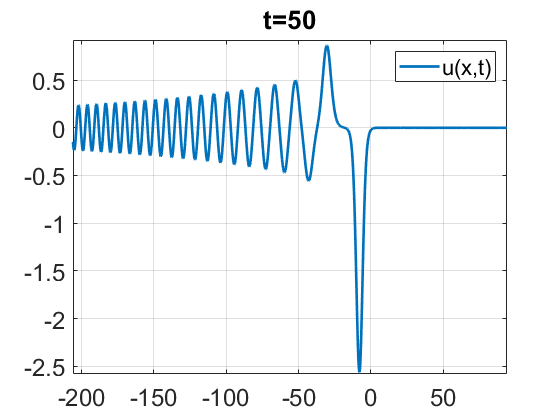}
\includegraphics[width=0.32\textwidth]{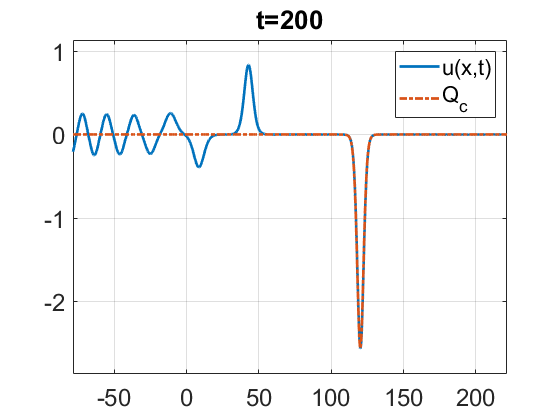}
\caption{\footnotesize Snapshots of time evolution in Schamel's equation of Gaussian data $u_0=A\, e^{-x^2}$ (left), at $t=50$ (middle) and $t=200$ (right) with the fitting to the rescaled soliton $Q_c$. 
Top row: $A=6$. Bottom row: $A=-6$.} 
\label{F:profile-Schamel}
\end{figure}
As typical for the KdV-type equations,  a part of the solution propagates to the right as a soliton (or several solitons) and another part of the solution produces dispersive oscillations to the left, referred to as the {\it radiation}, decaying toward negative infinity. Both of these parts are present, for example, at time $t=50$, see middle column of Figure \ref{F:profile-Schamel}. To confirm this further, we continue simulating both solutions up to time $t=200$, where we also fit the largest amplitude (in absolute value) lump with the rescaled soliton $Q_c$, the ground state rescaled solution of the equation \eqref{E:Q_c}, where the scaling parameter $c$, or the speed of the soliton, is obtained 
by taking the maximum height (at the time of fitting) as $\|u(t)\|_{L^\infty_x}=\|Q_c\|_{L^\infty_x}$, and then, finding the constant $c$ from \eqref{explground} and \eqref{E:Q_c}, see the $L^\infty_x$ norm in Figure \ref{F:supnorm-Schamel}.
Note that the time evolution of both conditions (see snapshots in the top and bottom rows of Figure \ref{F:profile-Schamel}) is symmetric with respect to the $x$-axis, as expected due to the absolute value in the nonlinearity (we show later that it will be different for nonlinearities without absolute value). 
In Figure \ref{F:profile-Schamel} (right column) one can notice that a second soliton is forming, and even a closer look will reveal a third soliton. The second soliton for the positive data (top row) is negative, and thus, can be tracked via $\min(u(t))$, since the second soliton has the largest magnitude among the negative solitons, this is shown in Figure \ref{F:supnorm-Schamel}. (By symmetry, one could track $\max(u(t))$ for the negative data in the bottom row.) Note that both graphs in Figure \ref{F:supnorm-Schamel} become horizontal after some time, confirming that the first and second solitons have formed.  

\begin{figure}[ht]
\includegraphics[width=0.32\textwidth]{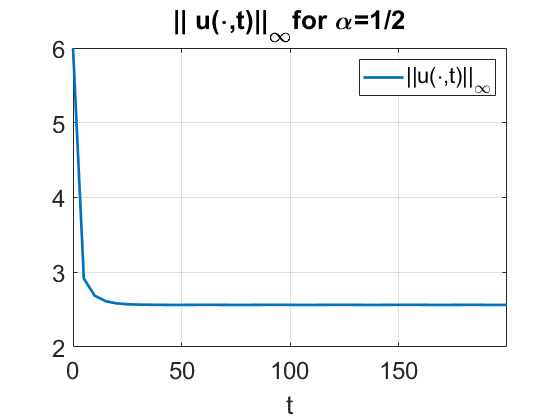}
\includegraphics[width=0.32\textwidth]{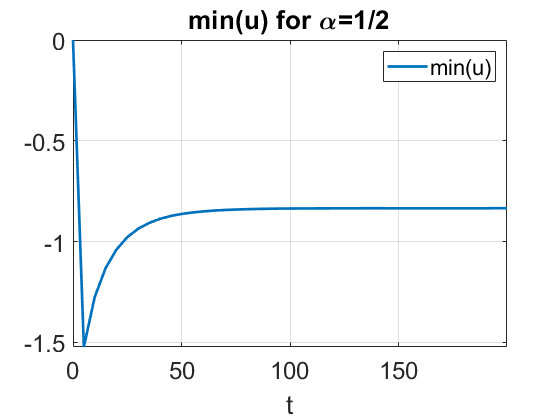}
\caption{\footnotesize Time dependence of the $L^\infty_x$ norm (left) and $\min(u(t))$ for the solutions to the Schamel equation shown in Figure \ref{F:profile-Schamel}.}
\label{F:supnorm-Schamel}
\end{figure}

\subsection{Perturbations of soliton}
From now on we consider powers $\alpha = \frac19$ and $\frac79$ in \eqref{gKdV} and \eqref{GK}  
to compare solutions of the gKdV and GKdV equations.
We start with the ground state initial data 
\begin{equation}\label{E:AQ}
\quad u_0(x)=A \, Q(x + a), \quad A \in \mathbb R\setminus \{0\}, ~~a \in \mathbb R.
\end{equation}

We take $u_0(x)=Q(x+25)$ as the initial condition for the gKdV and GKdV equations considering different powers of $\alpha$. 
To distinguish solutions to gKdV vs. GKdV, we denote by $u=u(x,t)$ the gKdV solution for \eqref{gKdV} and $v = v(x,t)$ the GKdV solution for \eqref{GK}. 

\begin{figure}[ht]
\includegraphics[width=0.32\textwidth]{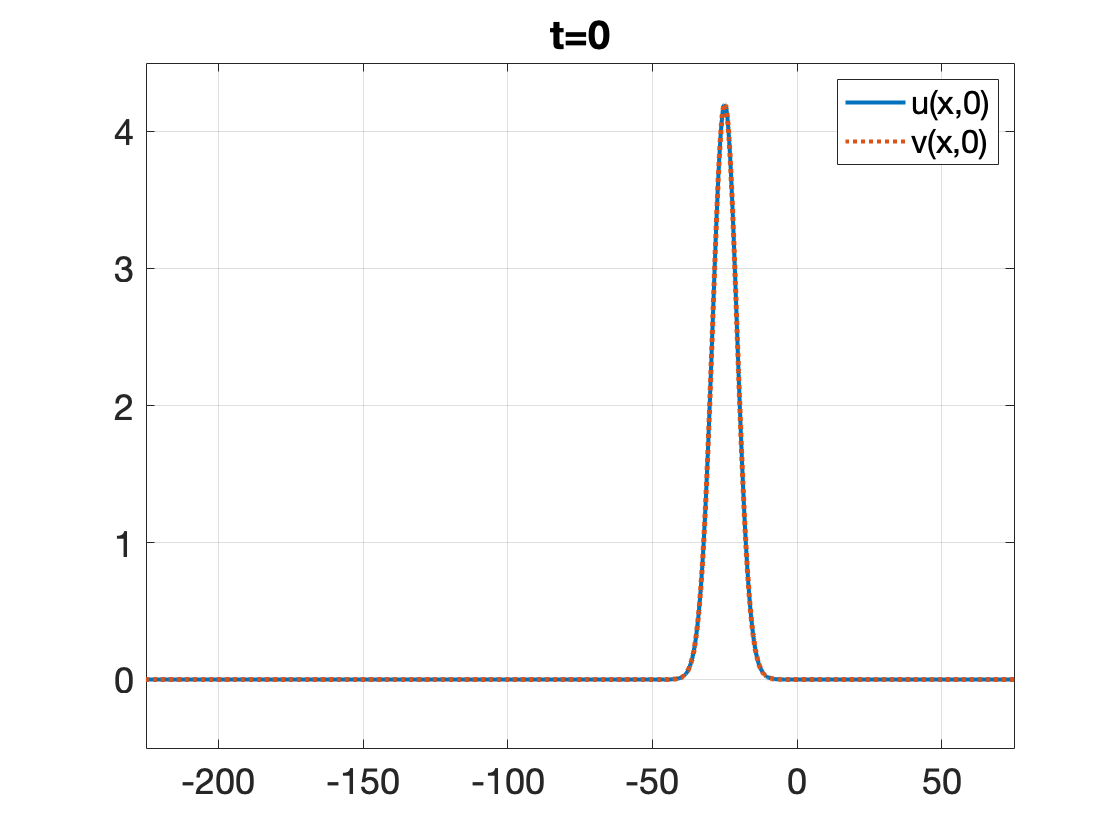}
\includegraphics[width=0.32\textwidth]{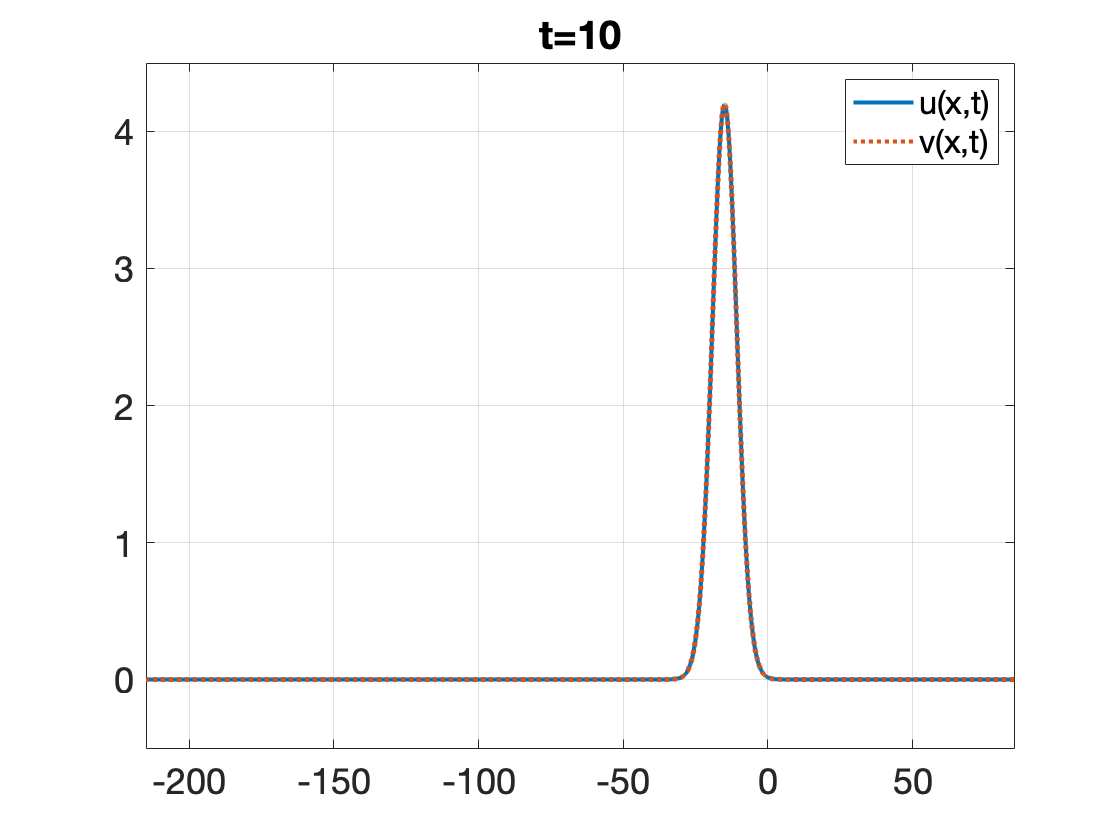}
\includegraphics[width=0.32\textwidth]{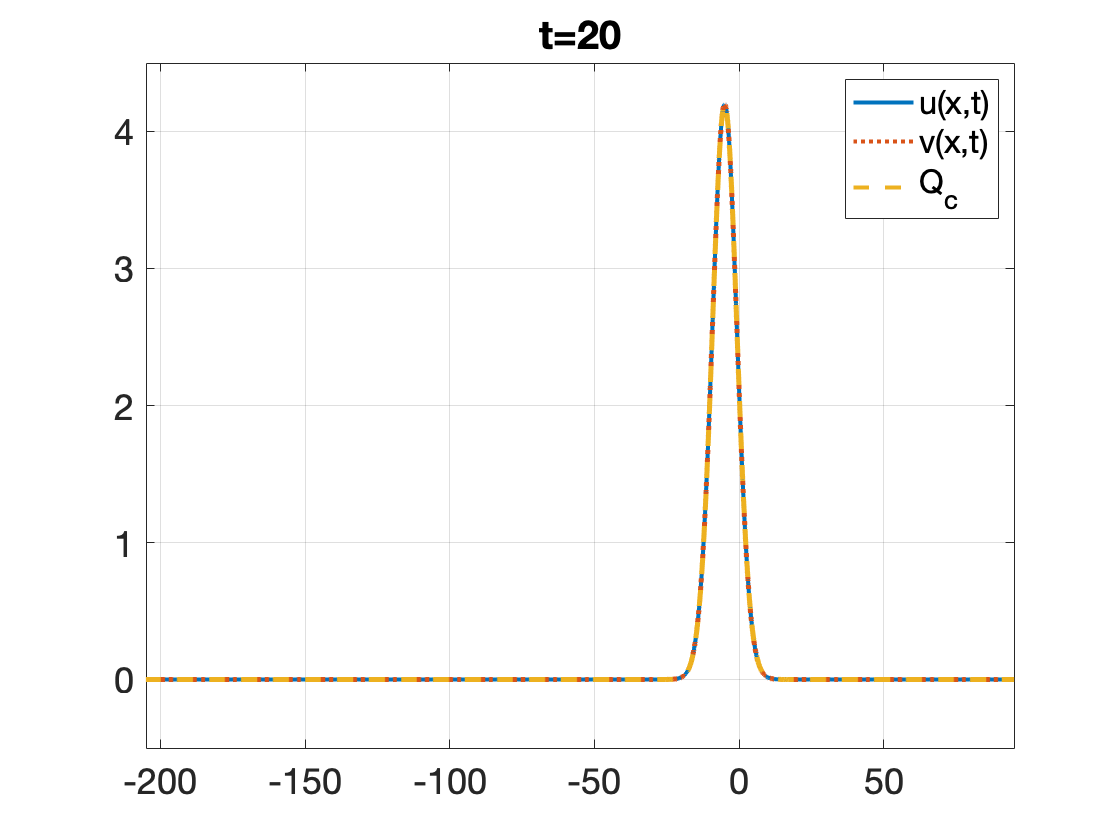}
\includegraphics[width=0.32\textwidth]{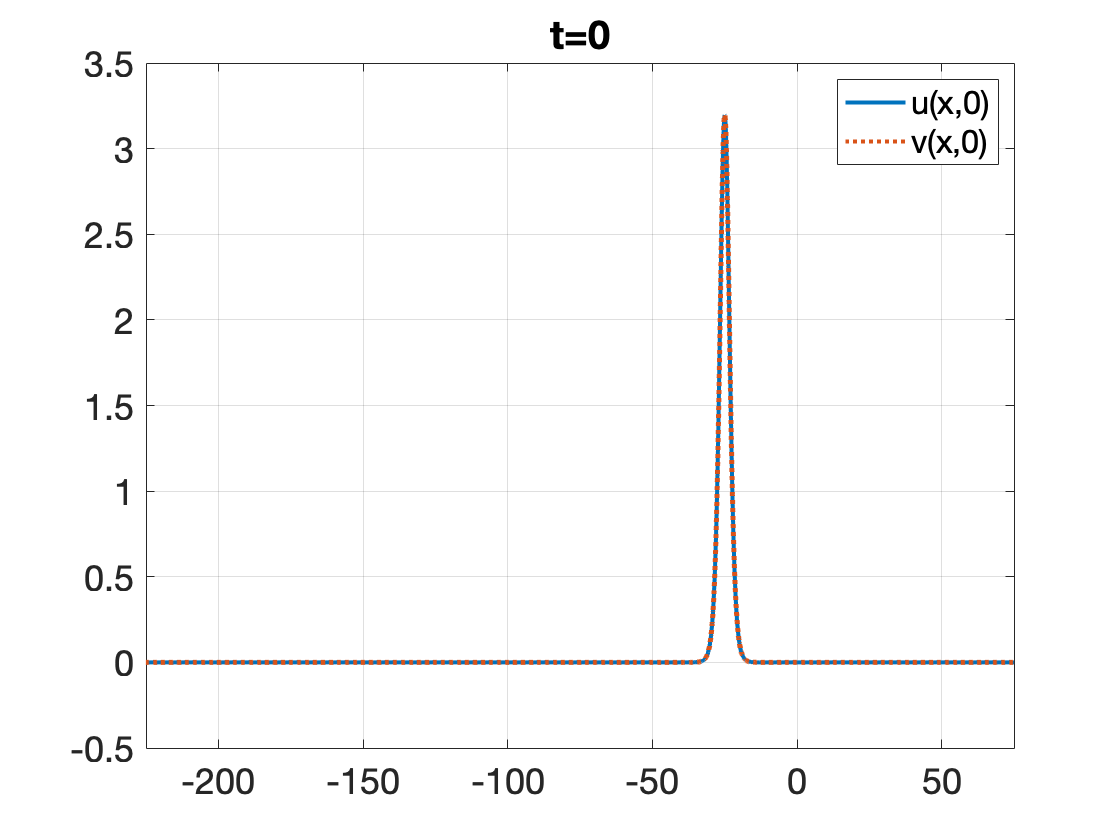}
\includegraphics[width=0.32\textwidth]{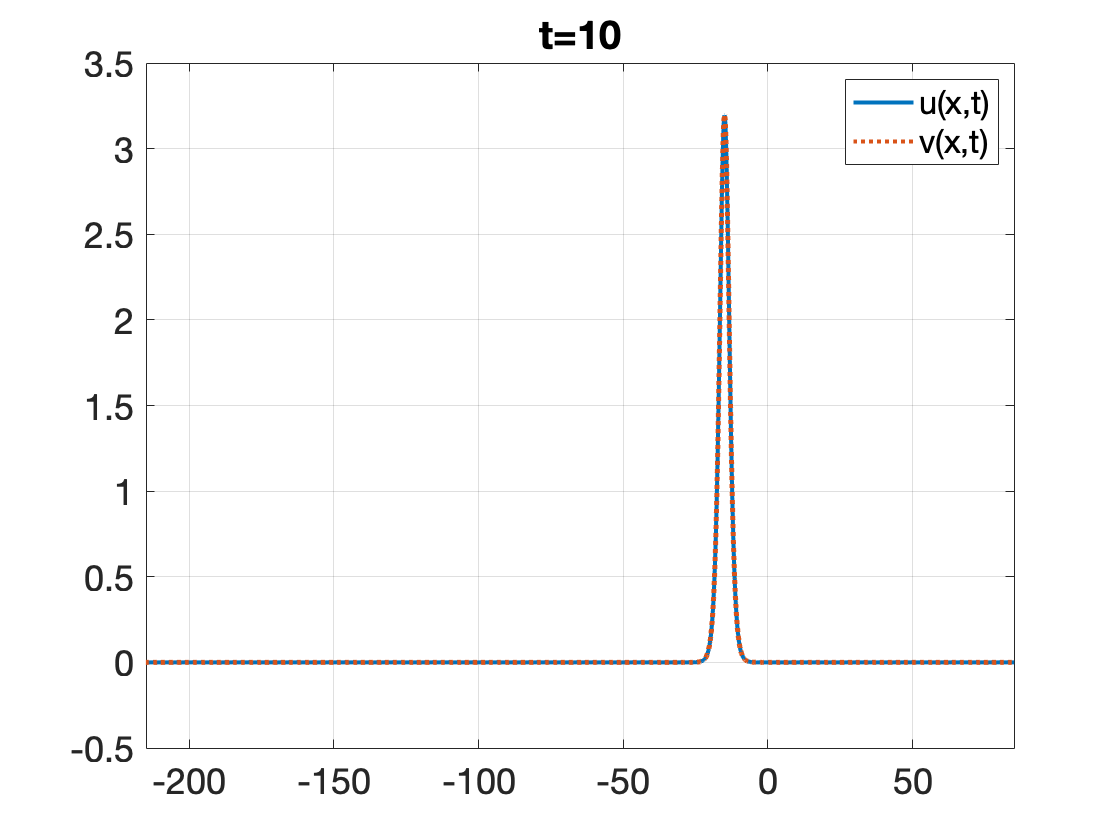}
\includegraphics[width=0.32\textwidth]{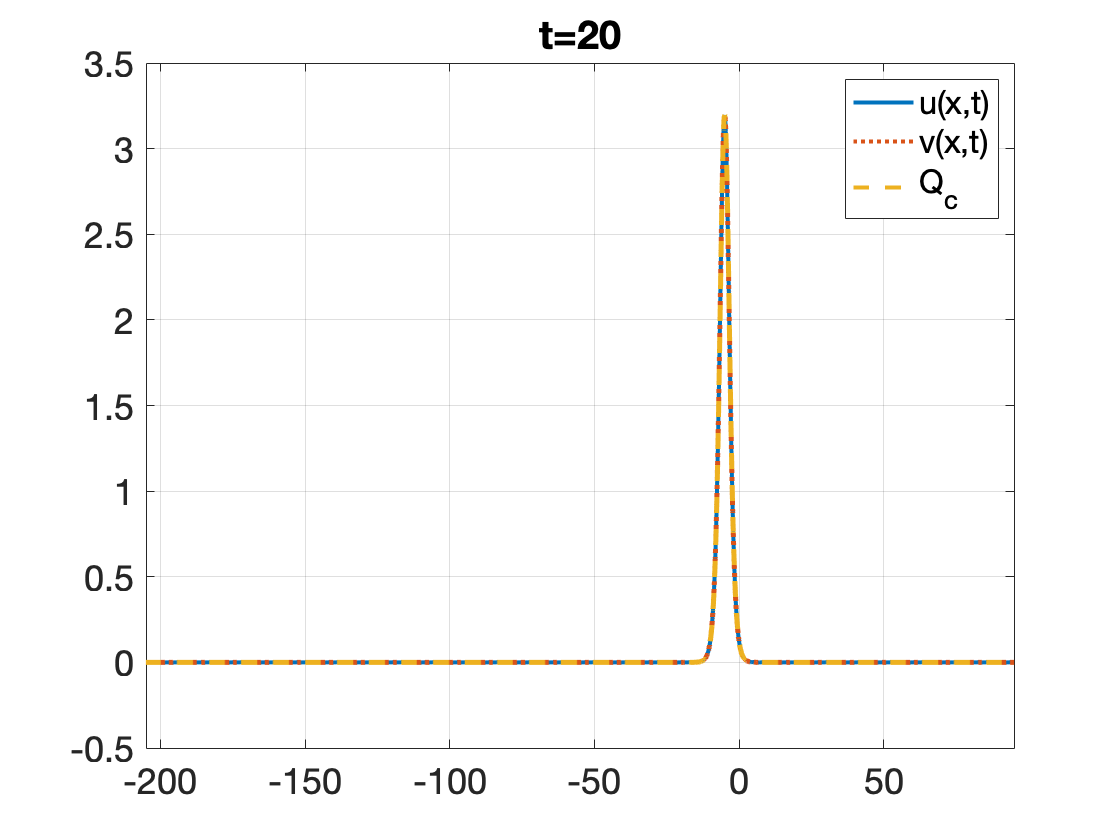}
\caption{\footnotesize Time evolution for $u_0=v_0 = Q(x+25)$ for $\alpha=\frac{1}{9}$ (top row) and   $\alpha=\frac{7}{9}$ (bottom row); solution $u$ of \eqref{gKdV} (solid blue) and $v$ of \eqref{GK} (dotted red). Right column: both solutions are fitted with shifted 
$Q_c = Q$ from \eqref{explground} ($c=1$). }
\label{F:profile Qfor1979}
\end{figure} 

Figure \ref{F:profile Qfor1979} shows snapshots of the time evolution of these solutions ($A=1$); both of them keep their shapes and they both propagate to the right as expected from a solitary wave in a KdV-type equation. Moreover, the two solutions coincide (with a difference at the level of a numerical error).  
Indeed, since the initial condition is positive, both solutions travel as a (positive) soliton, i.e., both $u(x,t)>0$ and $v(x,t)>0$ for all time $t$. Consequently, there is no difference between $u(x,t)$ and $v(x,t)$. 
\\

Time evolution becomes significantly different when initial data are negative, that is, \eqref{E:AQ} with  $A<0$. We look at the fractional powers $\alpha$ first, and later compare with the integer powers $\alpha$.  We take $u_0=v_0=-Q(x+25)$. 
Figure \ref{F:profile QnegA0} shows the time evolution of these conditions when $\alpha=\frac19$ and $\frac{7}{9}$.  Note, that for both nonlinearities the gKdV solution $u(x,t)$ (solid blue) all goes into the radiation (other nonlinearities have similar behavior as well, however, see more on emergence of solitons from radiation at the end of Section \ref{S:Gaussian-data}), while the GKdV evolution of the same data propagates the (negative) soliton, see the fitting with $Q_c$ in the right column of Figure \ref{F:profile QnegA0}, similar to what we observed in the Schamel equation. 

\begin{figure}[ht]
\includegraphics[width=0.32\textwidth]{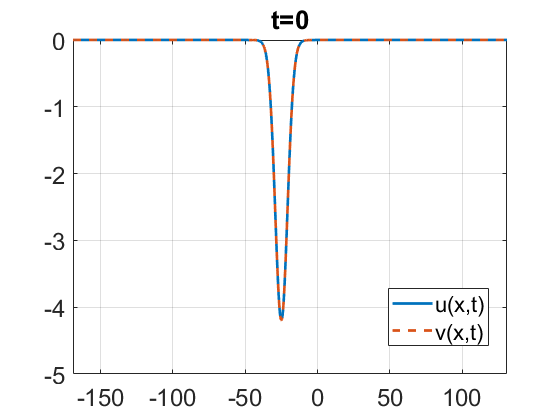}
\includegraphics[width=0.32\textwidth]{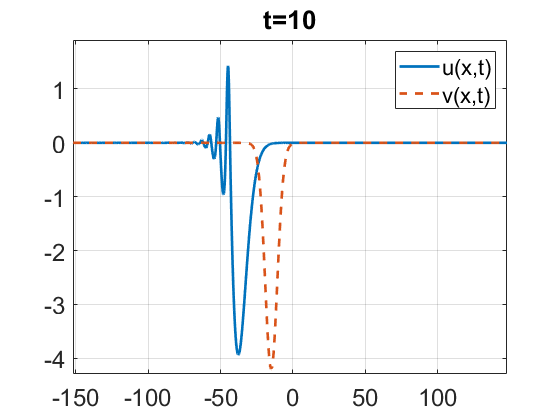}
\includegraphics[width=0.32\textwidth]{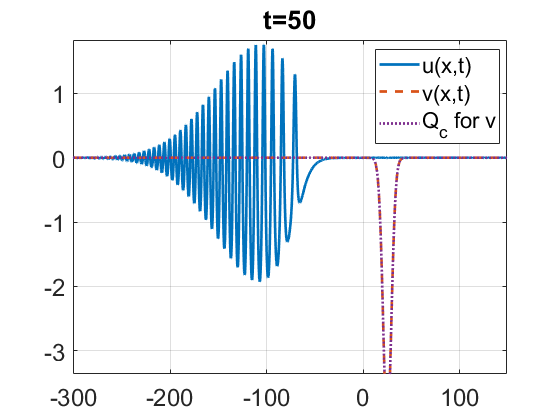}
\includegraphics[width=0.32\textwidth]{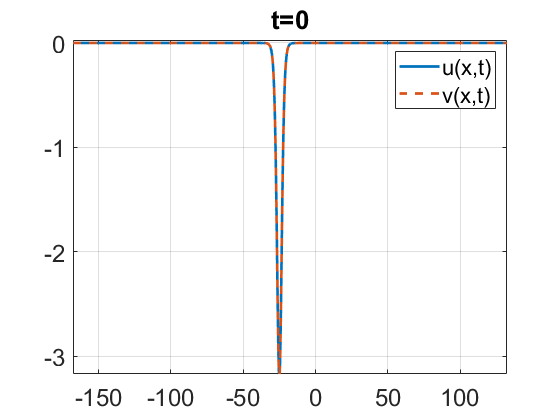}
\includegraphics[width=0.32\textwidth]{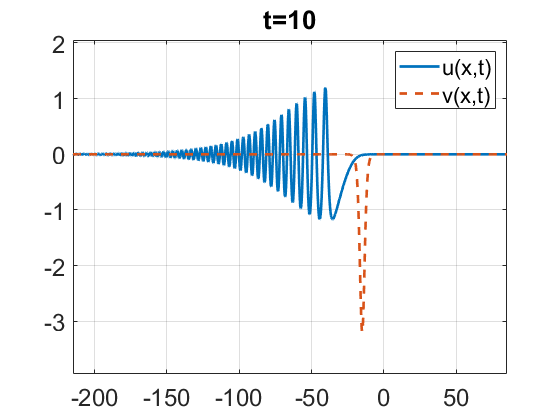}
\includegraphics[width=0.32\textwidth]{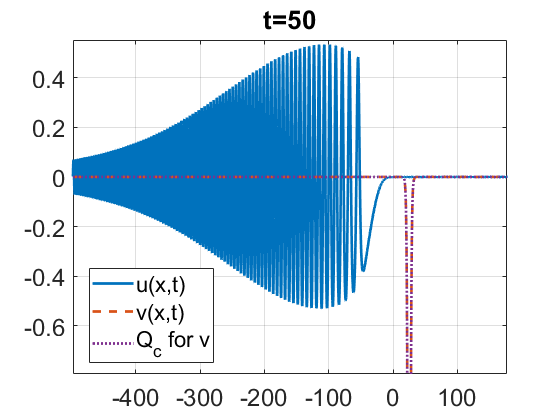}
\caption{\footnotesize Time evolution for $u_0=v_0= - Q(x+25)$ 
for $\alpha=\frac{1}{9}$ (top row) and $\alpha=\frac{1}{9}$ (bottom row).  
Right column: the GKdV solution $v(x,t)$ (dashed red) fitted 
to shifted $Q_c$ (dotted magenta). }
\label{F:profile QnegA0}
\end{figure}

\begin{figure}[ht]
\includegraphics[width=0.32\textwidth]{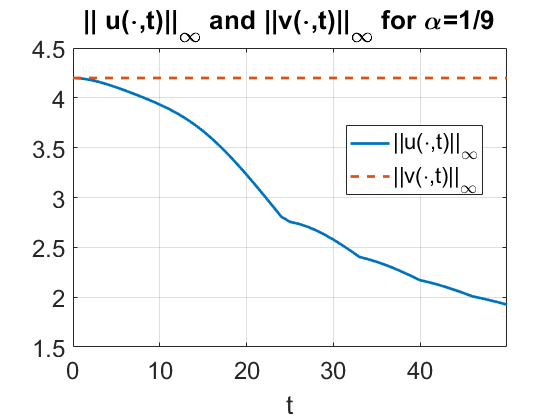}
\includegraphics[width=0.32\textwidth]{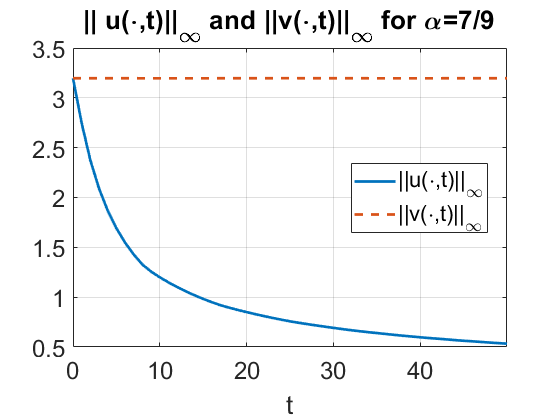}
\caption{\footnotesize Comparison of the $L^\infty_x$ norms for gKdV and GKdV solutions: 
$\alpha=\frac{1}{9}$ (left), $\alpha=\frac{7}{9}$ (right).}
\label{F:A NQ}
\end{figure}

\smallskip
 
In Figure \ref{F:A NQ}, we track the quantities $\|u(t)\|_{L^\infty_x}$ and 
$\|v(t)\|_{L^\infty_x}$ for $\alpha=\frac{1}{9}$ (left) and $\alpha=\frac{7}{9}$ (right). Observe that the GKdV solutions (red dash) in both cases of nonlinearity have a constant $L^\infty_x$ norm, confirming the propagation of the solitons $-Q$ (in the top row of Figure \ref{F:profile QnegA0}) and $-Q$ (in the bottom row), since the $L^\infty_x$ norm of $v$ (red dash) in Figure \ref{F:A NQ} shows the constant height as well as the fitting in Figure \ref{F:profile QnegA0} (right column) matches the shape. On the other hand, the gKdV solutions (blue solid line in Figure \ref{F:A NQ}) decrease their height as they radiate away to the left. 
\smallskip

We next compare the low fractional powers of $\alpha$ with the integer powers, $\alpha = 1$ and $3$; the powers are chosen so we could compare with the known cases of the KdV equation, but still would be different from GKdV (due to the absolute value in the nonlinear term). The time evolution (for the same initial data \eqref{E:AQ} with $A=-1$) is shown in Figure \ref{F:profile QnegA}. The behavior is similar to the low powers $\alpha <1$ in both cases: the negative data radiates to the left (solid blue), and the GKdV solution propagates the negative soliton to the right, see the fitting for $v(x,t)$ with shifted $-Q$ in the right column of Figure \ref{F:profile QnegA}. 
\begin{figure}[ht]
\includegraphics[width=0.32\textwidth]{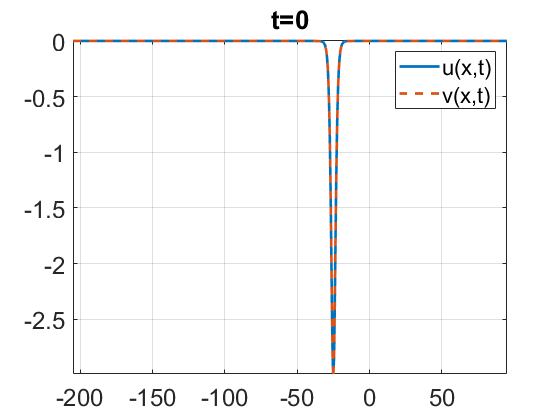}
\includegraphics[width=0.32\textwidth]{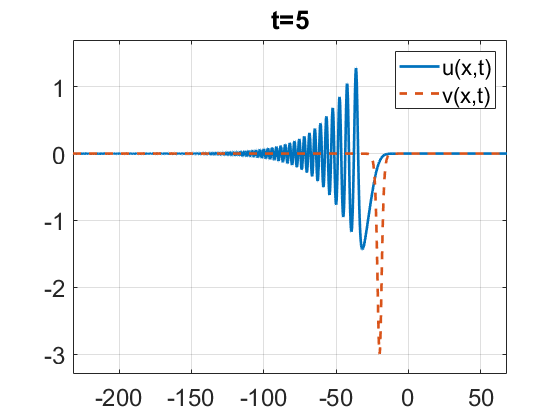}
\includegraphics[width=0.32\textwidth]{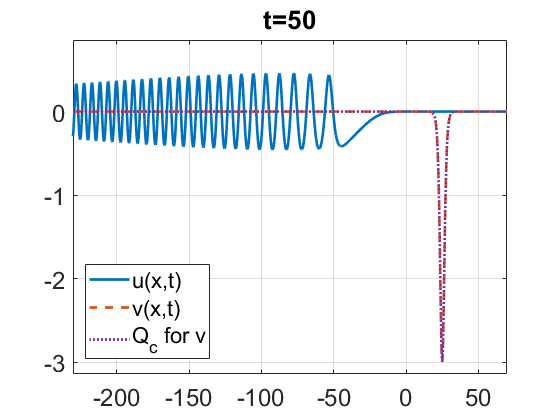}
\includegraphics[width=0.32\textwidth]{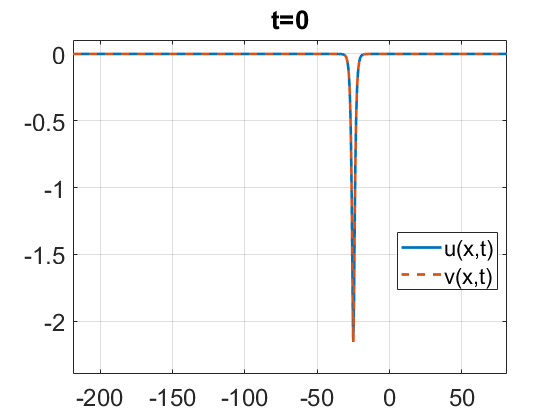}
\includegraphics[width=0.32\textwidth]{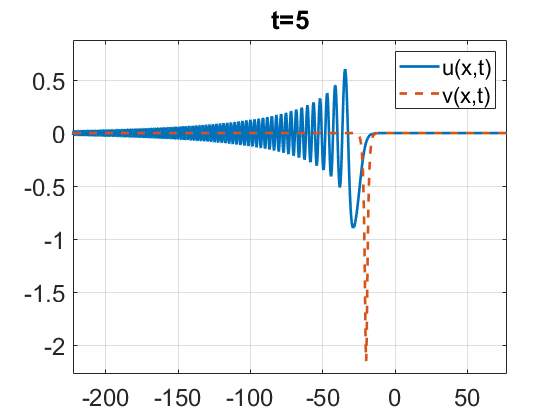}
\includegraphics[width=0.32\textwidth]{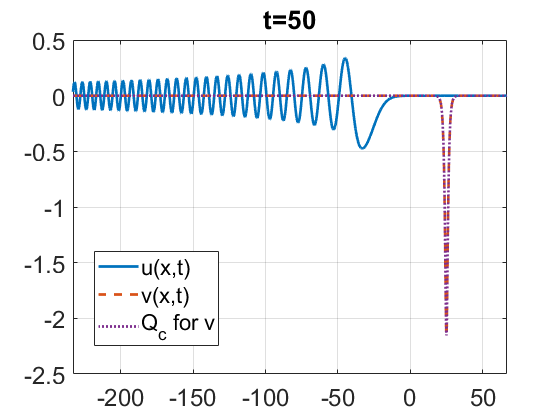}
\caption{\footnotesize Time evolution for $u_0=v_0= - Q(x+25)$. 
Top row: $\alpha=1$. Bottom row: $\alpha=3$.}
\label{F:profile QnegA}
\end{figure}

We conclude that solutions to gKdV and GKdV with negative initial data differ significantly; in the gKdV equation the negative soliton $Q$ (or its perturbation $A\,Q$) right away disperses to the left into the radiation (see more on that at the end of Section \ref{S:Gaussian-data}), while the GKdV equation has a symmetric behavior (with the positive data) and propagates solitons, regardless of the initial data sign.


\subsection{Soliton resolution}\label{S:3.2}
The ``soliton resolution conjecture" has bemused the researchers since 60's and yet it is still far from being understood even in more simple, or more studied cases of the gKdV equation. It states that any solution will eventually evolve into a  
finite number of solitons plus radiation, i.e., $u(x,t) \approx
\sum_{j=0}^{N} Q_{c_j}(x-c_jt-a_j)+r(x,t)$ as $t\rightarrow \infty$, 
where $r(x,t)$ is the radiation 
and $Q_c$ is some rescaled version of a soliton with a shift $a_j$ and speed $c_j = c_j(t) \to c_j^*$, \cite{T2004b, S2006}. This was an intriguing direction for our study as well, and in this section we report our numerical investigations on how various types of data resolve into solitons and radiation in both the gKdV and the GKdV equations. 


\subsubsection{Gaussian initial data}\label{S:Gaussian-data}
We start with $\alpha=\frac{1}{9}$ and consider the Gaussian initial data 
\begin{equation}\label{E:Gauss-data}
u_0=v_0=Ae^{-x^2}.
\end{equation} 
Snapshots of the time evolution for $A=6$ are shown in Figure \ref{F:profile EPos 19}. 
\begin{figure}[ht]
\includegraphics[width=0.32\textwidth]{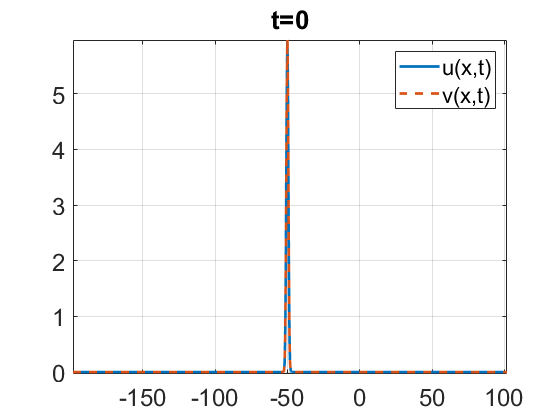}
\includegraphics[width=0.32\textwidth]{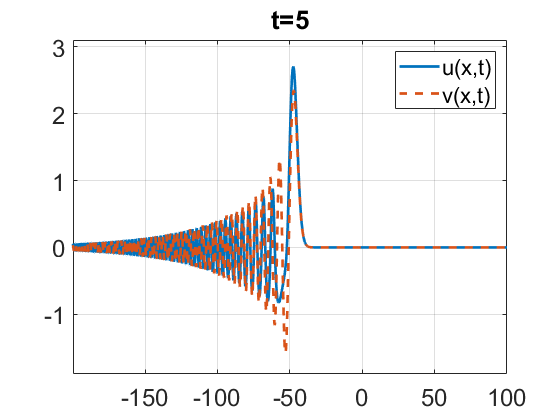}
\includegraphics[width=0.32\textwidth]{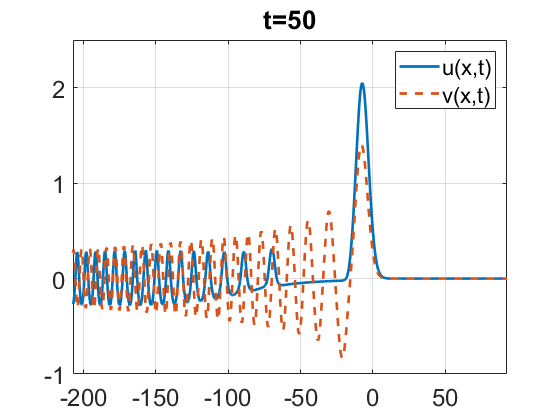}
\includegraphics[width=0.32\textwidth]{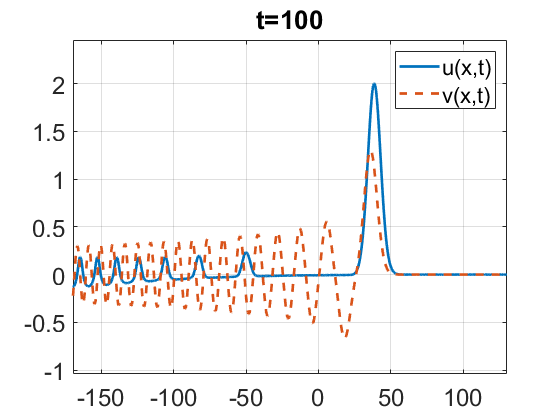}
\includegraphics[width=0.32\textwidth]{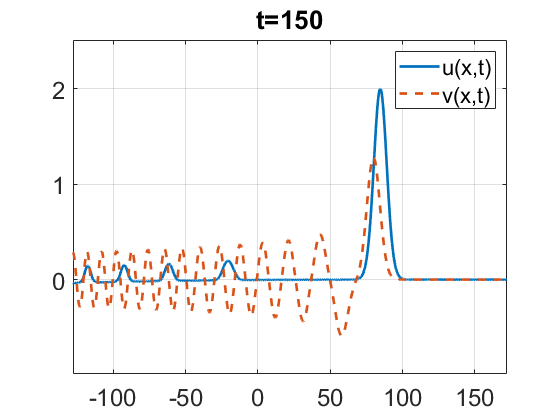}
\includegraphics[width=0.32\textwidth]{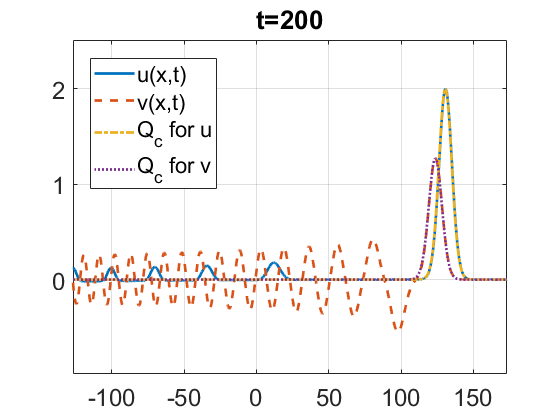}
\caption{\footnotesize Time evolution for $u_0=v_0=Ae^{-x^2}$,  $A=6$ and $\alpha=\frac{1}{9}$.}
\label{F:profile EPos 19}
\end{figure}
We specifically made the simulation for longer time (up to $t=200$) to show formation of what Miura called the ``parade of solitons", i.e., the formation of a train of solitons with decreasing heights (or speed), and thus, eventually separating further and further from each other, see separating bumps 
(solid blue line) in the bottom row of Figure \ref{F:profile EPos 19}. For $\alpha = \frac79$ we show the solution with the same initial data in Figure \ref{F:profile EPos 19b}.

\begin{figure}[ht]
\includegraphics[width=0.32\textwidth]{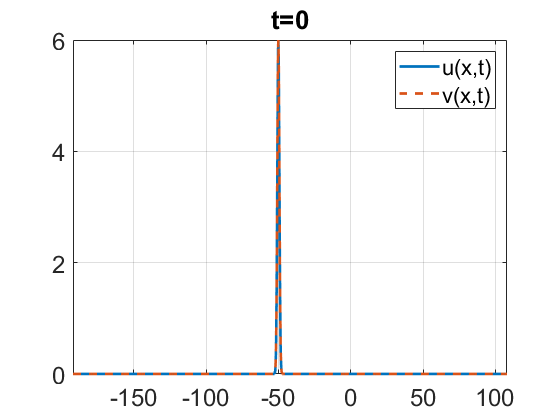}
\includegraphics[width=0.32\textwidth]{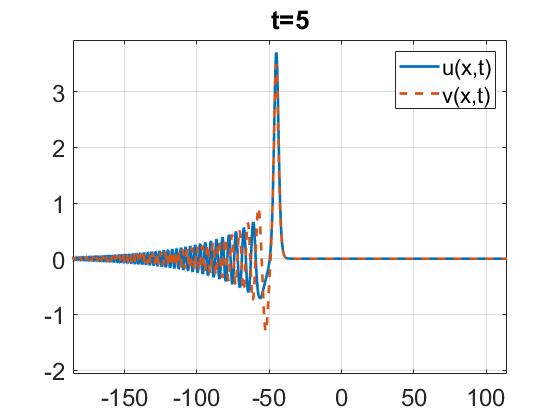}
\includegraphics[width=0.32\textwidth]{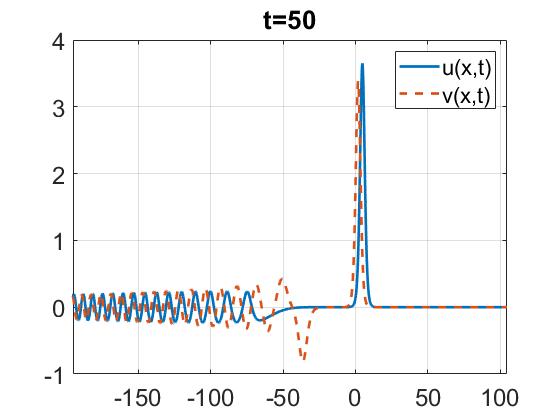}
\includegraphics[width=0.32\textwidth]{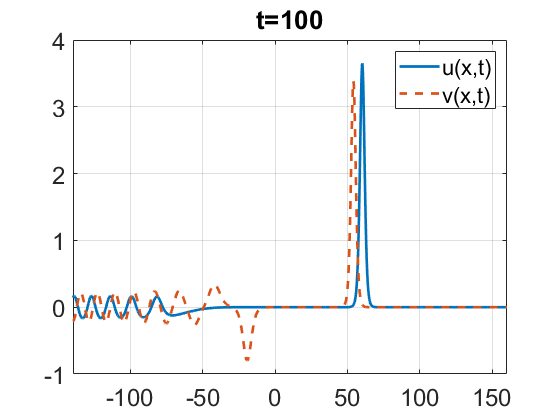}
\includegraphics[width=0.32\textwidth]{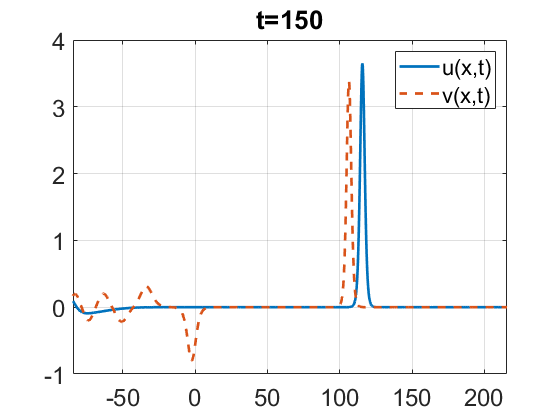}
\includegraphics[width=0.32\textwidth]{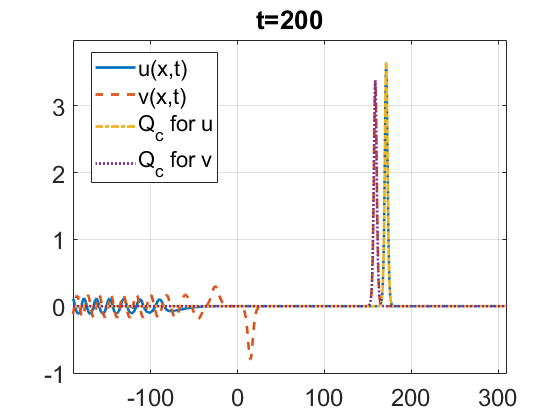}
\caption{\footnotesize Time evolution for $u_0=v_0=Ae^{-x^2}$, $A=6$ 
and $\alpha=\frac{7}{9}$.}
\label{F:profile EPos 19b}
\end{figure}
Computationally, it is often difficult to see the time evolution of {\it several} gKdV solitons in one plot, since a significantly large computational domain would be needed with a sufficiently refined mesh, however, in the case of small powers $\alpha \ll 1$ (in our computations, $\alpha=\frac19$), such a large domain is not needed, compared to higher powers. This observation might be useful for future studies, when it is important to track several solitons. (Compare this with the bottom row in Figure \ref{F:profile EPos 19b}, where even on a larger interval it is basically not possible to see emerging smaller solitons, thus, an advantage of considering very low powers of $\alpha$ in the nonlinear term.)  

{\bf Remark.} Note that the solitons forming in the GKdV model (with the absolute value) have a smaller height (or speed) compared to the gKdV model, this is due to the absolute value in the energy's second term (the potential energy term) in the GKdV equation, compare \eqref{E:GKdV} vs. \eqref{E:gKdV}. In Figures \ref{F:profile EPos 19} and \ref{F:profile EPos 19b}, one can see the difference in the height in the blue solid and red dash curves of solitons (and also the shift of the blue bump being more and more ahead of the red dash bump). In the last plot (when $t=200$),  we fit both bumps with the rescaled solitons $Q_{c}$ (of the respective heights for $u$ and $v$), see the bottom right subplots in both figures. We also observe that the GKdV solution $v(x,t)$ forms a second negative soliton (see the dash red curve in the bottom row in Figure \ref{F:profile EPos 19b}). It might also be the case that yet another (third) positive soliton forms, thus, the emerging smaller solitons alternate the positive and negative amplitudes.  The gKdV model forms only positive solitons, as discussed above. The higher the power $\alpha$ is (for a given amplitude $A$), 
there will be less mass ($L^2$ norm) to form another (smaller) soliton. 

\begin{figure}[ht]
\includegraphics[width=0.32\textwidth]{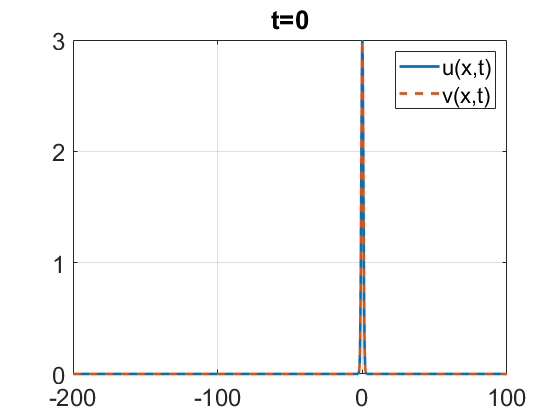}
\includegraphics[width=0.32\textwidth]{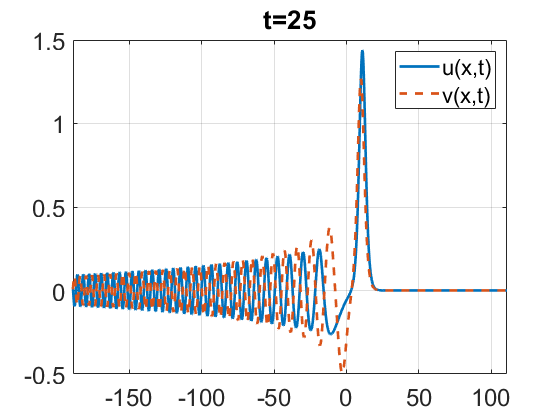}
\includegraphics[width=0.32\textwidth]{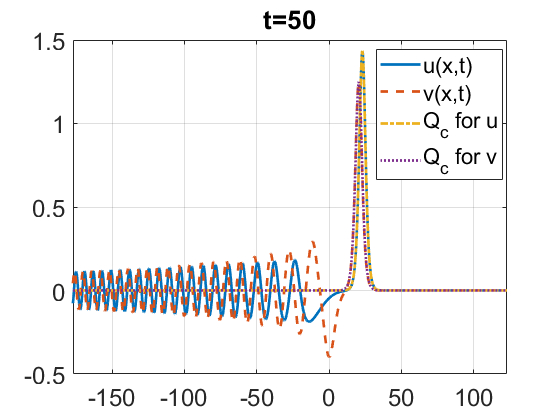}
\includegraphics[width=0.32\textwidth]{gaup11_p1.png}
\includegraphics[width=0.32\textwidth]{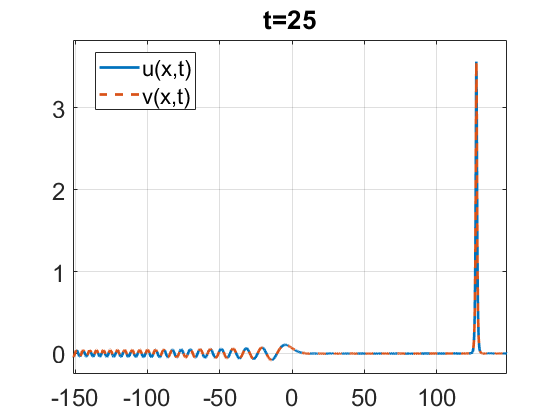}
\includegraphics[width=0.32\textwidth]{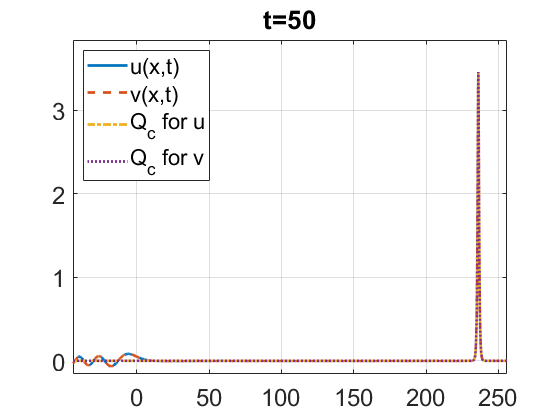}
\includegraphics[width=0.35\textwidth, height=4cm]{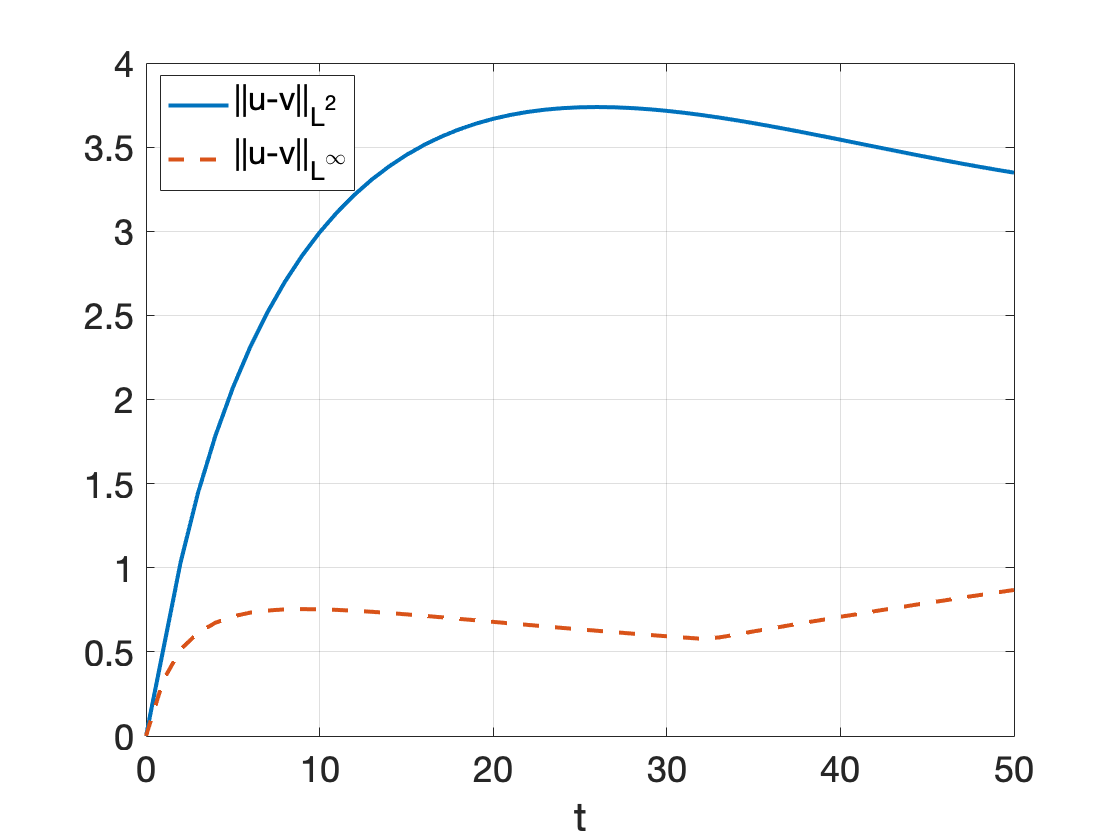}
\includegraphics[width=0.35\textwidth, height=4cm]{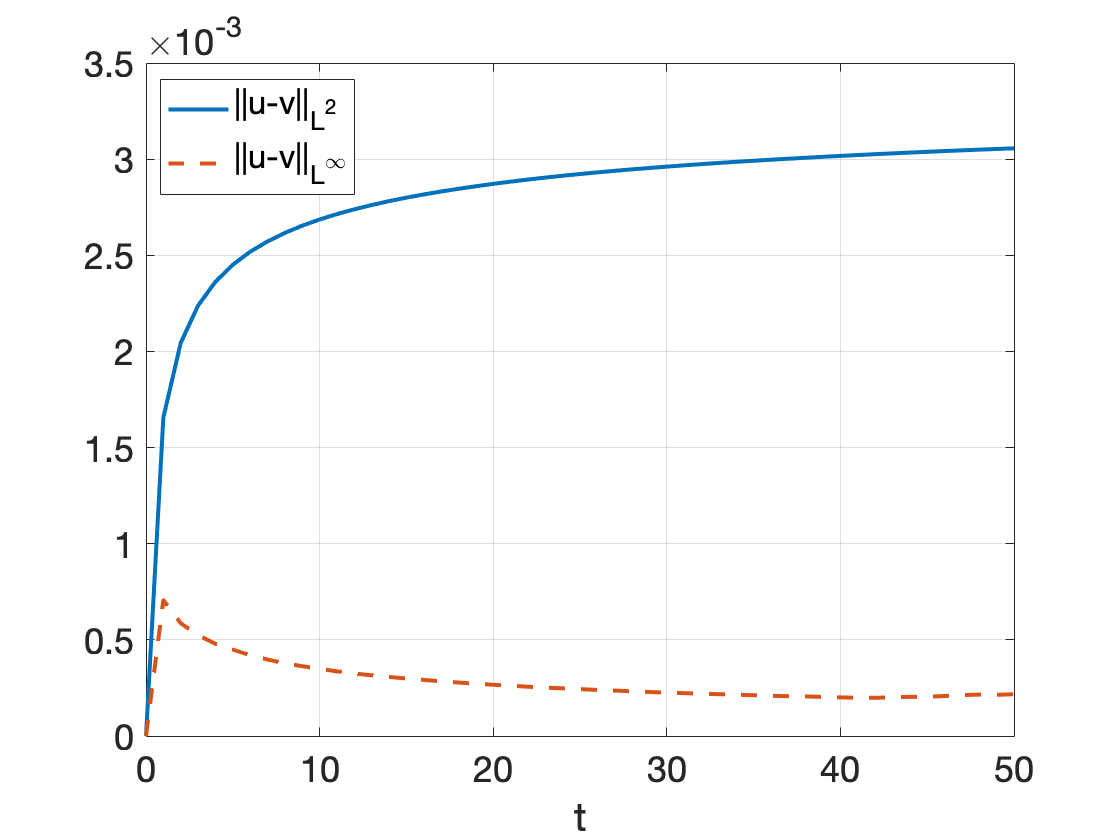}
\caption{\footnotesize  Time evolution for $u_0=v_0 = Ae^{-x^2}$, $A=3$. 
Top Row: $\alpha=1$.  
Middle Row: $\alpha=3$. 
Bottom row: the difference between $u$ and $v$ for $\alpha=1$ (left) and $\alpha=3$ (right).}
\label{F:profile EPos}
\end{figure}

For a comparison with the integer powers ($\alpha=1, 3$),
we show the time evolution of the same Gaussian data in Figure \ref{F:profile EPos}, taking $A=3$ (this amplitude is taken to fit the solution in approximately the same computations range). 
The left plots in both rows of Figure \ref{F:profile EPos} show the solution profiles for the gKdV equation $u(x,t)$ (solid lines) and GKdV equation $v(x,t)$ (dash lines) at times $t=0$ (solid yellow for $u$ and dotted black for $v$) 
and $t=25$ (solid blue for $u$ and dotted red for $v$), for $\alpha=1$ in the top row and $\alpha=3$ in the bottom row. 
The second column shows the solution profiles at time $t=50$. There, we also plotted the rescaled soliton solution $Q_c$ for the corresponding heights $c$. One can see that main lumps in both $u(x,t)$ and $v(x,t)$ fit quite well  the rescaled profiles $Q_c$. 
This indicates that both evolutions (gKdV and GKdV) resolve into the solitons and radiation: the GKdV solitons are shorter (and thus, slower) than the gKdV ones. In our numerical simulations, we tracked the solutions even longer, till $t=150$, observing moving to the right solitons, which completely separate from its radiation moving to the left (and in some cases from each other, where multiple solitons were formed). 
The $L^\infty_x$ and $L^2_x$ norms differences between the solutions $u(x,t)$ and $v(x,t)$ are plotted in the bottom row of Figure \ref{F:profile EPos At}. The difference in the $L^\infty_x$ norm indicates the difference in the height of the first (main) solitons between the two models and the difference in the $L^2_x$ norm is larger as it also incorporates the difference in the radiation (however, note the $10^{-3}$ scale in the right plot for $\alpha=3$). 

\begin{figure}[ht]
\includegraphics[width=0.4\textwidth]{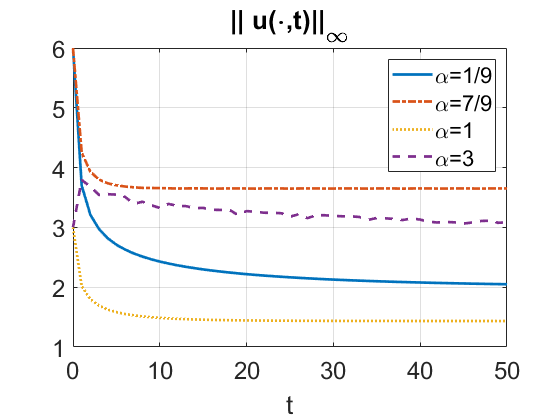}
\includegraphics[width=0.4\textwidth]{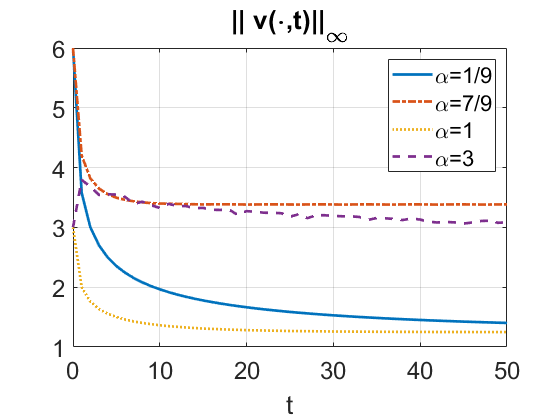}
\caption{\footnotesize  Time dependence of the $L^\infty_x$ norm 
for the solutions to gKdV (left) and GKdV (right) for different values of $\alpha$
from the positive Gaussian data as in Figure \ref{F:profile EPos}.}
\label{F:profile EPos At}
\end{figure}

To further confirm asymptotic convergence to solitons for solutions with positive amplitude Gaussian data in the considered cases of the power $\alpha$, we plot the time dependence of the $L^\infty_x$ norm in Figure \ref{F:profile EPos At}. 
The norms (or height) of all solutions converge to horizontal asymptotes, though in some cases slower than in others, for example, in the case of $\alpha=\frac79$ and $\alpha=1$ the horizontal asymptote forms almost immediately and in other cases it takes longer than $t=50$ time to level. This indicates the solutions asymptotically approach some stable state.   
Moreover, by comparing the two graphs, one can note that $u(x,t)$ always forms solitons of larger amplitude. 


Next, we show the evolution for the negative amplitude Gaussian initial data in all four cases of $\alpha = \frac19, \frac79, 1, 3$. 
As we have already seen the negative data may lead to very different behavior of solutions to the two equations gKdV  \eqref{gKdV} and GKdV \eqref{GK}, we confirm this here as well, moreover, we observe a different process of soliton generation in the negative data in gKdV (see Figures \ref{F:profile Ngau 19} and \ref{F:profile Ngau 59}).

\begin{figure}[ht]
\includegraphics[width=0.32\textwidth]{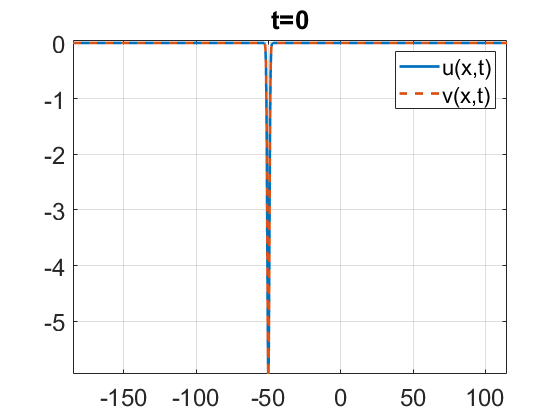}
\includegraphics[width=0.32\textwidth]{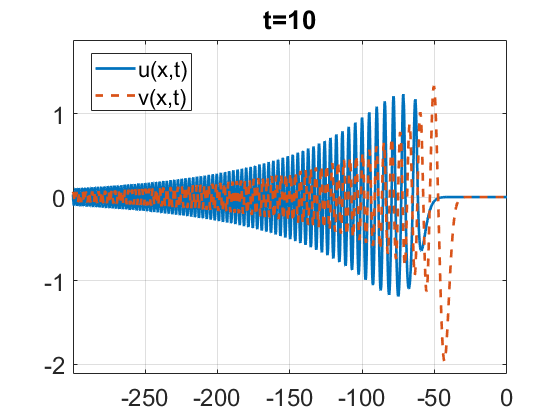}
\includegraphics[width=0.32\textwidth]{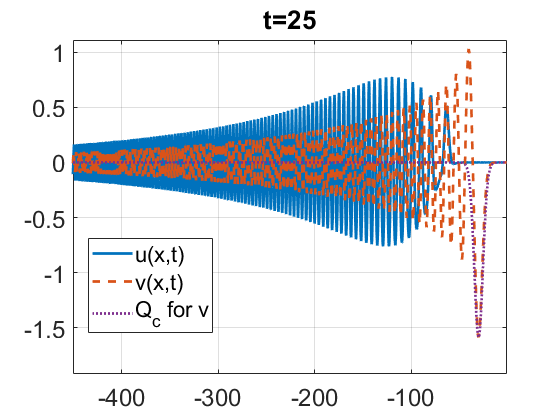}
\includegraphics[width=0.32\textwidth]{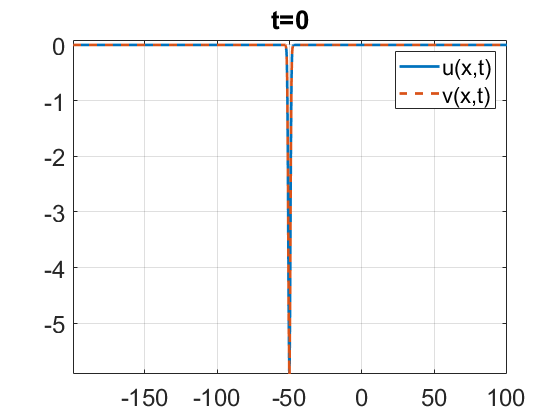}
\includegraphics[width=0.32\textwidth]{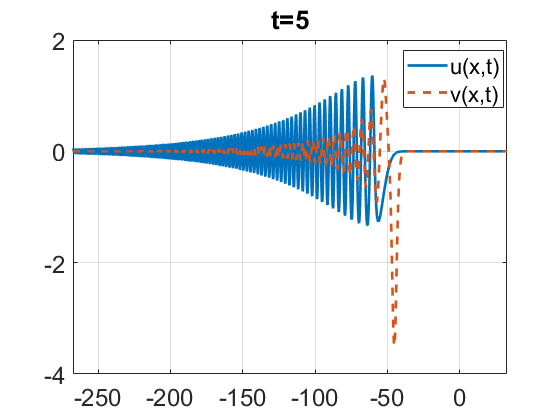}
\includegraphics[width=0.32\textwidth]{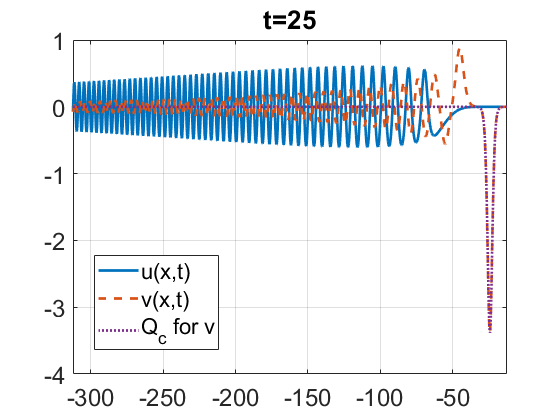}
\includegraphics[width=0.32\textwidth]{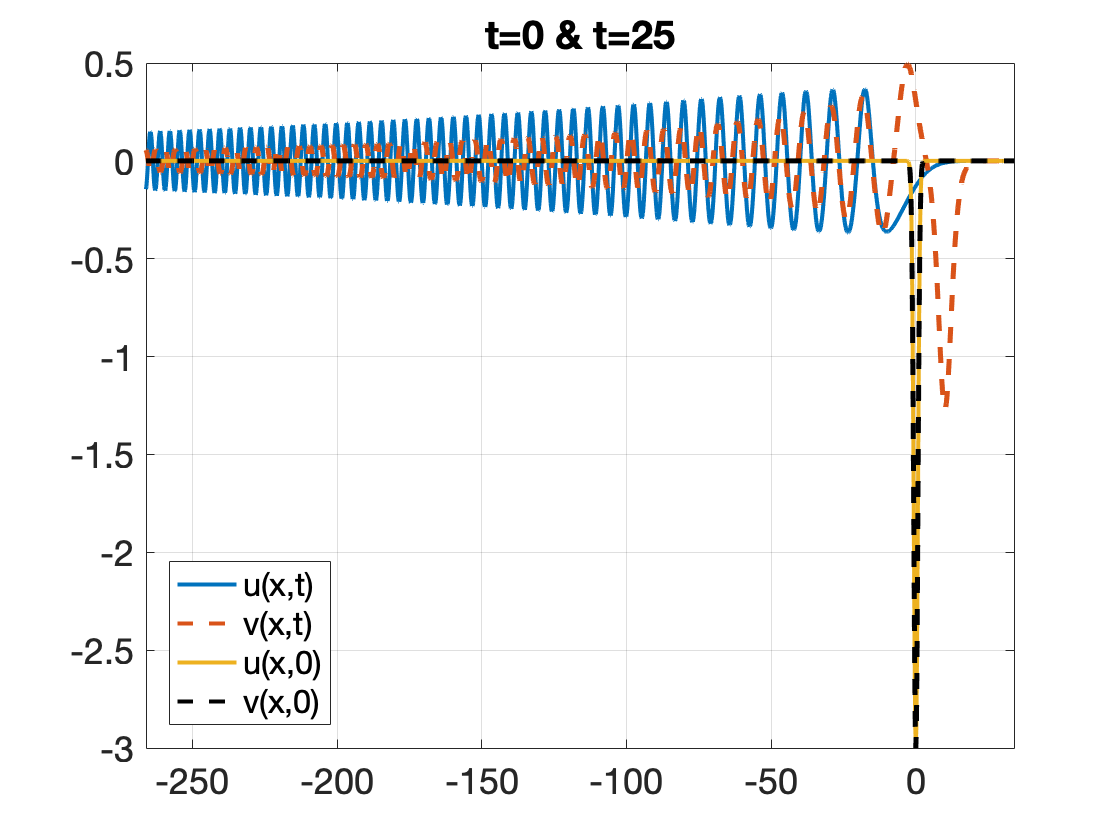}
\includegraphics[width=0.32\textwidth]{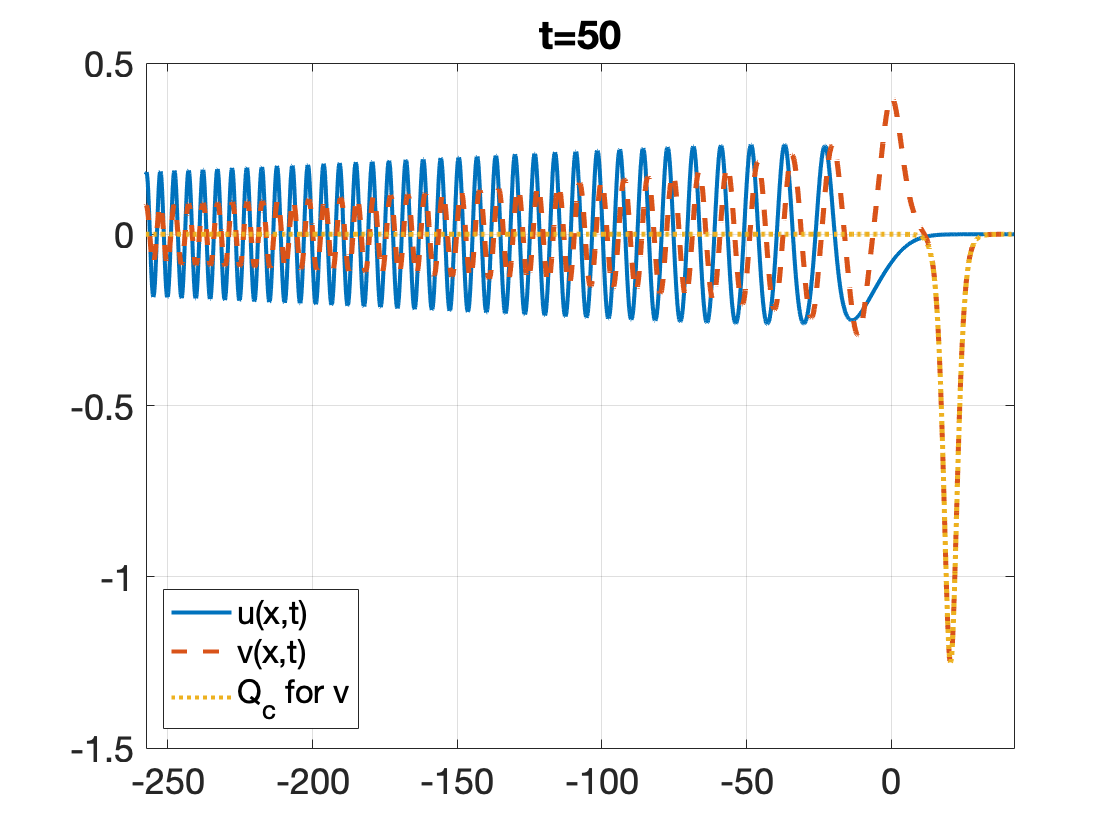}
\includegraphics[width=0.32\textwidth]{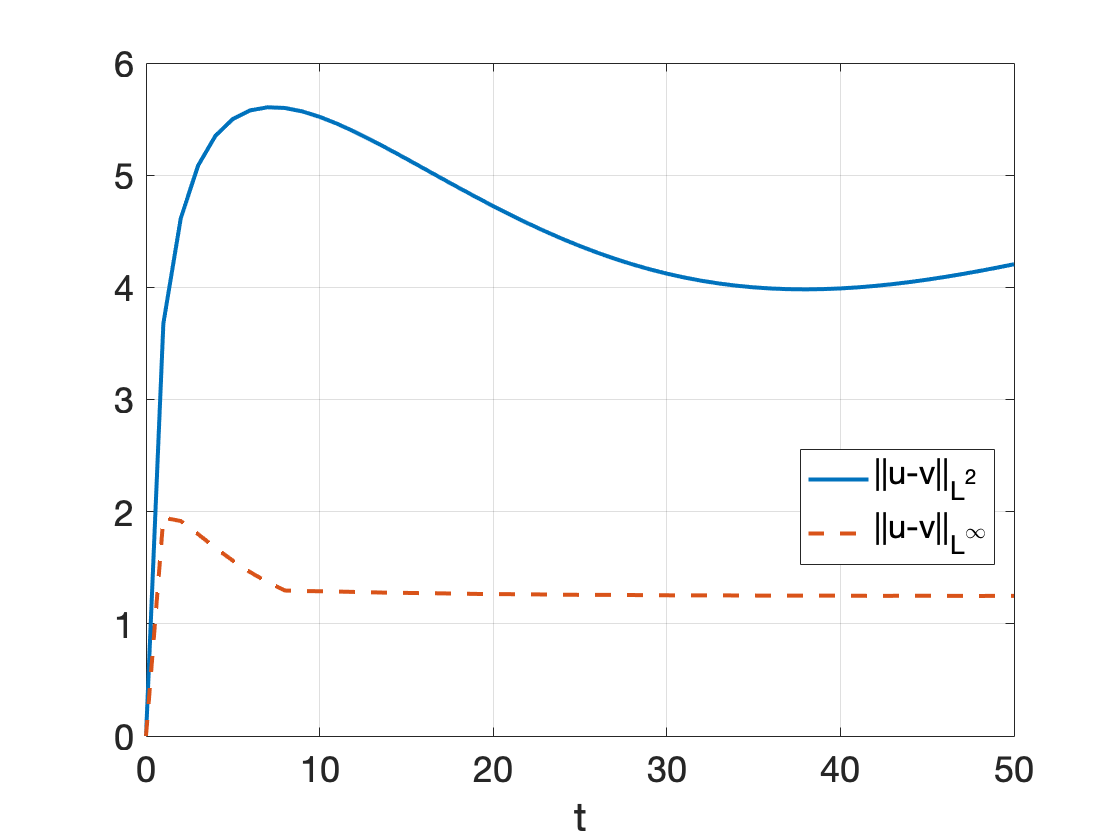}
\includegraphics[width=0.32\textwidth]{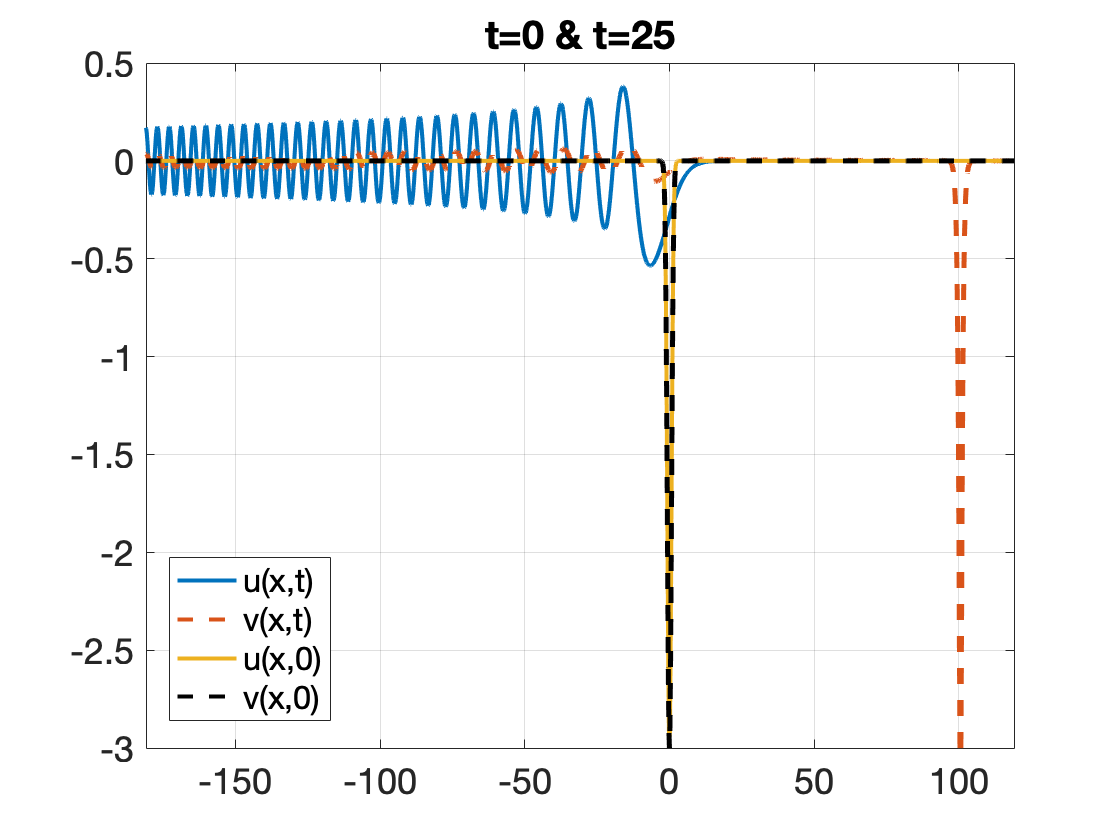}
\includegraphics[width=0.32\textwidth]{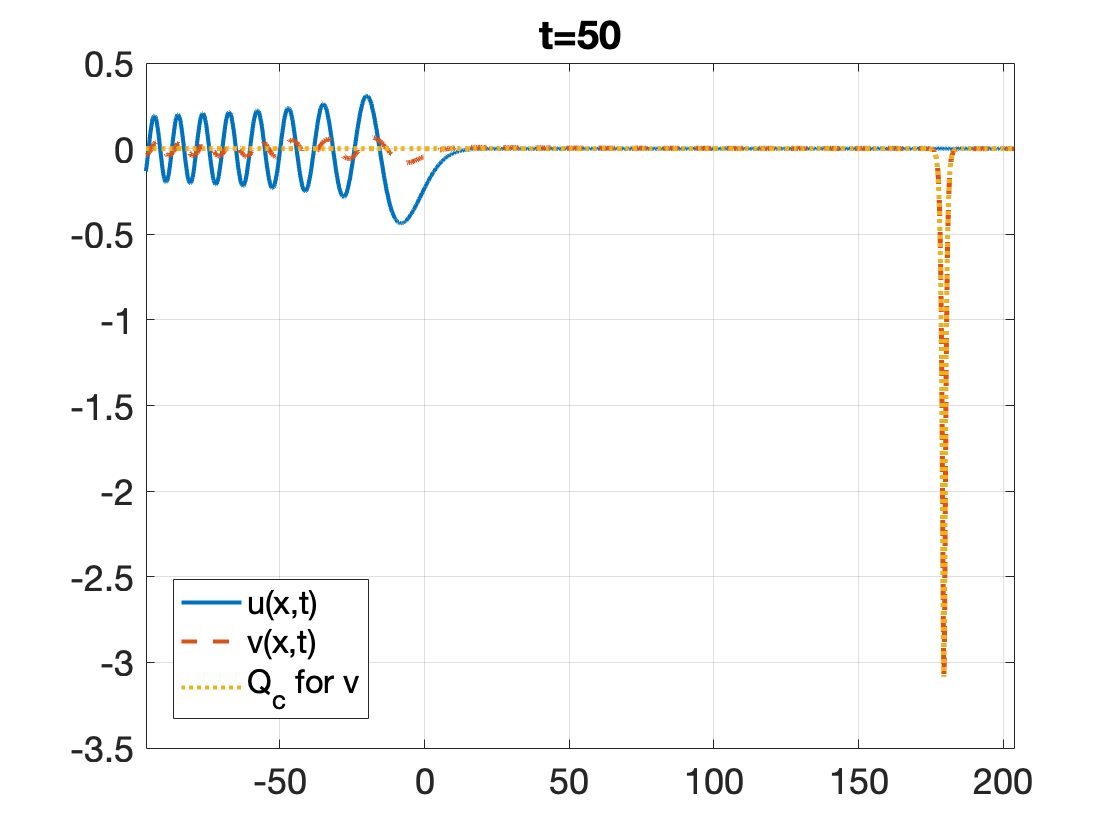}
\includegraphics[width=0.32\textwidth]{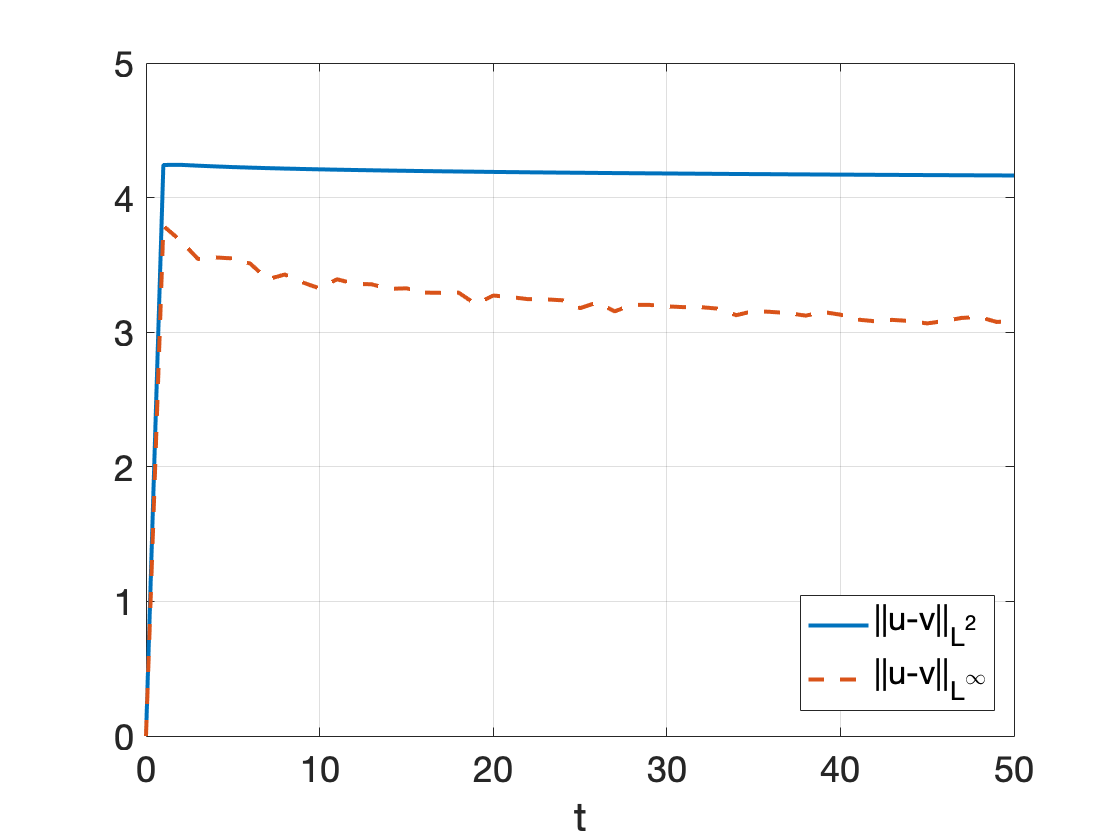}
\caption{\footnotesize  Time evolution for $u_0=v_0=A\,e^{-x^2}$. Top row: $\alpha=\frac{1}{9}$, $A=-6$.  Second row: $\alpha=\frac{7}{9}$, $A=-6$. Third row: $\alpha=1$, $A=-3$. Bottom row: $\alpha=3$, $A=-3$.}
\label{F:profile ENeg}
\end{figure}

In Figure \ref{F:profile ENeg},  snapshots of time evolution of the initial data $u_0=v_0=Ae^{-x^2}$ with $A<0$ is shown. For $\alpha = \frac19, \frac79$ we take $A=-6$ (large enough to allow formation of multiple solitons at least in GKdV) and for $\alpha=1, 3$ we take $A=-3$ to allow tracking of solitons and radiation in approximately the same window range. 

In all cases, the gKdV solutions start radiating to the left (solid blue), whereas the solutions to the GKdV evolve the negative bump into a (negative) soliton propagating to the right, and smaller in amplitude radiation outgoing to the left. Note, that the largest amplitude is negative, hence, if a second soliton forms, then it will have a positive amplitude. For  $\alpha = \frac19, \frac79$, we fit the largest bump with a rescaled (negative) soliton -$Q_c$ at $t=25$ (see the right graphs in the top two rows) and for $\alpha = 1, 3$ at $t=50$ (see the middle plots in the bottom two rows of Figure \ref{F:profile ENeg}). Observe also that the larger the power $\alpha$ is, the faster the separation of the soliton(s) from radiation occurs (e.g., compare the bottom two rows middle plots at $t=50$). The difference between the gKdV and GKdV solutions is more clearly seen in the right plots of the bottom two rows, where we plotted the difference $u(t)-v(t)$ in the $L^2_x$ norm (solid blue) and in the $L^{\infty}_x$ norm (dash red).

Similar to Figure \ref{F:profile EPos At}, we show the time dependence of the $L^\infty_x$ norms for negative data in Figure \ref{F:profile ENeg At}, tracking the quantities $\|u(t)\|_{L^\infty_x}$ (left), $\|v(t)\|_{L^\infty_x}$ (middle) and $\max (v(t) )$ (right) for the solutions shown in Figure \ref{F:profile ENeg}. 
When comparing the left and middle graphs in Figure \ref{F:profile ENeg At}, one notices that in the middle graph the norms $\|v(t)\|_{L^\infty_x}$ converge to horizontal asymptotes, which are the same as in the right graph of Figure \ref{F:profile EPos At}, indicating that in GKdV the formation of solitons is not influenced by the sign of the initial data, and the  solution $v(x,t)$ approaches to either a positive or negative soliton of the same magnitude. More interestingly, by tracking in time the maximum spatial value of $v(x,t)$ as in the right plot of Figure \ref{F:profile ENeg At}, we find that the second positive soliton forms 
because of the convergence to a horizontal asymptote (for $\alpha=\frac19$ a longer computational time is needed; we also tested $\alpha = \frac59$, not included in Figures \ref{F:profile ENeg}, \ref{F:profile ENeg At}, and observed the formation of the second soliton). 
\begin{figure}[ht]
\includegraphics[width=0.32\textwidth]{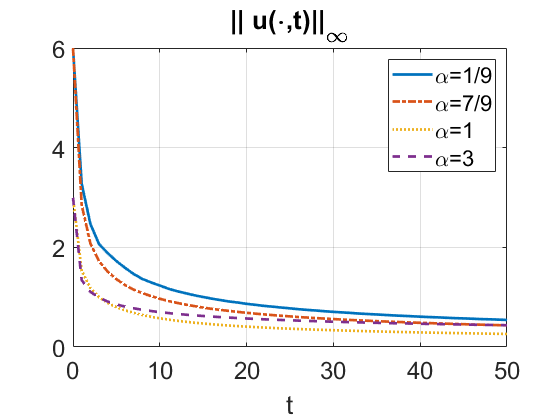}
\includegraphics[width=0.32\textwidth]{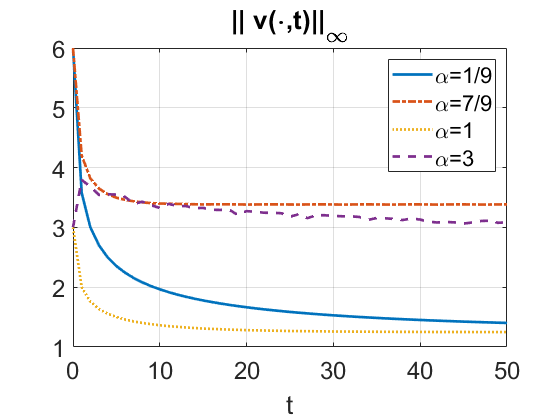}
\includegraphics[width=0.32\textwidth]{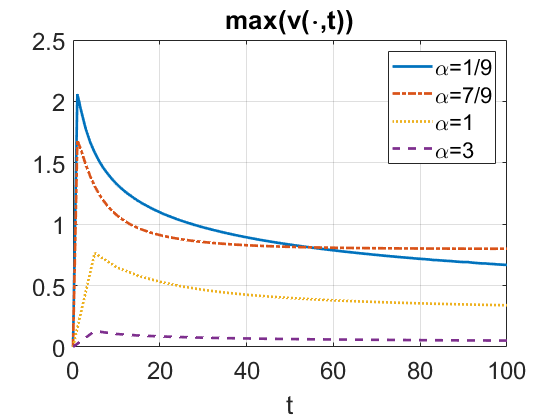}
\caption{\footnotesize  Time dependence of the $L^\infty_x$ norm of solutions to gKdV (left) vs. GKdV (middle), and of $\max (v(\cdot,t))$(right) 
for different $\alpha$ from the negative Gaussian data 
as in Figure \ref{F:profile ENeg}. }
\label{F:profile ENeg At}
\end{figure}

On the other hand, $\|u(t)\|_{L^\infty_x}$ keeps decreasing in time (left plot in Figure \ref{F:profile ENeg At}, indicating that the solution might be dispersing (or scattering), and we investigate this further. For that, take a closer look at the top right plot in Figure \ref{F:profile ENeg} solid blue line. Notice that the blue curve oscillates forming a pulse-like shape for its envelope, which decays to the left, but to the very right of it there is a separate positive bump (about $0.5$ in height), which is not present in the blue curve in the previous (middle) plot (at $t=10$).  We zoom-in and track this part of the solution in a bit more detail, see snapshots of that region for times up to $t=100$ in Figure \ref{F:profile Ngau 19}. 
\begin{figure}[ht]
\includegraphics[width=0.32\textwidth]{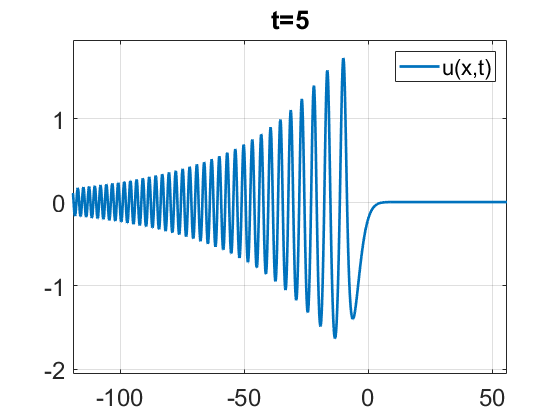}
\includegraphics[width=0.32\textwidth]{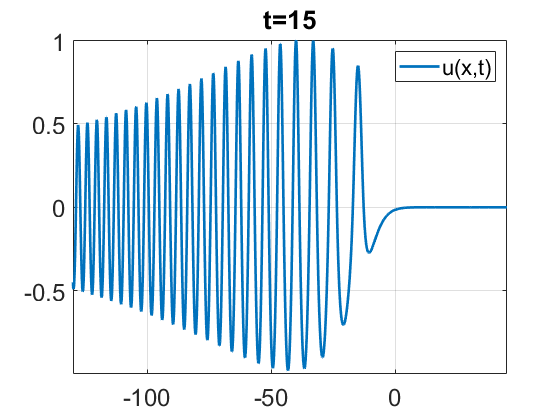}
\includegraphics[width=0.32\textwidth]{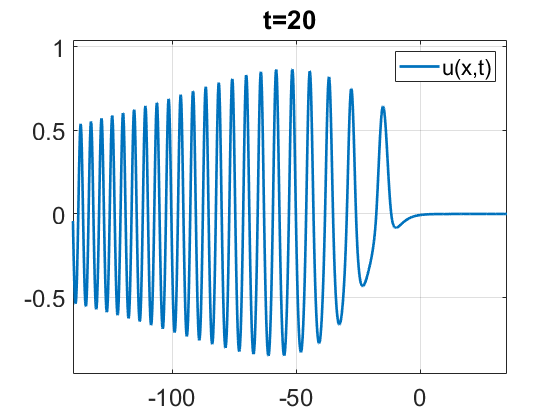}
\includegraphics[width=0.32\textwidth]{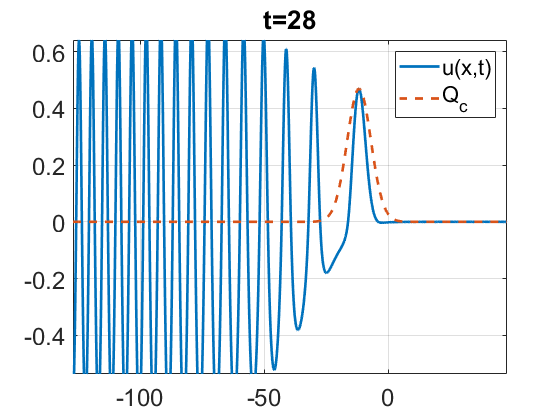}
\includegraphics[width=0.32\textwidth]{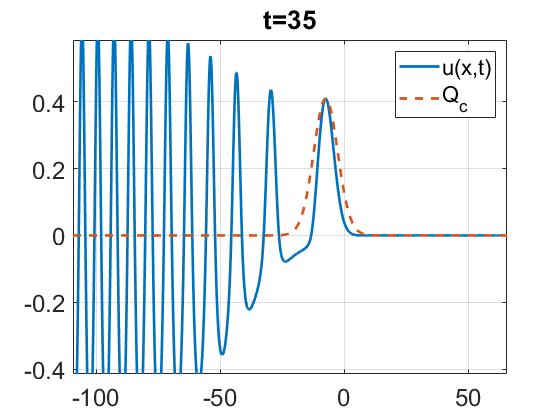}
\includegraphics[width=0.32\textwidth]{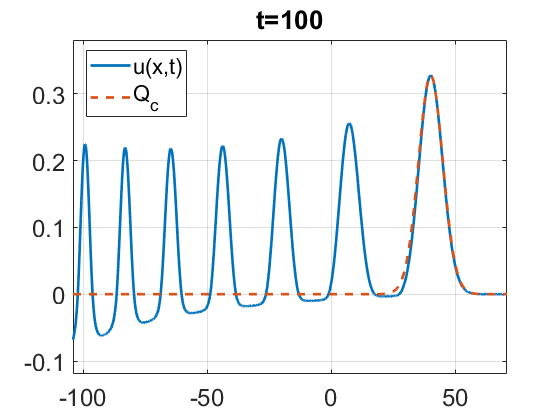}
\caption{\footnotesize  The gKdV time evolution for $u_0= -6\,e^{-x^2}$, $\alpha=\frac19$, till $t = 100$. }
\label{F:profile Ngau 19}
\end{figure}

Tracking the time evolution for $u_0 = -6 \,e^{-x^2}$ in Figure \ref{F:profile Ngau 19}, one can observe that the first negative bump in the radiation decreases in its magnitude (becomes smaller, see the top row) and then eventually disappears (it is almost absent at $t=20$). Then the next positive bump starts to separate from the pulse-like radiation, and because it is positive (in gKdV model), it starts forming a soliton, or asymptotically approaches a rescaled version of it; that soliton can be seen in the bottom row, where we fit it with the rescaled and shifted soliton $Q_c$ (dash red). 
Furthermore, note that after the first soliton starts forming and separating, the negative bump between it and the next positive bump is decreasing in magnitude (see the left and middle plots in the bottom row of Figure \ref{F:profile Ngau 19}) and eventually disappears, allowing the next positive bump form into a soliton (this can be seen in the right plot at $t=100$). In fact, this continues, and more and more smaller solitons will emerge from the radiative pulse-like part of the solution.  

A natural question then is, would the same phenomena happen for other powers $\alpha$? (at least when $\alpha$ is odd?) 
We tried other powers for this data and simulated the solutions till $t=500$. While we cannot see this phenomenon for $\alpha=\frac{7}{9}$ (need much longer computational time), we observe a similar emergence of solitons for $\alpha=\frac{1}{3}$ and $\alpha=\frac{5}{9}$. 
\begin{figure}[ht]
\includegraphics[width=0.32\textwidth]{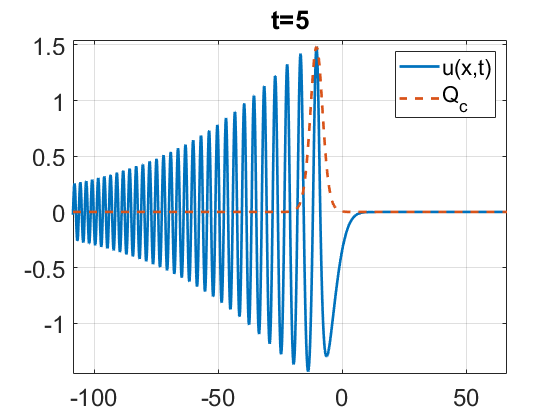}
\includegraphics[width=0.32\textwidth]{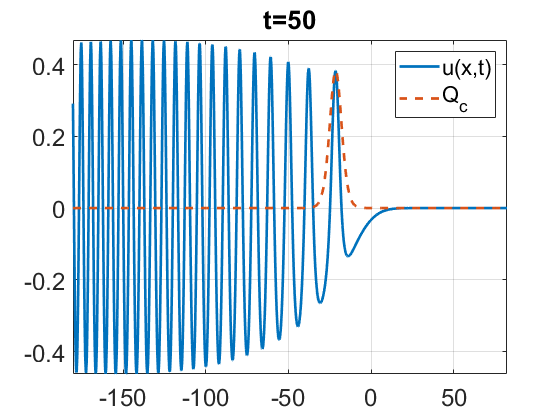}
\includegraphics[width=0.32\textwidth]{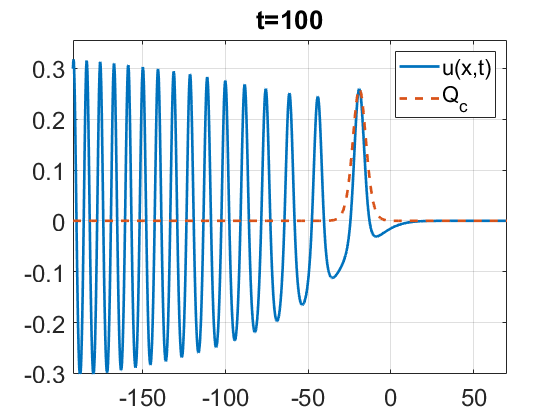}
\includegraphics[width=0.32\textwidth]{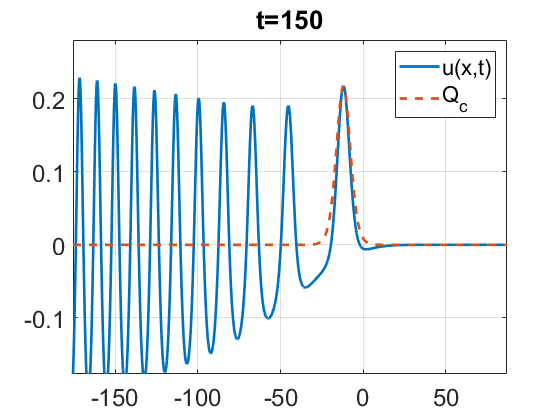}
\includegraphics[width=0.32\textwidth]{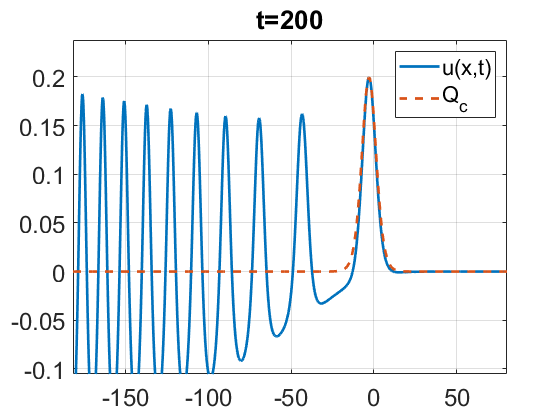}
\includegraphics[width=0.32\textwidth]{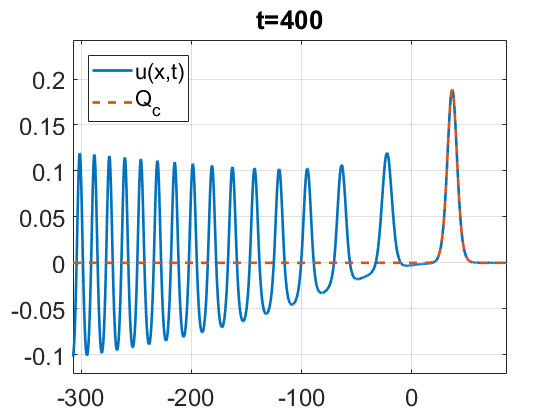}
\caption{\footnotesize  The gKdV time evolution for $u_0= -6\,e^{-x^2}$, $\alpha=\frac59$, till $t=400$.}
\label{F:profile Ngau 59}
\end{figure}
Figure \ref{F:profile Ngau 59} shows the time evolution for $\alpha=\frac{5}{9}$ case. Besides the formation of solitons, we note that this case takes a significantly longer time to form the first soliton. Thus, 
the larger power $\alpha$  is, the longer it takes to form such solitons, which explains the reason why we could not observe the emergence of a soliton in the case $\alpha=\frac{7}{9}$ or $1$ (and hence, larger powers).

\subsubsection{Super-Gaussian initial data}
Now, we examine the fractional $\alpha$ values with the super-Gaussian initial data of the form
\begin{equation}\label{E:super-G1}
u_0=A\,e^{-x^4},
\end{equation}
where the decay is faster than that of a soliton or Gaussian data. \smallskip

In Figure \ref{F:profile SGPos}, snapshots of the time evolution for positive amplitude, $A=6$, are shown for both cases of $\alpha = \frac19$ (top row) and $\alpha=\frac79$ (bottom row). 
\begin{figure}[ht]
\includegraphics[width=0.32\textwidth]{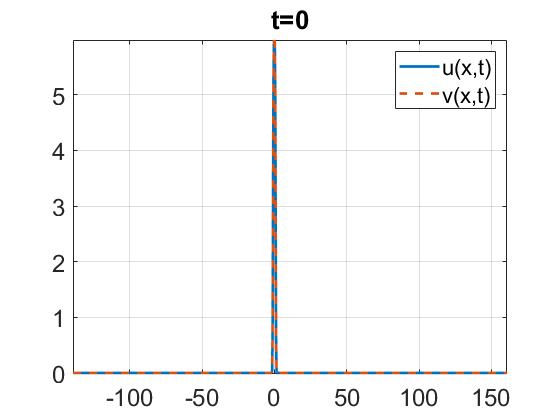}
\includegraphics[width=0.32\textwidth]{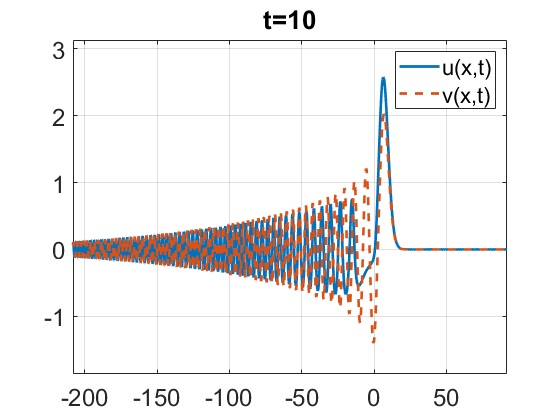}
\includegraphics[width=0.32\textwidth]{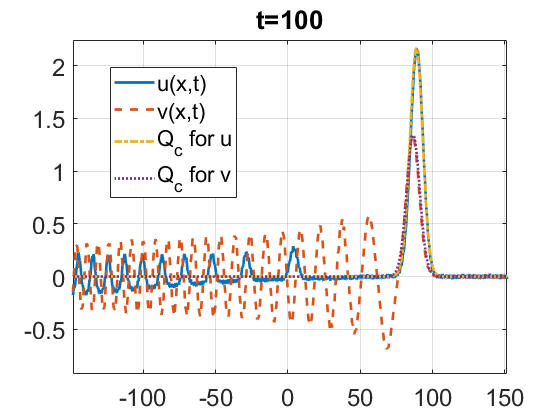}
\includegraphics[width=0.32\textwidth]{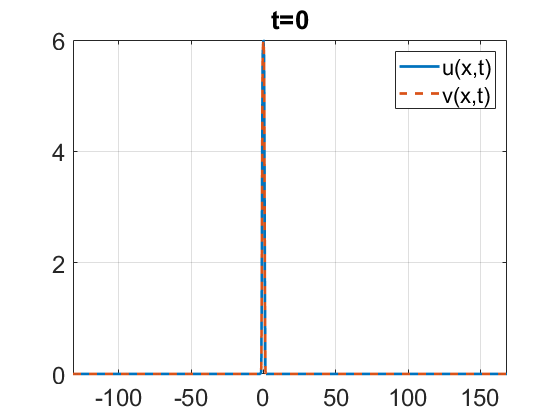}
\includegraphics[width=0.32\textwidth]{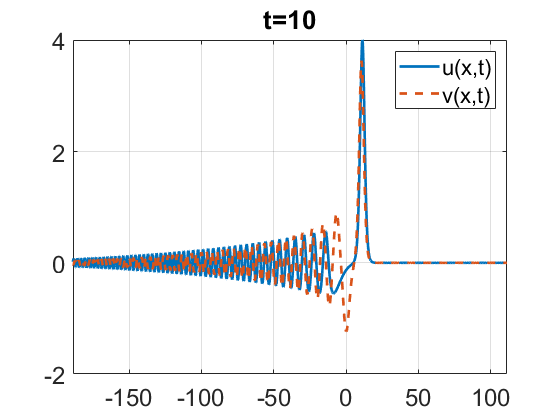}
\includegraphics[width=0.32\textwidth]{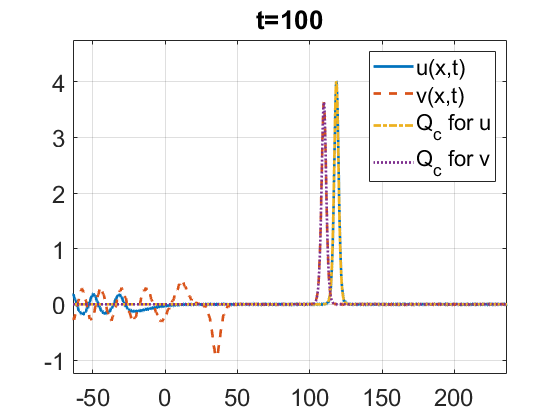}
\caption{\footnotesize  Time evolution for $u_0=v_0=A\,e^{-x^4}$, $A=6$; 
$\alpha=\tfrac19$ (top row), $\alpha=\tfrac79$ (bottom row).}
\label{F:profile SGPos}
\end{figure}

Solitons are formed in solutions to both equations,  $u(x,t)$ and $v(x,t)$. 
Observe that with a positive initial condition, the soliton, formed in $u(x,t)$, is higher, and therefore, moves to the right faster than the one formed in $v(x,t)$. A separation between $u(x,t)$ and $v(x,t)$ is more evident for larger powers of $\alpha$ (e.g., in the bottom row of Figure \ref{F:profile SGPos}). Notice also a train of solitons that forms in the solution $u(x,t)$ (solid blue in the top right plot). Multiple solitons, with alternating signs, form in the solution $v(x,t)$ (dash red in the bottom row). In the right column, when $t=100$, we fit the largest bumps with the rescaled solitons $Q_c$ to check the profiles. 
\smallskip

\begin{figure}[ht]
\includegraphics[width=0.4\textwidth]{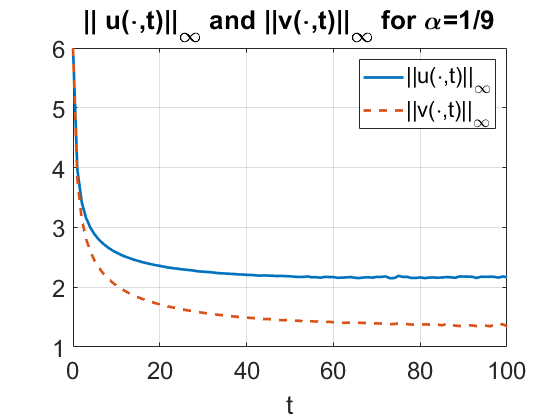}
\includegraphics[width=0.4\textwidth]{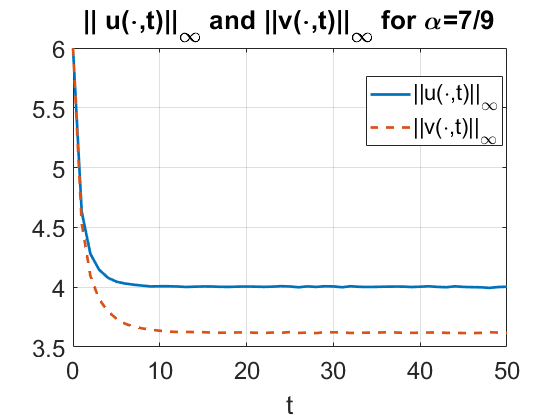}
\caption{\footnotesize  Time dependence of the $L^\infty_x$ norm of solutions to gKdV vs. GKdV for different $\alpha$ for the positive super-Gaussian data as in Figure \ref{F:profile SGPos}.}
\label{F:profile SGPos At}
\end{figure}

To confirm the convergence to solitons, in Figure \ref{F:profile SGPos At} we show the time dependence of the $L^\infty_x$ norm of $u(x,t)$ and $v(x,t)$, which is the height of the first, or the largest, bump. As the sup norm, or the height, of each solution converges to a horizontal asymptote, it indicates the asymptotic convergence to the soliton. Similar to previous cases, the gKdV solutions form higher solitons, and thus faster moving, compared to the GKdV solutions. \\


\newpage

Next, we show the case of the negative super-Gaussian data $u_0=A\, e^{-x^4}$ with $A<0$. Figure \ref{F:profile SGNeg} shows snapshots of the time evolution with $A=-6$ for $\alpha= \frac{1}{9}$ (top row) and $\frac{7}{9}$ (bottom row). 
\begin{figure}[ht]
\includegraphics[width=0.32\textwidth]{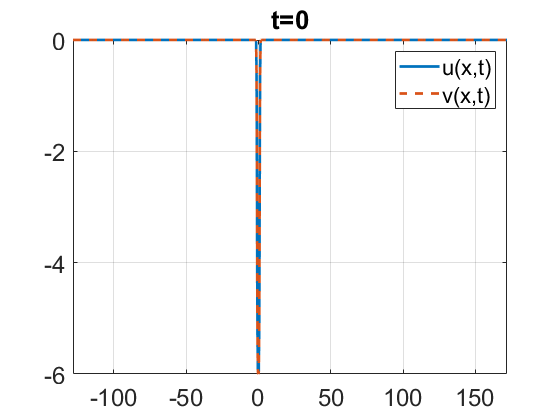}
\includegraphics[width=0.32\textwidth]{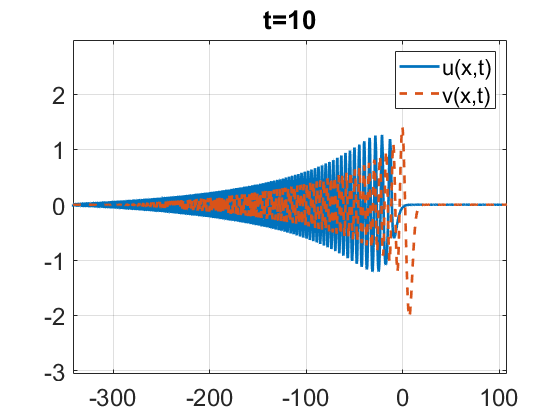}
\includegraphics[width=0.32\textwidth]{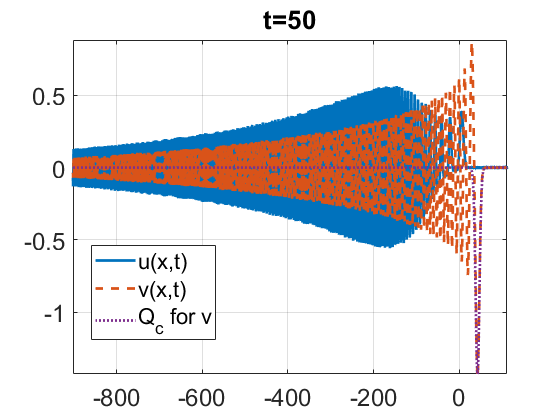}
\includegraphics[width=0.32\textwidth]{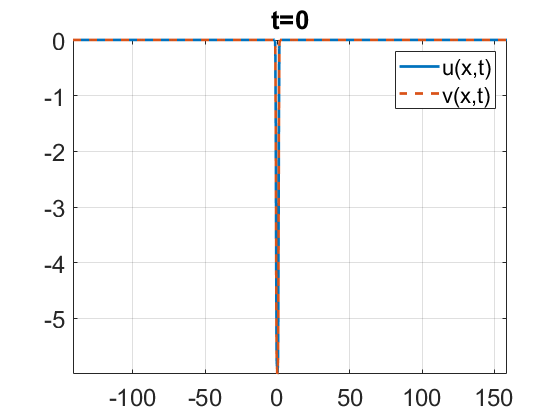}
\includegraphics[width=0.32\textwidth]{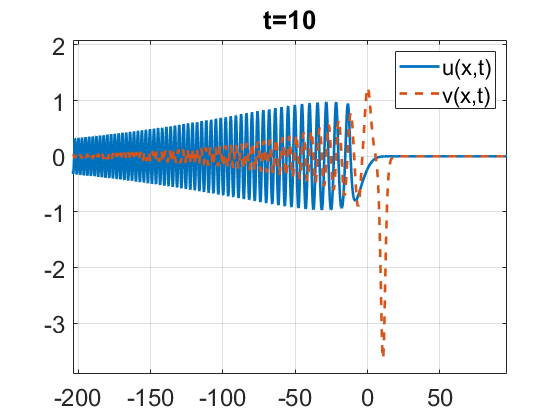}
\includegraphics[width=0.32\textwidth]{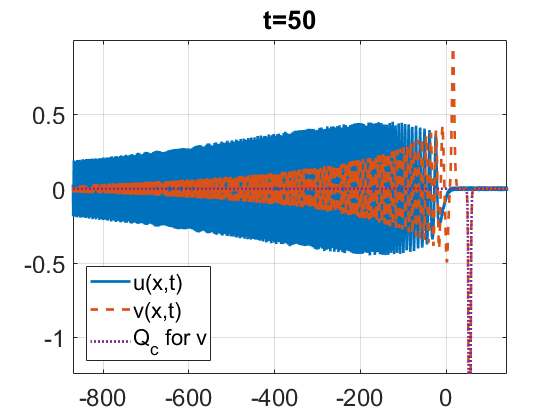}
\caption{\footnotesize  Time evolution for $u_0=v_0=Ae^{-x^4}$, $A=-6$; $\alpha=\frac19$ (top row), $\alpha=\frac79$ (bottom row).}
\label{F:profile SGNeg}
\end{figure}

The GKdV solution $v(x,t)$ forms several solitons (red dash line), symmetric to the positive case, while the gKdV solution $u(x,t)$ starts dispersing to the left forming a decaying oscillatory tail. A very close examination of the top right plot gives an indication of an emerging soliton, and possibly not just one (solid blue line). 
In the bottom middle and right subplots of Figure \ref{F:profile SGNeg}, one can notice that after the main negative soliton is formed (with the largest amplitude), a second positive soliton is also visible, and even a third (negative) one is forming as well. Hence, again, in the GKdV model the solitons from negative data alternate the sign. In the right column, the fitting of the largest bump is with the rescaled and shifted soliton, confirming the shape.    \\

\begin{figure}[ht]
\includegraphics[width=0.4\textwidth]{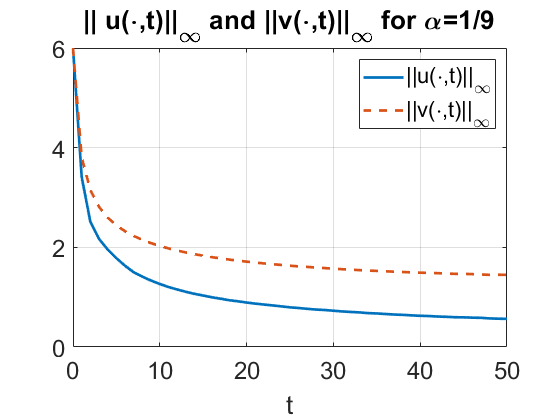}
\includegraphics[width=0.4\textwidth]{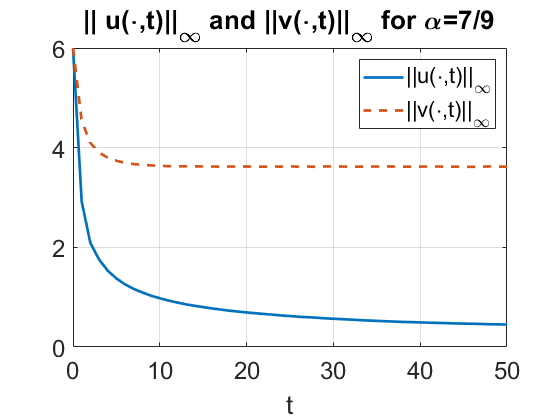}
\caption{\footnotesize  Time dependence of the $L^\infty_x$ norm of solutions to gKdV vs. GKdV for different $\alpha$ from the negative super-Gaussian data 
as in Figure \ref{F:profile SGNeg}. }
\label{F:profile SGNeg At}
\end{figure}

We show the time dependence of the sup norm for the gKdV and GKdV solutions for different powers of $\alpha$ in Figure \ref{F:profile SGNeg At}. 
We note that $\|u(t)\|_{L^\infty_x}$ keeps decreasing in both cases of $\alpha$, which partially justifies our observation that $u(x,t)$ goes into the oscillatory radiation in the observable time. Eventually, it will level to a horizontal asymptote, since the solitons will start emerging from the radiation, however, note that the height of the emerging solitons is smaller than the height of the oscillatory radiation part (during the simulation time). On the other hand, we can see $\|v(t)\|_{L^\infty_x}$ converges to a horizontal asymptote, indicating the formation of the soliton. (We also checked that $\max(v(t))$ converges to a horizontal asymptote for $\alpha=\frac{1}{9}$ and $\frac{7}{9}$, indicating the formation of a second positive soliton, but we omit the plot here.)

\subsubsection{Initial data of polynomial decay}
We next study the time evolution of the polynomially decaying data with as slow decay as $1/|x|$. Note that such data would not satisfy the Fadeev condition, and thus, the IST (inverse scattering transform) is not applicable to investigate the soliton resolution even in the standard KdV case ($\alpha=1$ in \eqref{gKdV}). Nevertheless, such data is in $L^2$, and thus, the dispersive theory (as described in the introduction) guarantees the local and global well-posedness in that case. Our main theorem gives the local well-posedness for such data (for any $\alpha>0$), and even for the data of slower decay in some cases. 
\medskip

We consider 
\begin{equation}\label{E:A-poly}
u_0(x)=\frac{A}{\sqrt{1+x^2}}, \quad A \in \mathbb R \setminus \{0\}.
\end{equation}
(We could consider higher polynomial decay as well, but this is the slowest rate of decay we are able to handle by numerics by using rational basis functions, see Appendix or \cite{RWY2021}, \cite{RRY2021}.)
We take the computational domain $[-L,L]$ large enough ($L=1600\pi$) to approximate the actual solution on $\mathbb{R}$. We point out that taking the larger values of $L$ leads to similar numerical results. Therefore, even though the domain truncation strategy may lead to larger error in the original model, the numerical results we show are
quite close to the true behavior of solutions. 
\medskip

We start with the positive amplitude ($A=1$ or $1.5$). From top to bottom, Figure \ref{F:profile PRLforPOS} shows the time evolution for $\alpha=\frac{1}{9}$, 
$\frac{7}{9}$, $1$, and $3$, for times $t=50, 100, 200$. 

\begin{figure}[ht]
\includegraphics[width=0.32\textwidth]{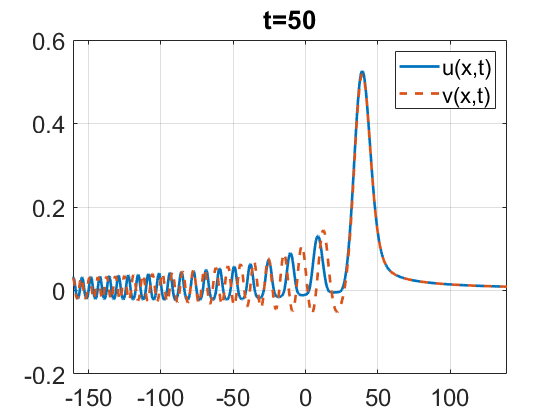}
\includegraphics[width=0.32\textwidth]{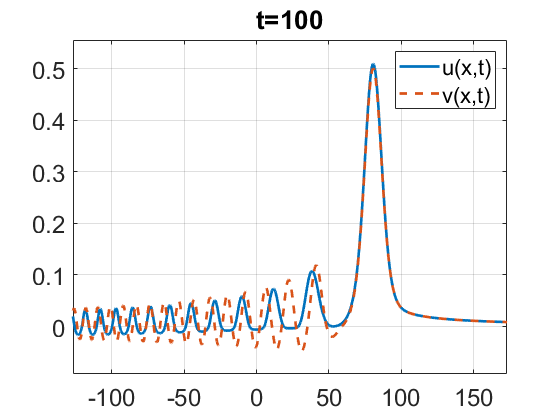}
\includegraphics[width=0.32\textwidth]{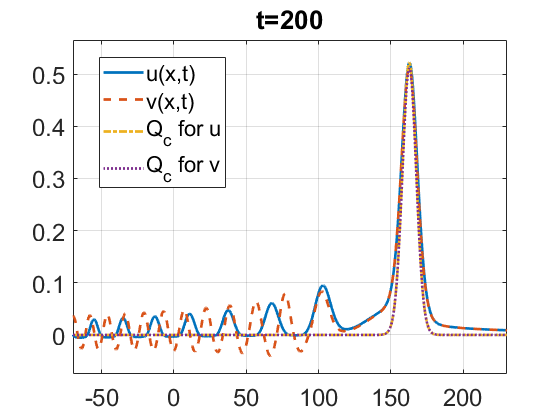}
\includegraphics[width=0.32\textwidth]{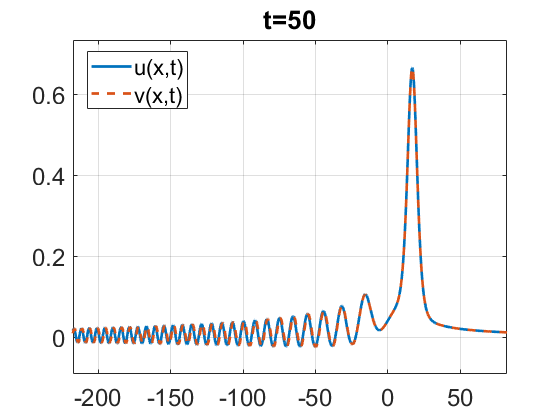}
\includegraphics[width=0.32\textwidth]{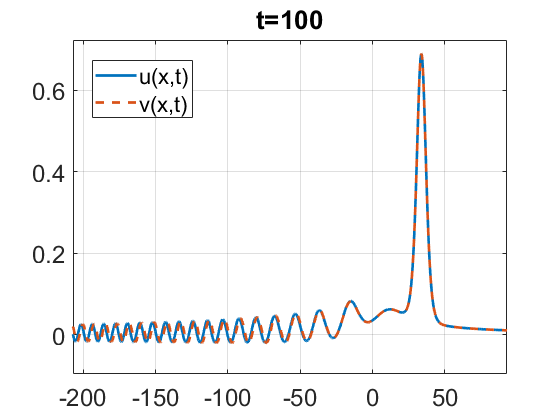}
\includegraphics[width=0.32\textwidth]{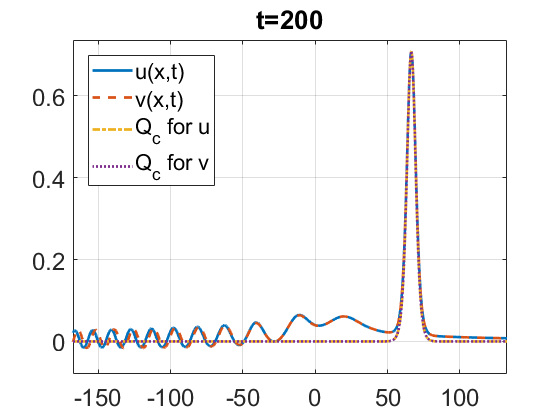}
\includegraphics[width=0.32\textwidth]{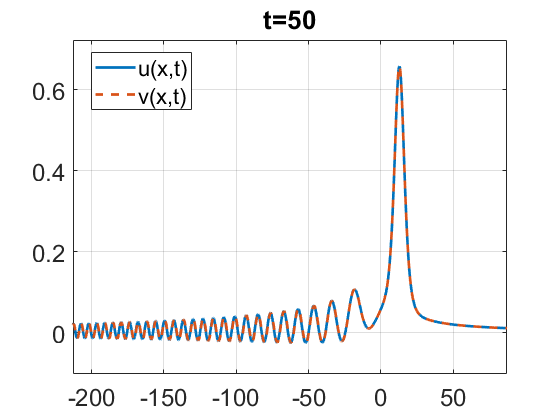}
\includegraphics[width=0.32\textwidth]{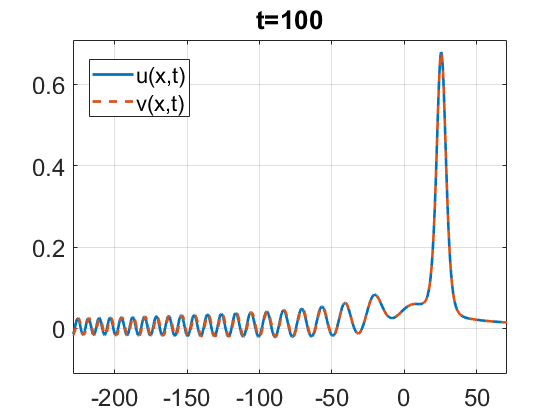}
\includegraphics[width=0.32\textwidth]{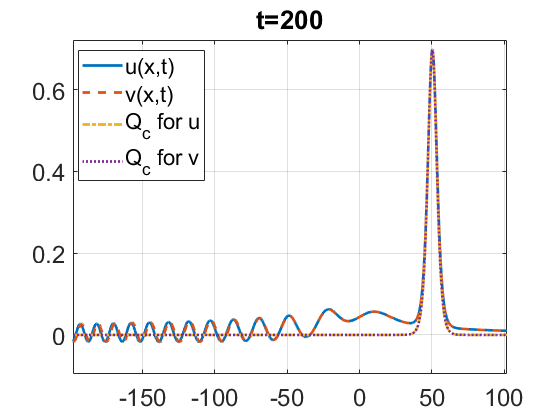}
\includegraphics[width=0.32\textwidth]{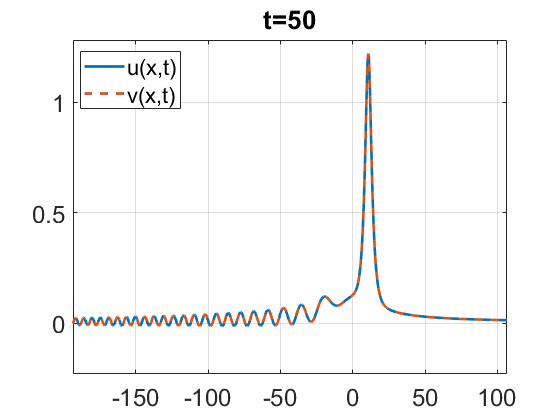}
\includegraphics[width=0.32\textwidth]{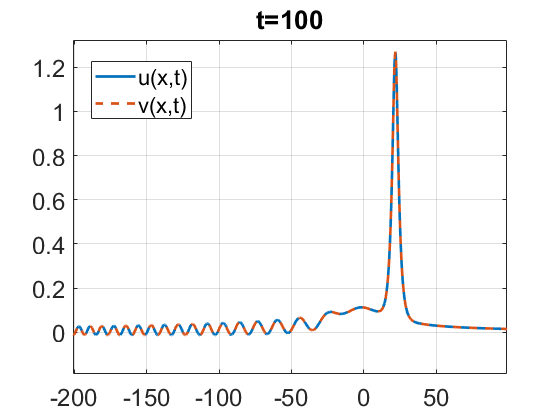}
\includegraphics[width=0.32\textwidth]{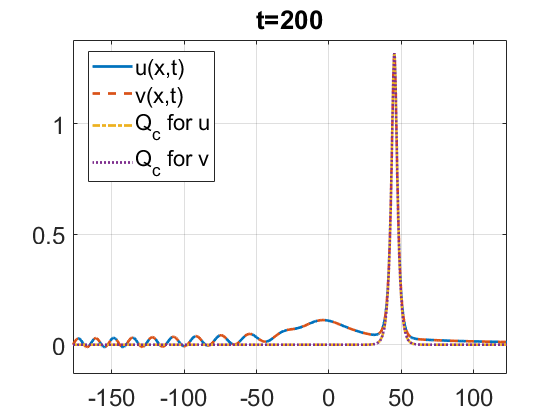}
\includegraphics[width=0.35\textwidth, height=3.5cm]{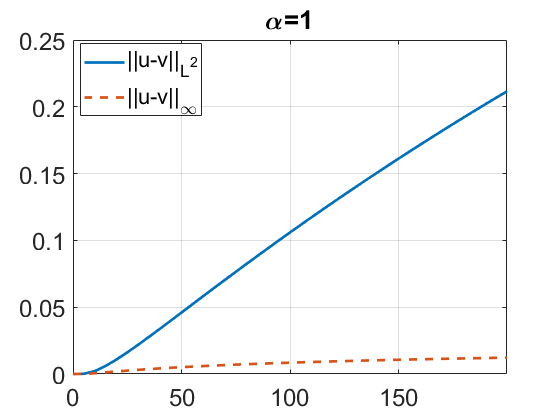}
\includegraphics[width=0.35\textwidth, height=3.5cm]{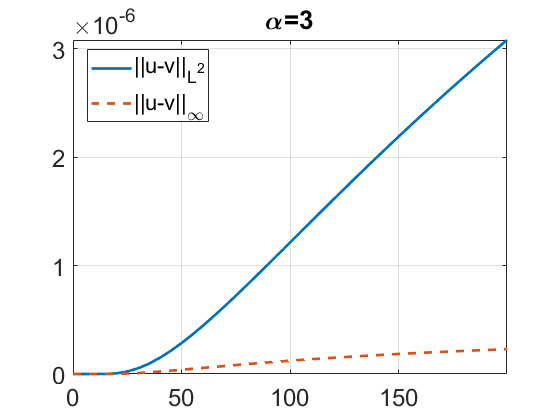}
\caption{\footnotesize  Time evolution of  $u_0=v_0=\frac{A}{\sqrt{1+x^2}}$ at $t=50, 100, 200$. 
Top row: $\alpha=\frac{1}{9}$, $A=1$; second row: $\alpha=\frac{7}{9}$, $A=1$; third row: $\alpha=1$, $A=1$; fourth row: $\alpha=3$, $A=1.5$.
Bottom row: the $L^\infty_x$ and $L^2_x$ norms difference of $u$ and $v$ for $\alpha=1$ and $\alpha=3$.}
\label{F:profile PRLforPOS}
\end{figure}

In both models solitons are forming, traveling to the right together with the radiation outgoing to the left for all $\alpha>0$, similar to  the fast (exponentially) decaying initial data. 
Note that a train of solitons is being formed in the top row ($\alpha=\frac19$) by the gKdV propagation (solid blue line in the middle and right plots), which is also the case for the solutions with Gaussian and super-Gaussian data. 
Furthermore, the slower decay of the initial condition is, the slower is the asymptotic convergence of the solutions to a soliton; the largest bump is starting to fit the rescaled soliton (see the fitting at $t=200$ in the right column of Figure \ref{F:profile PRLforPOS}), though there is still much more mass on the right of the soliton to transfer into the radiation on the left before getting close to the exponential decay (as in $Q$), this can be seen in the decay of the tail on the right of the largest bump when fitting to the rescaled $Q_c$ in the right column plots. 
Another point worth mentioning is that the radiation part seems to have higher positive values than negative in both $u$ and $v$ solutions. As the radiation initially starts escaping to the left, it does so via some positive oscillations (see, for instance, the middle and right plots in the second, third and fourth rows of Figure \ref{F:profile PRLforPOS}), after which it will then oscillate around the $x$-axis (still with more positive part).  
The bottom row of Figure \ref{F:profile PRLforPOS} shows the difference of solutions in the cases $\alpha=1, 3$ in the $L^\infty_x$ and $L^2_x$ norms, note the scale in the right plot of the order of $10^{-6}$, thus, the difference in that case is insignificant.  

\begin{figure}[ht]
\includegraphics[width=0.4\textwidth]{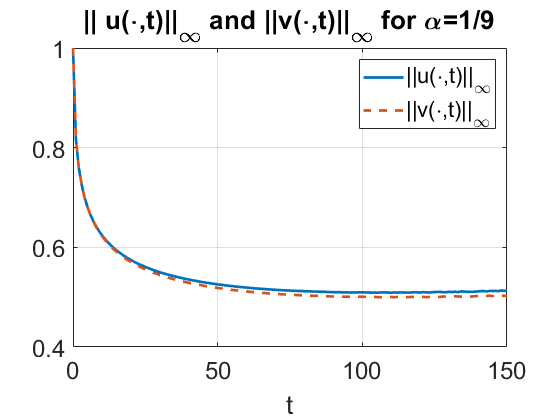}
\includegraphics[width=0.4\textwidth]{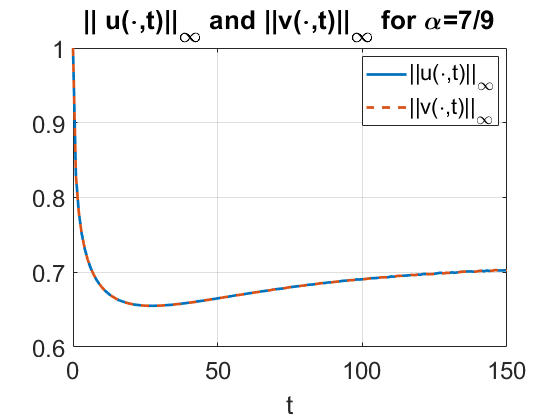}
\caption{\footnotesize Time dependence of the $L^\infty_x$ norm of solutions to gKdV vs. GKdV for different $\alpha$ from the initial data $u_0=v_0=1 / \sqrt{1+x^2} $. }
\label{F:profile x1 At}
\end{figure}
Furthermore, 
in Figure \ref{F:profile x1 At}, a comparison of the $L^\infty_x$ norms for both $u$ and $v$ solutions is given for $\alpha=\frac19$ (left) and $\alpha=\frac79$ (right), showing almost no difference in the height of the (largest) solitons.  \\

We now check the case of negative initial data ($A<0$). We take $A=-2$ in \eqref{E:A-poly}
and show the corresponding solutions in Figure \ref{F:profile PRLforNEG}.
For $\alpha=\frac19$ we see a pulse-like radiation, which has a similar decreasing first negative bump (see the top row), we investigate it further and run the simulations till $t=500$, see Figure \ref{F:profile PRLforNEG19}, where the first positive soliton has almost started forming, and more positive solitons are coming. A similar behavior should be happening in other cases of $\alpha$, but to observe that it requires a much longer computational time. The GKdV solutions (dash red in Figure \ref{F:profile PRLforNEG}) form a negative soliton and 
a second negative soliton is starting to form as well, see the middle and last plots in each row. In the right column we do fittings for the rescaled soliton shape, which gives a very good matching. \\

\begin{figure}[ht]
\includegraphics[width=0.32\textwidth]{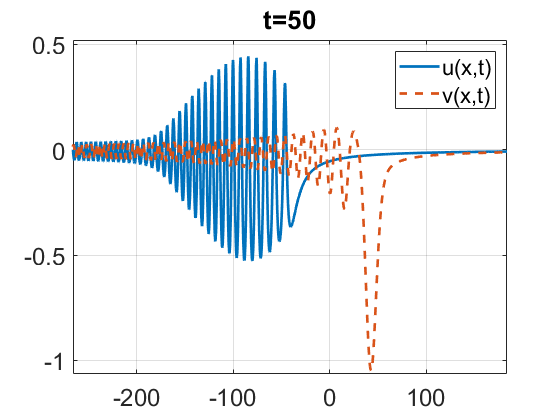}
\includegraphics[width=0.32\textwidth]{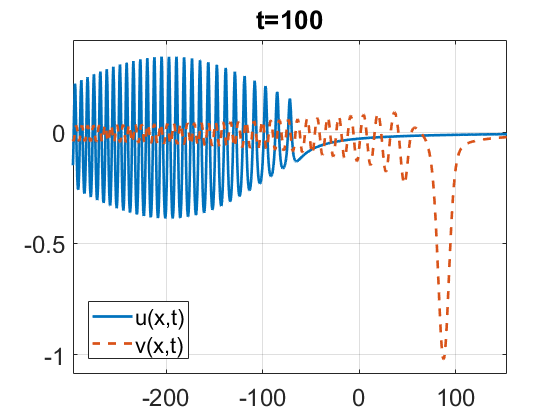}
\includegraphics[width=0.32\textwidth]{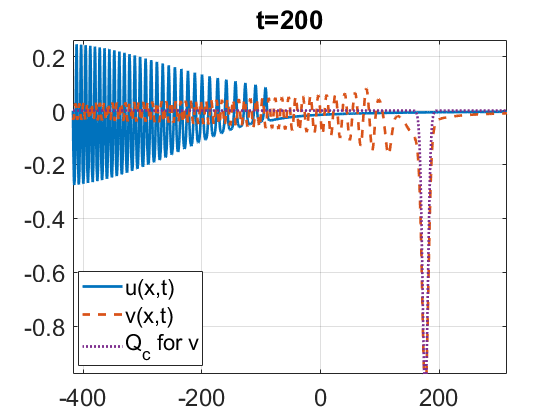}
\includegraphics[width=0.32\textwidth]{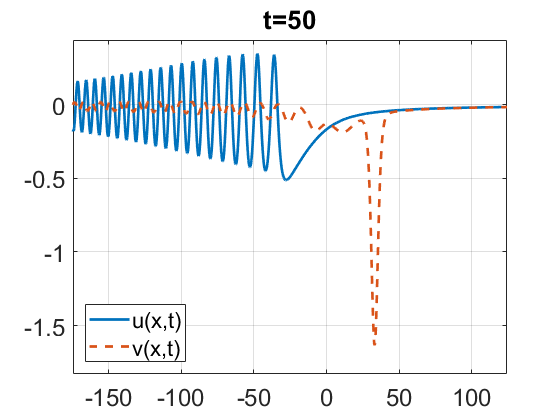}
\includegraphics[width=0.32\textwidth]{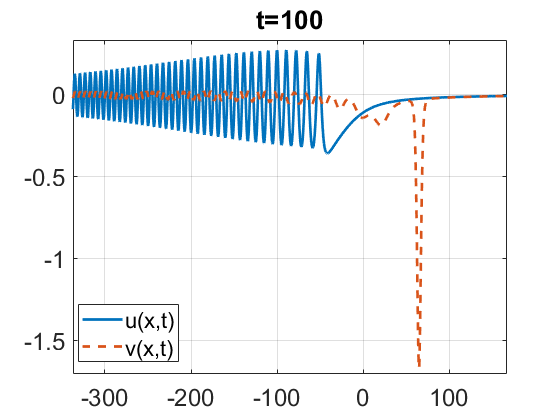}
\includegraphics[width=0.32\textwidth]{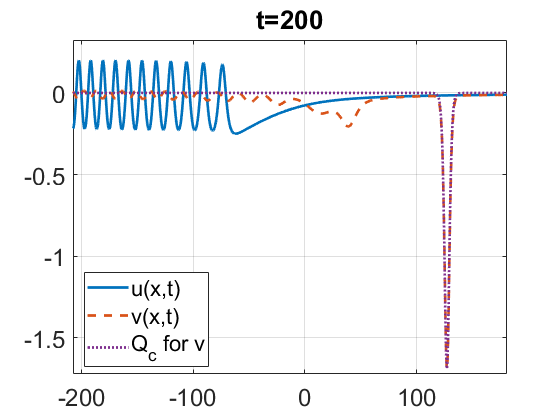}
\includegraphics[width=0.32\textwidth]{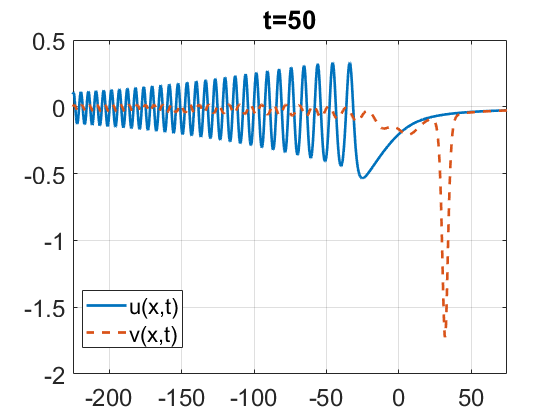}
\includegraphics[width=0.32\textwidth]{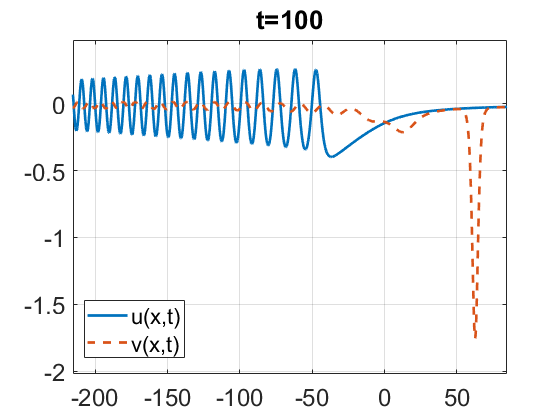}
\includegraphics[width=0.32\textwidth]{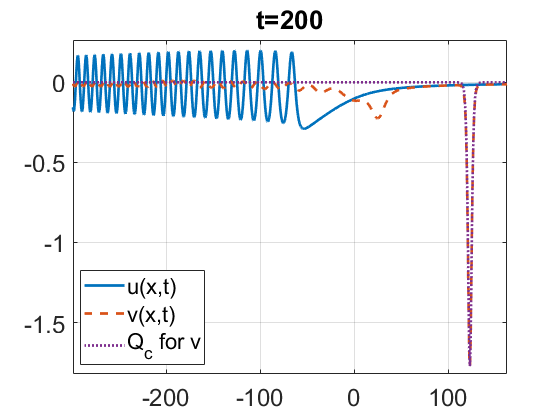}
\includegraphics[width=0.32\textwidth]{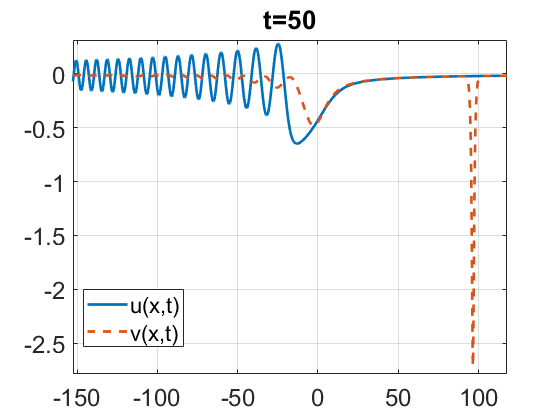}
\includegraphics[width=0.32\textwidth]{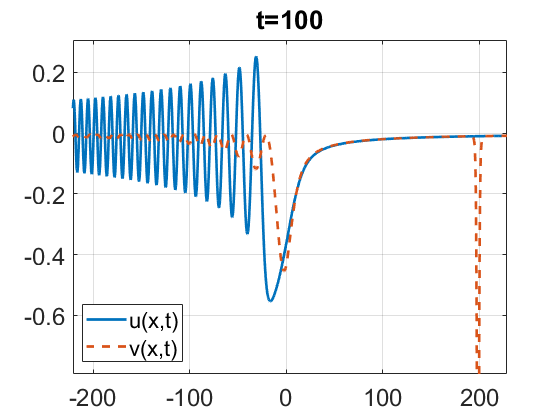}
\includegraphics[width=0.32\textwidth]{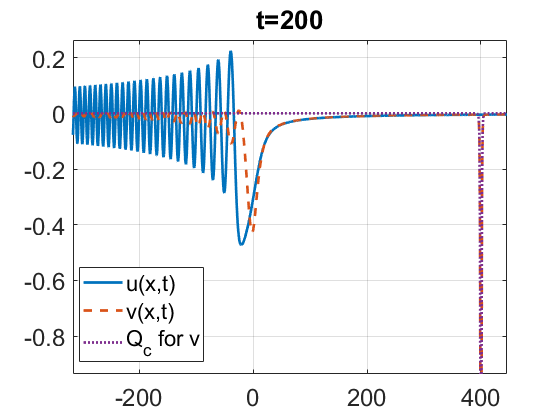}
\caption{\footnotesize  Time evolution of $u_0=v_0=\frac{A}{\sqrt{1+x^2}}$ with $A=-2$. 
Top row: $\alpha=\frac{1}{9}$, second row $\alpha=\frac{7}{9}$, third row: $\alpha=1$, bottom row: $\alpha=3$.
}
\label{F:profile PRLforNEG}
\end{figure}

\begin{figure}[ht]
\includegraphics[width=0.6\textwidth, height=4.6cm]{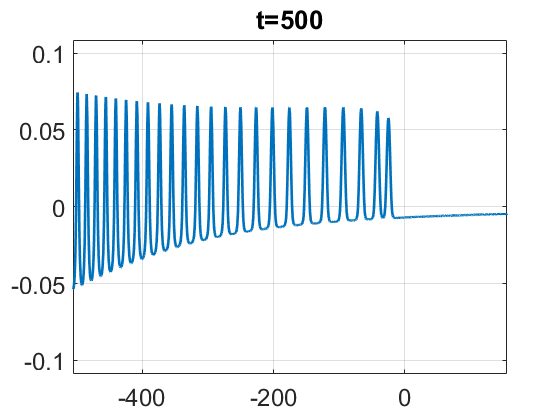}
\caption{\footnotesize Emergence of positive solitons from radiation in gKdV solution $u(x,t)$ in the case $\alpha = \frac19$ at $t=500$ from the initial data as in Figure \ref{F:profile PRLforNEG}. }
\label{F:profile PRLforNEG19}
\end{figure}

\newpage

Summarizing the results in the section, we mention that regardless of initial data decay, the GKdV solutions always generate solitons symmetric with respect to the sign of the initial data and radiation outgoing to the left. The gKdV equation generates solitons in the same manner from the positive data and we have also observed that solitons emerge later from the radiation part, thus, negative data can also produce positive solitons.


\newpage

\section{Multi-bump data}\label{S:4}

We conclude this paper with showing the interaction of various types of two bumps for small powers of $\alpha$ ($\alpha=\frac19$ and $\frac79$). 

\subsection{Interaction of two solitons}\label{S:4.1}
We start with considering the interaction between two solitons, i.e., the two-soliton initial data such as 
\begin{equation}\label{E:multidata-1}
u_0= v_0 = A \, Q_{c_1}(x+a_1) + B\, Q_{c_2}(x+a_2), 
\end{equation}
where we consider $A$ and $B$ of the same and different signs, as well as different shifts $a_1, a_2$.
\begin{figure}[ht]
\includegraphics[width=0.32\textwidth]{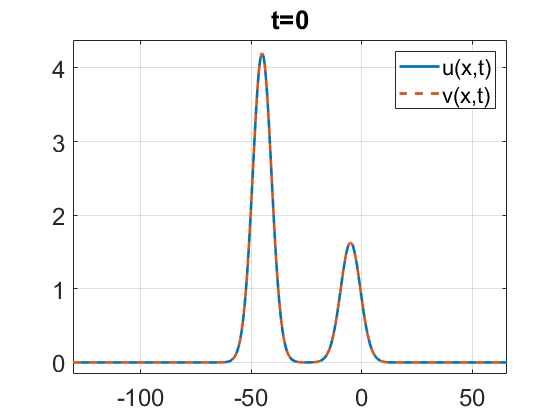}
\includegraphics[width=0.32\textwidth]{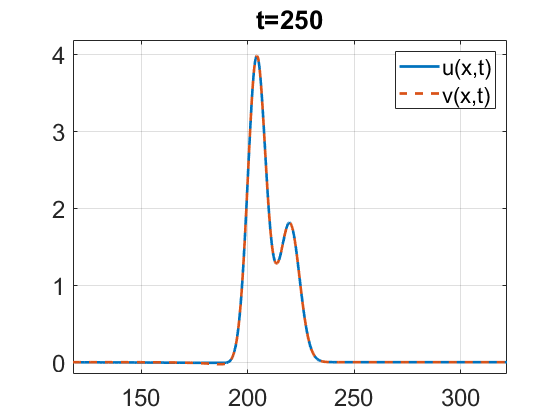}
\includegraphics[width=0.32\textwidth]{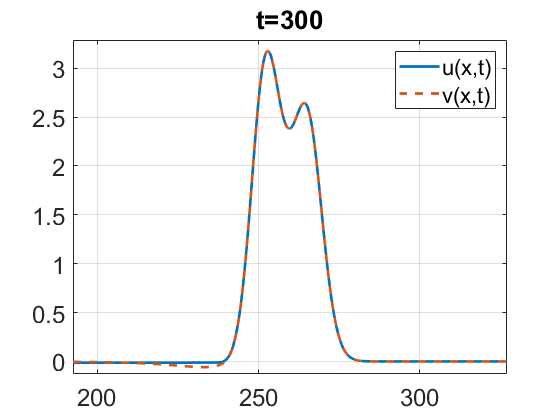}
\includegraphics[width=0.32\textwidth]{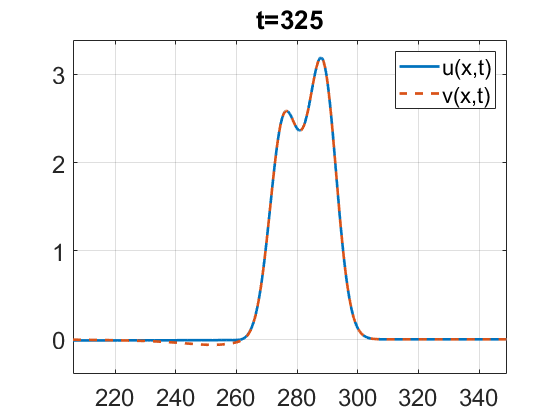}
\includegraphics[width=0.32\textwidth]{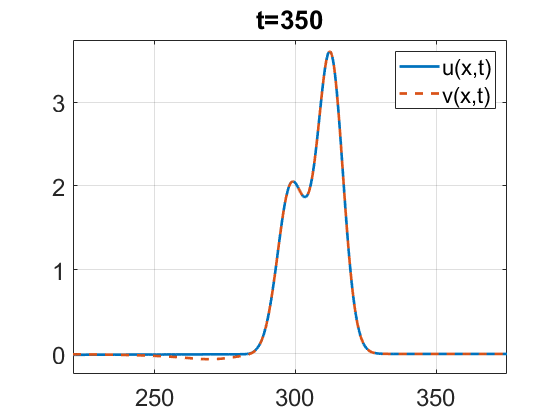}
\includegraphics[width=0.32\textwidth]{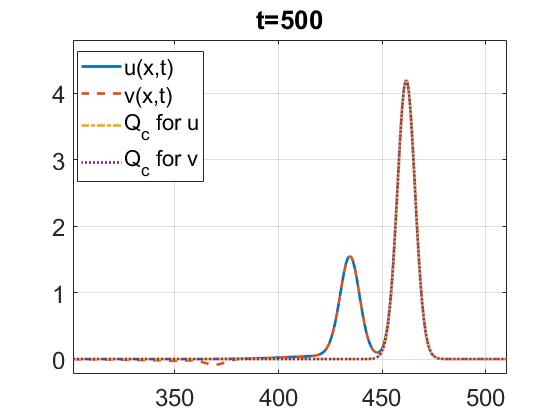}
\caption{\footnotesize  Time evolution for $\alpha=\frac{1}{9}$, $u_0=v_0=Q(x+45)+ Q_{0.9}(x+5)$. } 
\label{F:profile PQPQ p19}
\end{figure}
Figure \ref{F:profile PQPQ p19} illustrates the interaction of two solitons in the case $\alpha=\frac{1}{9}$. Here, we take a larger (and thus, faster) soliton on the left and a smaller one on the right, to be precise, the initial condition is  $u_0=v_0= Q(x+45)+  Q_{0.9}(x+5)$, that is, the data \eqref{E:multidata-1} with $A=B=1$, $c_1=1$, and $c_2=0.9$ (recall that  $Q_c(x) = c^{\,9}Q(\sqrt c \, x)$ when $\alpha=\frac19$ and that is why the second soliton $Q_{0.9}$ is much smaller in amplitude than $Q$ in Figure \ref{F:profile PQPQ p19}). We run our simulations till $t=500$ to observe the interaction. 
After the solitons collide and separate, the solution evolves into basically the same 
original size and shape, see the fitting of the larger bump with the shifted $Q$ in the bottom right plot of Figure \ref{F:profile PQPQ p19}. Furthermore, $u(x,t)$ and $v(x,t)$ practically coincide with each other; there is a tiny amount of radiation outgoing to the left in $v(x,t)$, which is more noticeable after the collision. To track the radiation, we plot the $L^2_x$ and $L^\infty_x$ differences in the solutions on the left of Figure \ref{F:profile PQPQ79diff} (see also Figure \ref{F:profile PQPQ79min} discussed below).  
\begin{figure}[ht]
\includegraphics[width=0.32\textwidth]{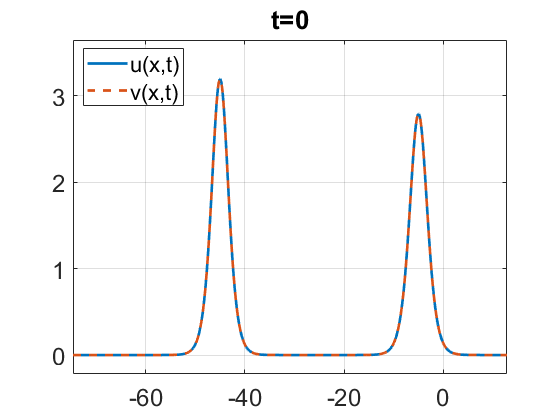}
\includegraphics[width=0.32\textwidth]{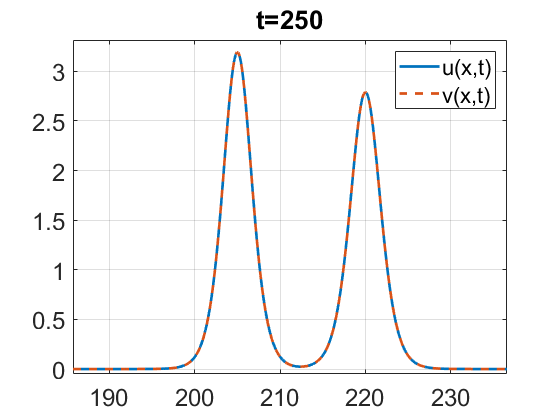}
\includegraphics[width=0.32\textwidth]{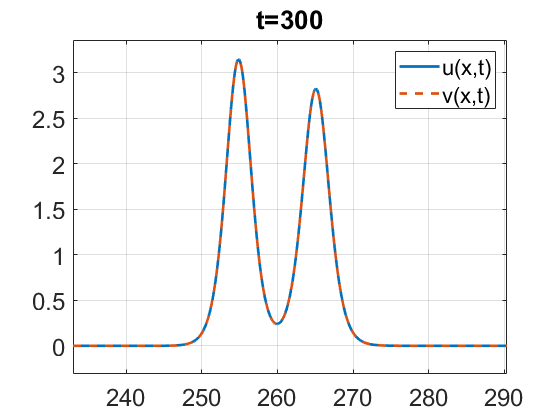}
\includegraphics[width=0.32\textwidth]{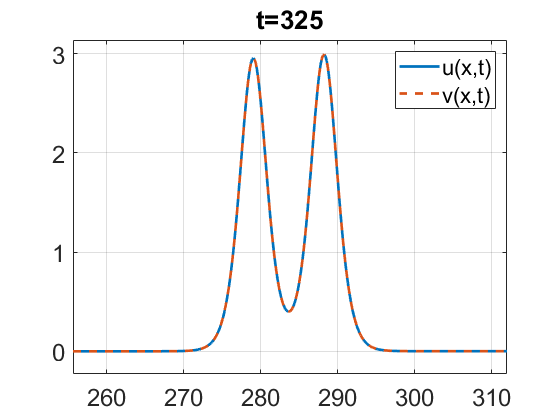}
\includegraphics[width=0.32\textwidth]{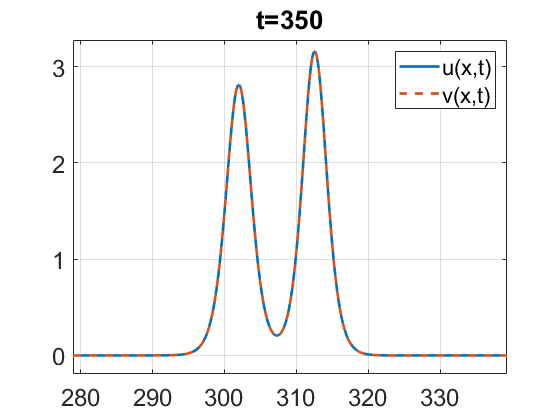}
\includegraphics[width=0.32\textwidth]{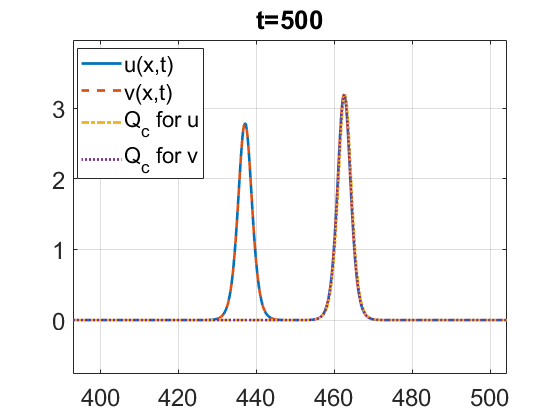}
\caption{\footnotesize  Time evolution for $\alpha=\frac{7}{9}$, $u_0=v_0=Q(x+45)+ Q_{0.9}(x+5)$. }
\label{F:profile PQPQ p79}
\end{figure}

For $\alpha=\frac{7}{9}$ a similar phenomenon occurs, 
the two solitons have the same original size and shape after the collision. Almost no radiation can be seen in this case, since it is close to $\alpha=1$, the integrable case, where the interaction of solitons is elastic;  
the tiny difference between $u(x,t)$ and $v(x,t)$ can be seen in Figure \ref{F:profile PQPQ79diff} (right), which shows that the difference starts closer to the collision (around $t=300$), when the two bumps get relatively close to each other. The difference grows after the interaction, see the right graph of Figure \ref{F:profile PQPQ79diff}, however, note the scale $10^{-4}$. 
\begin{figure}[ht]
\includegraphics[width=0.34\textwidth]{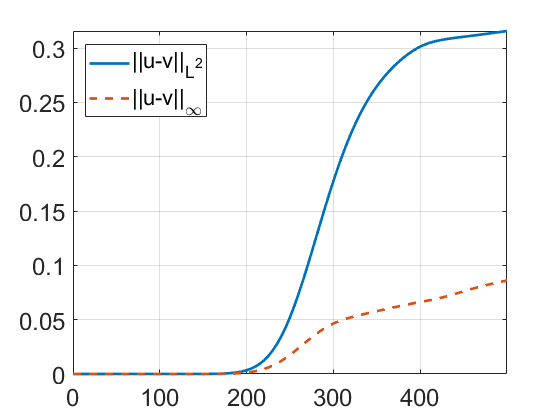}
\includegraphics[width=0.34\textwidth]{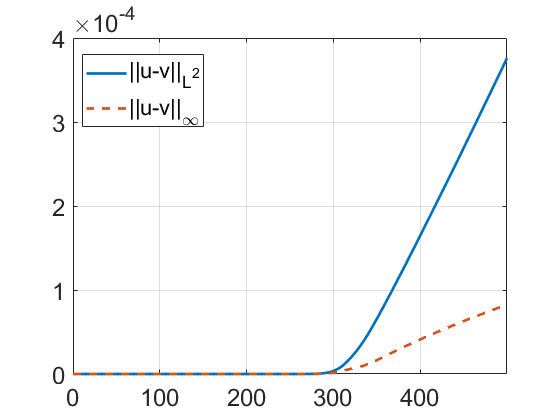}
\caption{\footnotesize Time dependence of the $u(x,t)$ and $v(x,t)$ difference in the $L^2_x$ and $L^{\infty}_x$ norms for $u_0=v_0=Q(x+45)+Q_{0.9}(x+5)$.  
Left: $\alpha=\frac{1}{9}$. Right: $\alpha=\frac{7}{9}$.}
\label{F:profile PQPQ79diff}
\end{figure}

Another way to track how large the radiation is in the dispersion part after the interaction is to look at the minimum value of the solutions at a given time: $\min_{x}(u(x,t))$ and $\min_{x}(v(x,t))$, which is shown in Figure \ref{F:profile PQPQ79min}. In the case $\alpha=\frac{1}{9}$ dispersion appears after $t \approx 200$; after the interaction in the gKdV case (solid blue) the dispersion stays on the order of $10^{-2}$, and in the GKdV case (dash red) it is on the level of $10^{-1}$. The case of $\alpha=\frac{7}{9}$ has dispersion after $t \approx 250$, which drops down after the interaction to the order of $10^{-3}$, much smaller than for $\alpha=\frac19$, confirming that this case is close to the integrable case $\alpha=1$, where no dispersion happens. 

\begin{figure}[ht]
\includegraphics[width=0.34\textwidth]{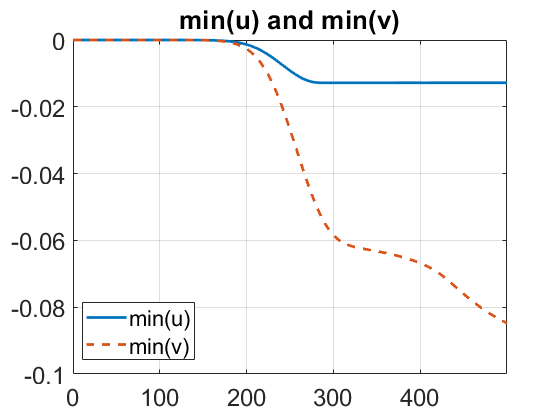}
\includegraphics[width=0.34\textwidth]{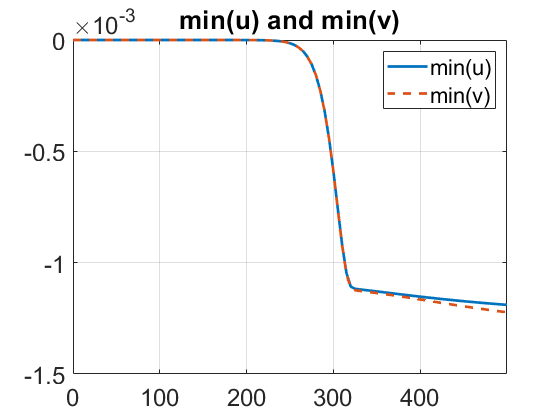}
\caption{\footnotesize  Minimum values of the solution $u$ and $v$ from $u_0=v_0=Q(x+45)+ Q_{0.9}(x+5)$. 
Left: $\alpha=\frac{1}{9}$. Right: $\alpha=\frac{7}{9}$.}
\label{F:profile PQPQ79min}
\end{figure}

We now consider the initial data with solitons of opposite signs and take 
$$
u_0=v_0= - Q_{c_1}(x +a_1)+ Q_{c_2}(x+a_2).
$$
For $\alpha=\frac{1}{9}$ we fix $c_1=1$, $c_2=0.75$, $a_1=100$ and $a_2=60$ and show the time evolution in Figure \ref{F:profile NQPQ p19}. 
\begin{figure}[ht]
\includegraphics[width=0.32\textwidth]{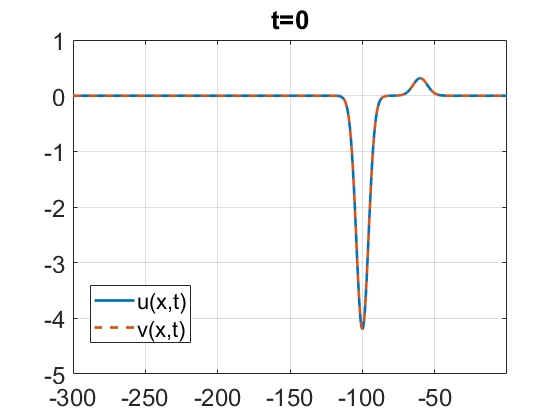}
\includegraphics[width=0.32\textwidth]{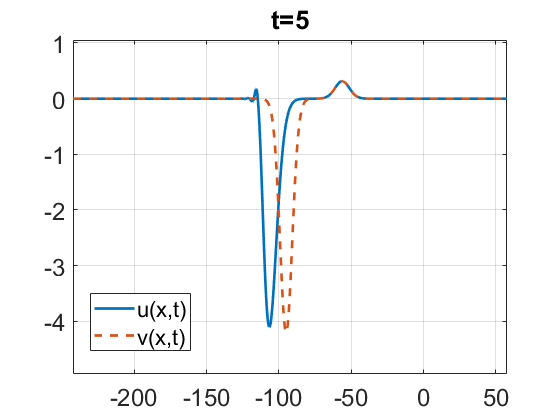}
\includegraphics[width=0.32\textwidth]{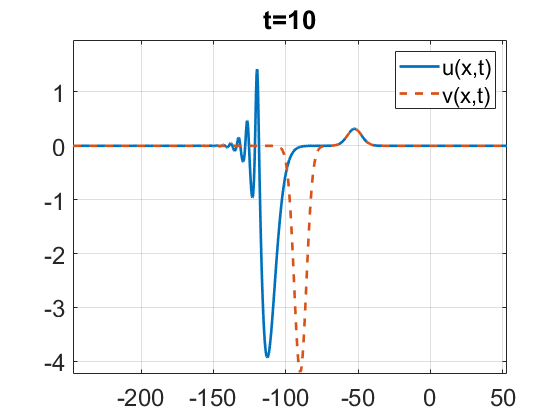}
\includegraphics[width=0.32\textwidth]{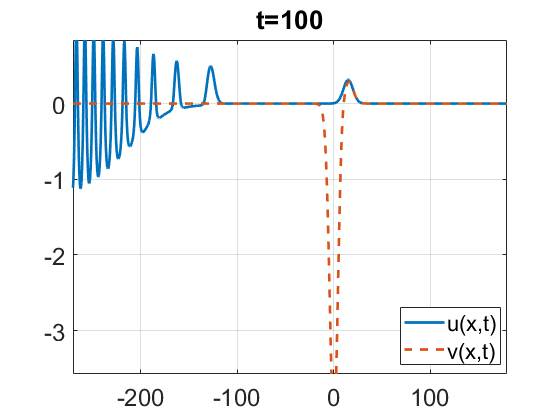}
\includegraphics[width=0.32\textwidth]{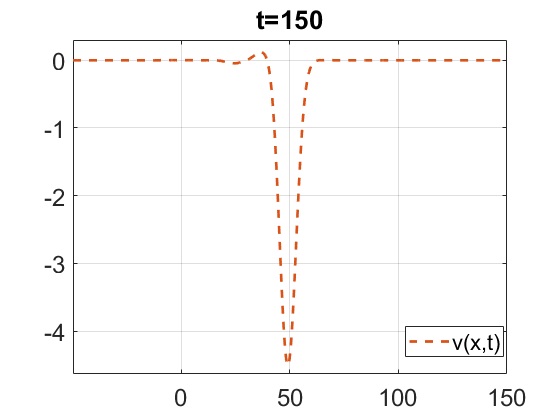}
\includegraphics[width=0.32\textwidth]{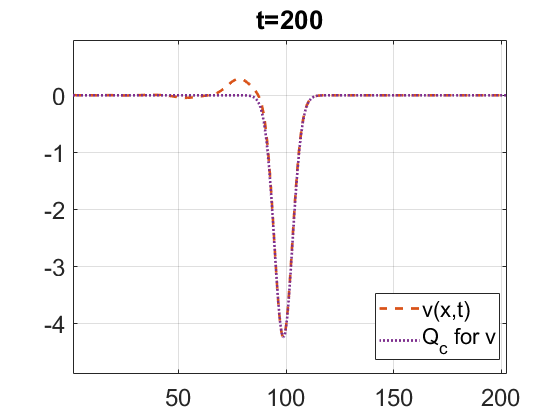}
\caption{\footnotesize Time evolution for $\alpha=\frac{1}{9}$, $u_0=v_0=-Q(x+100)+ Q_{0.75}(x+60)$.}
\label{F:profile NQPQ p19}
\end{figure}
The negative solition has a larger amplitude and speed and expected to catch up and collide with the smaller positive soliton. This is indeed the case for the GKdV solution $v(x,t)$ as can be seen by the dash red curve in Figure \ref{F:profile NQPQ p19}; the bottom right plot shows that after the interaction the large soliton appeared and has the same form as before, the smaller soliton is now behind. A very small radiation is outgoing to the left of both (red dash) solitons, which can be noticed in the middle and right bottom plots in Figure \ref{F:profile NQPQ p19}. 
On the other hand, the solution $u(x,t)$ has the negative soliton disperse into the radiation without any interaction, the smaller positive soliton keeps traveling to the right (solid blue in Figure \ref{F:profile NQPQ p19}). Eventually the radiation part of the solution will generate more solitons as we have discussed previously, one of them is already appearing in the bottom left plot of Figure \ref{F:profile NQPQ p19}. After that we stopped tracking $u(x,t)$ as no new information is produced (on a given computational interval). 

For $\alpha=\frac{7}{9}$ with the same initial data there is bigger difference in positive solitons formed at time $t=200$ (after the interaction of solitons in GKdV model), see Figure \ref{F:profile NQPQ p79}. The negative soliton in $u(x,t)$ as in the previous case radiates completely to the left, and then after some (significant) time will start forming positive solitons (in this case of smaller amplitude than the main positive soliton, see the blue line in the bottom row of Figure \ref{F:profile NQPQ p79}), it has not influence onto the main positive soliton, which travels to the right with the constant speed $c_2=0.75$. On the other hand, in the GKdV solution $v(x,t)$ the negative soliton catches up with the positive one and generate some radiation to the left (see the bottom right plot in Figure \ref{F:profile NQPQ p79}). The two positive 
bumps travel to the right with slightly different values of speed. In the last two plots of the bottom row in Figure \ref{F:profile NQPQ p79}, one can see the positive soliton for $v(x,t)$ is smaller than the positive soliton for $u(x,t)$. This indicates that during the collision some mass was pumped out into the radiation from the positive soliton in $v(x,t)$. The negative soliton in $v(x,t)$ has no visual changes, since the power $\alpha = \frac79$ is very close to the integrable case $\alpha=1$. 

\begin{figure}[ht]
\includegraphics[width=0.32\textwidth]{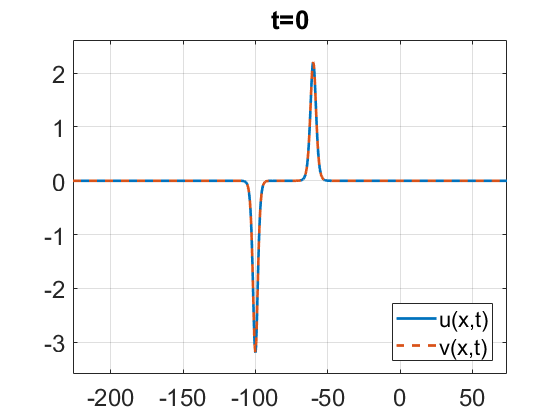}
\includegraphics[width=0.32\textwidth]{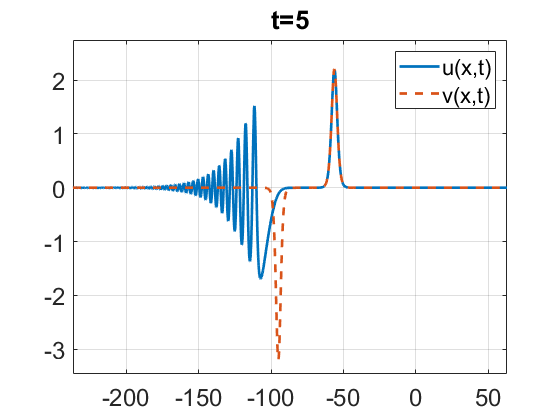}
\includegraphics[width=0.32\textwidth]{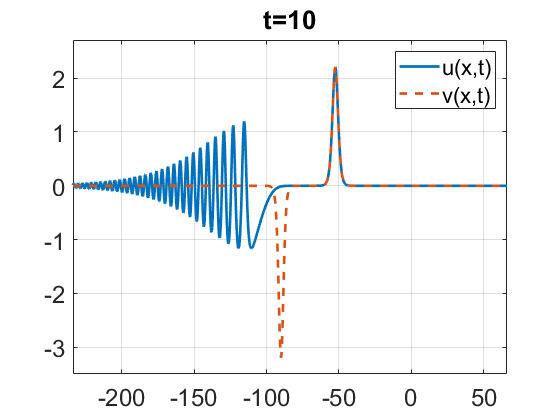}
\includegraphics[width=0.32\textwidth]{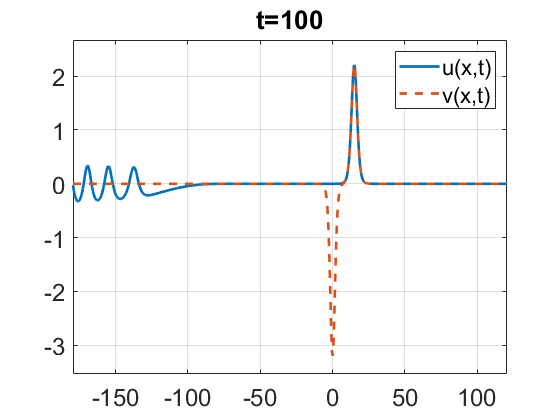}
\includegraphics[width=0.32\textwidth]{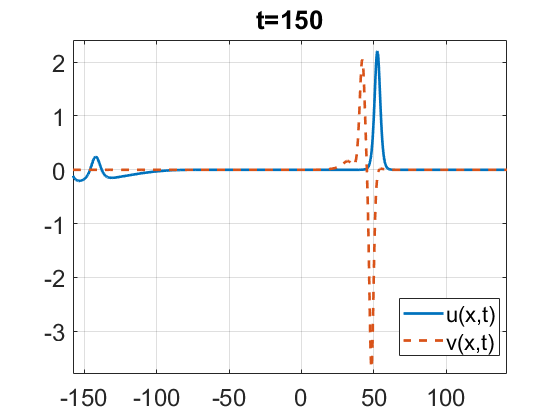}
\includegraphics[width=0.32\textwidth]{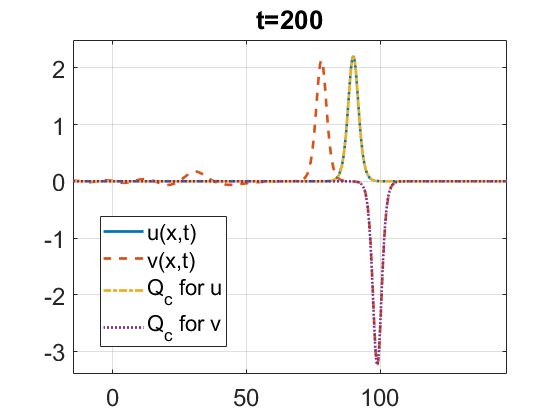}
\caption{\footnotesize Time evolution for $\alpha=\frac{7}{9}$, 
$u_0=v_0=-Q(x+100)+ Q_{0.75}(x+60)$.}
\label{F:profile NQPQ p79}
\end{figure}

In the previous two examples, we had a faster negative soliton placed initially behind the smaller positive soliton and we tracked a possible interaction, though the dispersion of the negative solution into the radiation in the gKdV case did not affect the dynamics of the positive soliton. Therefore, now we place the negative soliton ahead of the positive soliton and track only the gKdV solution $u(x,t)$ to see how the scattering part, which goes to the left, affects the positive soliton. We consider the solution of the form $u_0= - Q(x-100) + Q_{0.9}(x)$.  
\begin{figure}[ht]
\includegraphics[width=0.32\textwidth]{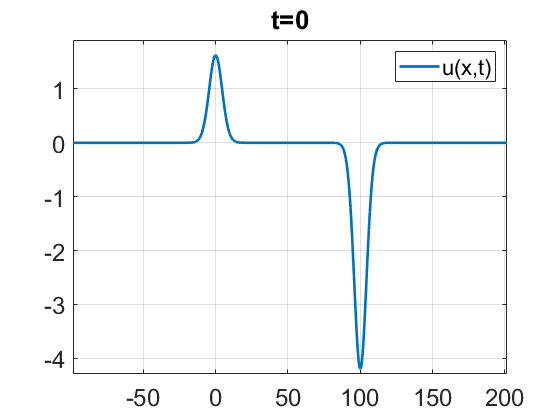}
\includegraphics[width=0.32\textwidth]{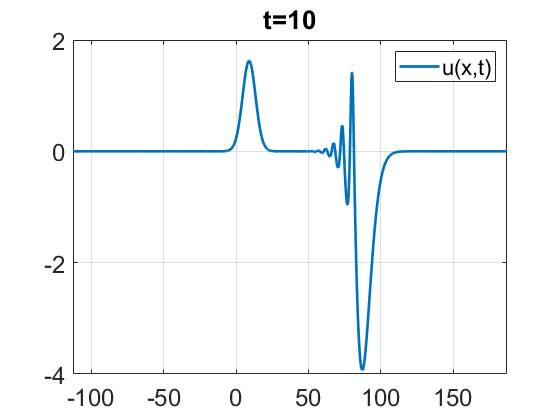}
\includegraphics[width=0.32\textwidth]{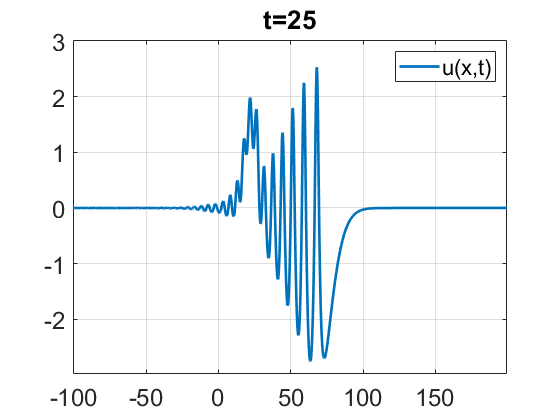}
\includegraphics[width=0.32\textwidth]{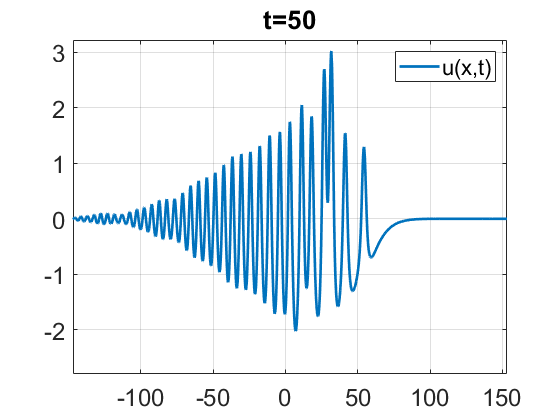}
\includegraphics[width=0.32\textwidth]{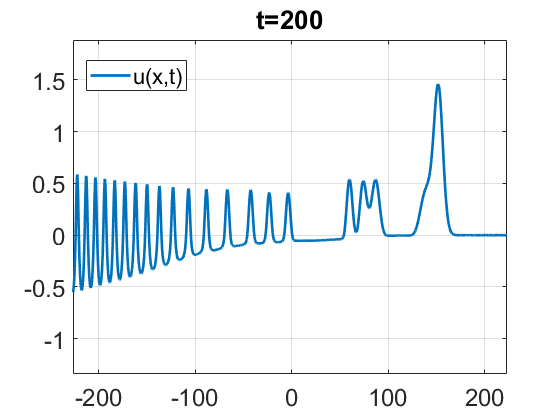}
\includegraphics[width=0.32\textwidth]{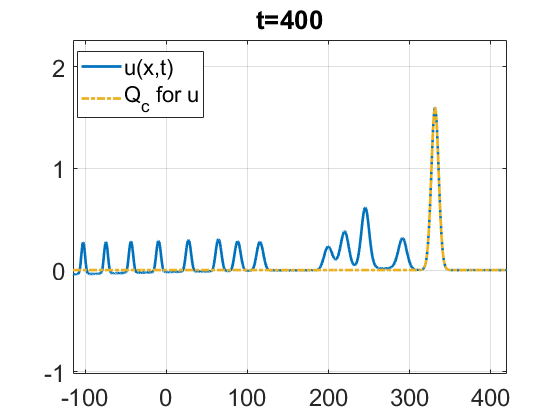}
\caption{\footnotesize Time evolution for $\alpha=\frac{1}{9}$, $u_0=v_0=-Q(x-100)+Q_{0.9}(x)$.}
\label{F:profile NQPQ u p19}
\end{figure}
\begin{figure}[ht]
\includegraphics[width=0.32\textwidth]{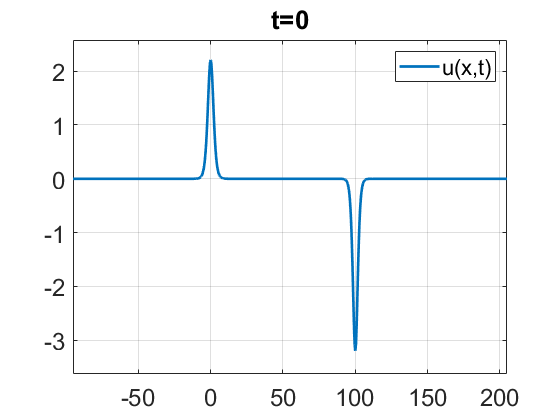}
\includegraphics[width=0.32\textwidth]{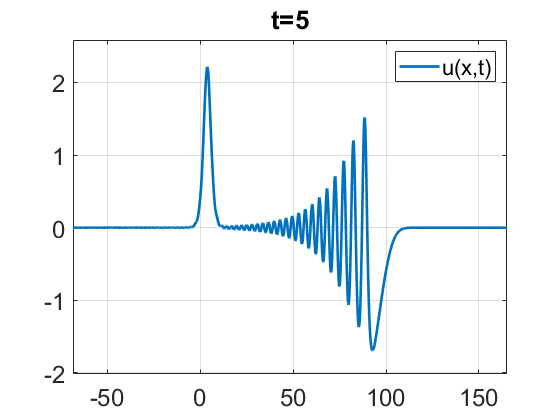}
\includegraphics[width=0.32\textwidth]{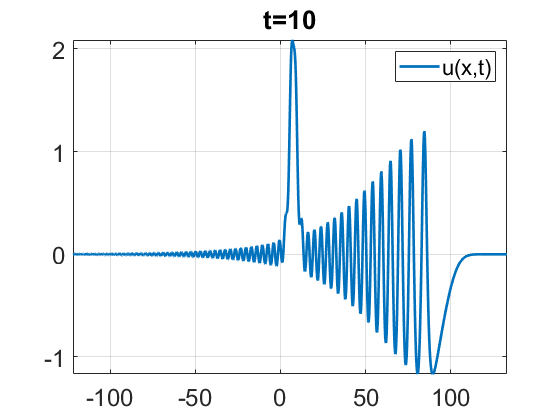}
\includegraphics[width=0.32\textwidth]{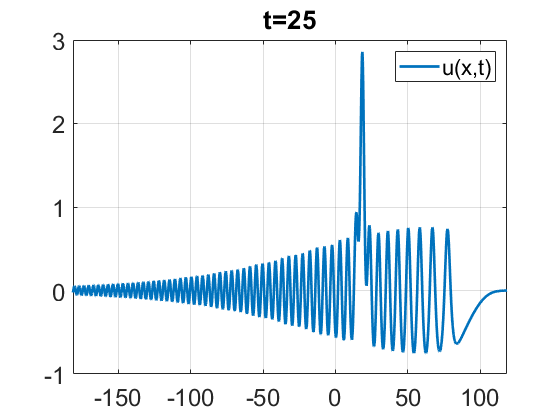}
\includegraphics[width=0.32\textwidth]{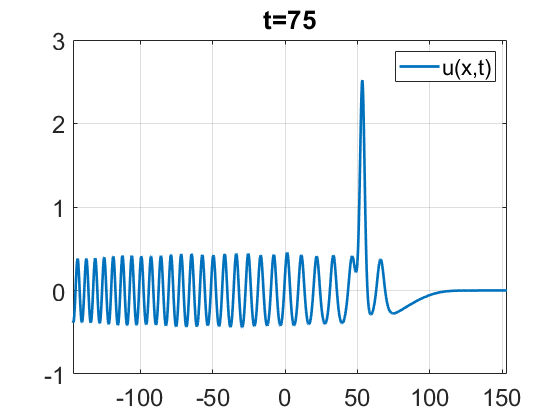}
\includegraphics[width=0.32\textwidth]{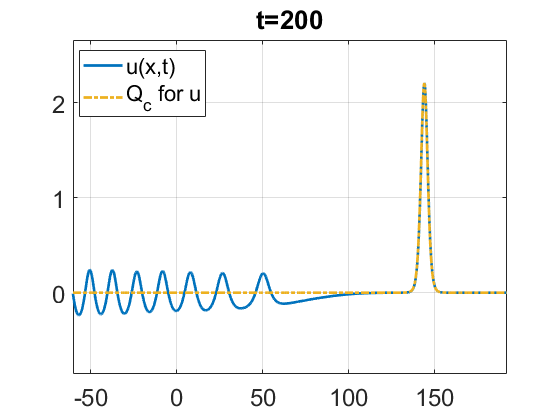}
\caption{\footnotesize Time evolution for $\alpha=\frac{7}{9}$, $u_0=v_0=-Q(x-100)+Q_{0.75}(x)$.}
\label{F:profile NQPQ u p79}
\end{figure}
As the time evolves, the dispersion significantly disturbs the positive soliton, see Figure \ref{F:profile NQPQ u p19}, however, after the interaction, the soliton bump is restored, moreover, it absorbed some mass from the dispersion part, and thus, formed a larger soliton after the interaction (see the bottom right plot in Figure \ref{F:profile NQPQ u p19}).

Similar phenomenon can be seen for $\alpha=\frac{7}{9}$ for the initial condition $u_0=v_0= -Q(x-100)+ Q_{0.75}(x)$ in Figure \ref{F:profile NQPQ u p79}. The positive soliton seems to be absorbing some mass from the radiation during the interaction (e.g., see bottom left plot), however, its height returns to being very close to the initial amplitude (see the bottom right plot). We also expect that more solitons would eventually form from the radiation on the left,  of smaller amplitudes, see the bottom right plots in both Figures \ref{F:profile NQPQ u p19} and \ref{F:profile NQPQ u p79}.

\subsection{Interaction between soliton and Gaussian}

Now we investigate the case when the two bumps have different shapes. For that we take the initial data composed of a single soliton and a Gaussian type bump: 
$$
u_0=v_0=AQ_c(x+a_1)+B\, e^{-x^2}.
$$ 

When both bumps are positive, $A, B>0$, and $\alpha=\frac{1}{9}$, we observe that the Gaussian produces a (small) soliton moving to the right and radiation moving to the left, as expected. The radiation, produced by the Gaussian component, collides with the second soliton, disturbs it a bit, however, continues as a soliton and then interacts with a smaller soliton, see Figure \ref{F:profile PQPE p19}. From the collision, the larger soliton comes out and eventually it shows that it maintained its original shape but slightly smaller amplitude. Figure \ref{F:profile PQPE p19} shows that this phenomenon holds true for both $u(x,t)$ and $v(x,t)$, the largest soliton in both solutions seems to be unaffected in terms of shape, both $u(x,t)$ and $v(x,t)$ visually coincide in the main solitons, but differ in their respective radiations. This difference can be further examined in Figure \ref{F:profile PQPE infty} (left) as well as the height of the main soliton after the interaction. 

\begin{figure}[ht]
\includegraphics[width=0.32\textwidth]{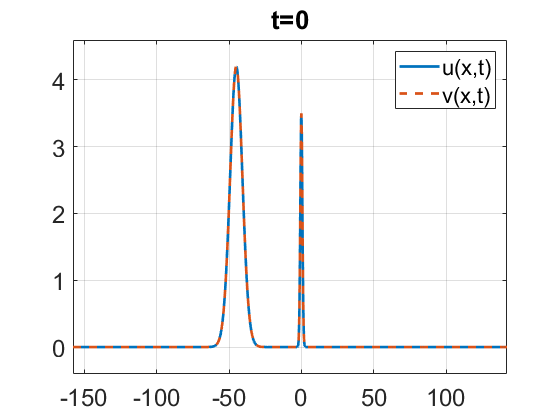}
\includegraphics[width=0.32\textwidth]{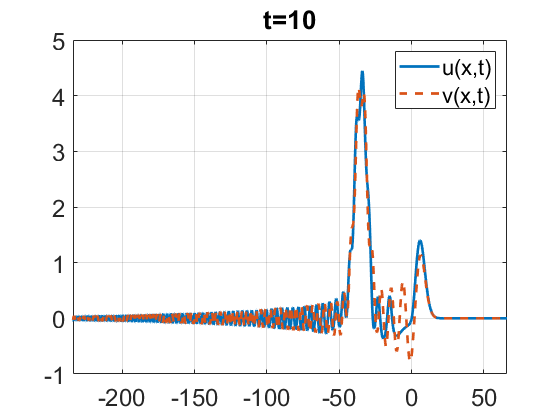}
\includegraphics[width=0.32\textwidth]{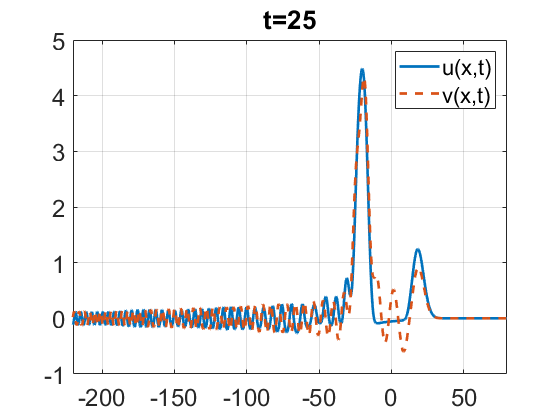}
\includegraphics[width=0.32\textwidth]{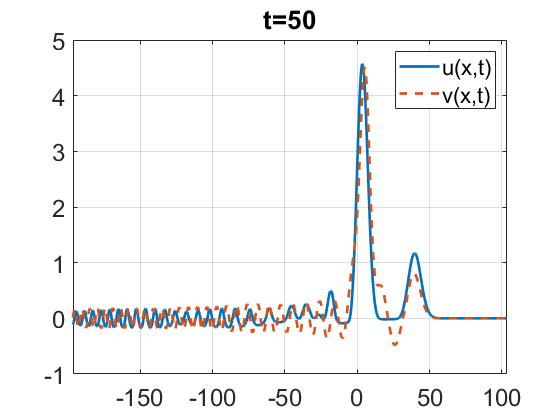}
\includegraphics[width=0.32\textwidth]{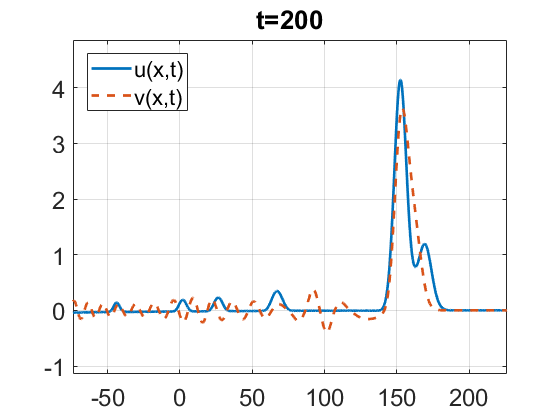}
\includegraphics[width=0.32\textwidth]{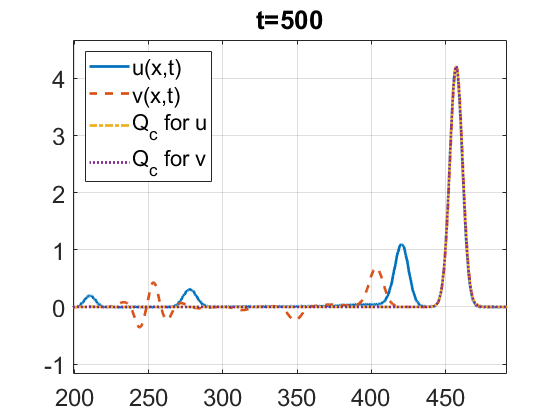}
\caption{\footnotesize Time evolution for $\alpha=\frac{1}{9}$, 
$u_0=v_0=Q(x+45)+3.5\,e^{-x^2}$.}
\label{F:profile PQPE p19}
\end{figure}

For $\alpha=\frac{7}{9}$, the same initial condition $u_0=v_0=Q(x+45)+3.5\, e^{-x^2}$ produces similar dynamics, see Figure \ref{F:profile PQPE p79}. However, one can note that $v(x,t)$ (red dash) collides with its second soliton before $u(x,t)$ (solid blue). The solitons separate after the collision, and again, both largest solitons in $u(x,t)$ and $v(x,t)$ coincide with each other. 
At $t=200$, one can see that $v(x,t)$ (red dash) has formed two positive solitons with at least one more negative forming, the gKdV solution $u(x,t)$ (solid blue) formed only two solitons. We also see that from the last plot in Figure \ref{F:profile PQPE p79}, the second largest (positive) soliton for $v(x,t)$ is slightly behind the second one for $u(x,t)$. 

\begin{figure}[ht]
\includegraphics[width=0.32\textwidth]{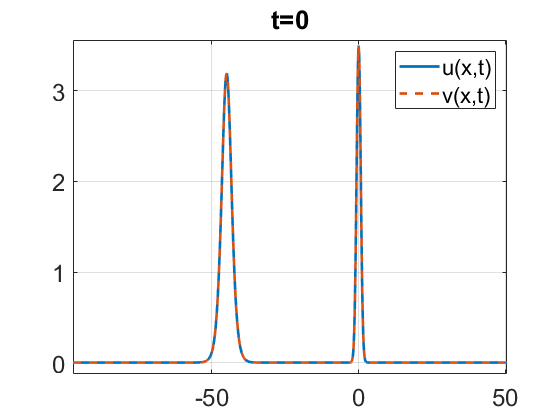}
\includegraphics[width=0.32\textwidth]{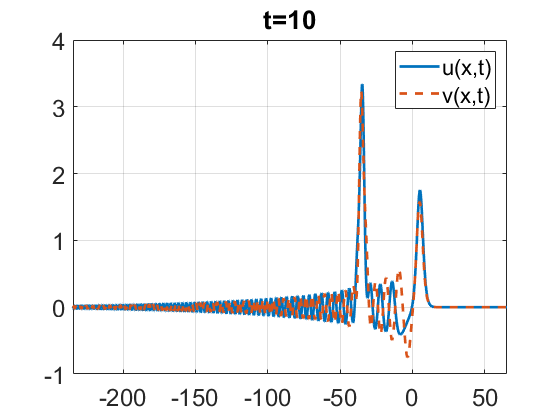}
\includegraphics[width=0.32\textwidth]{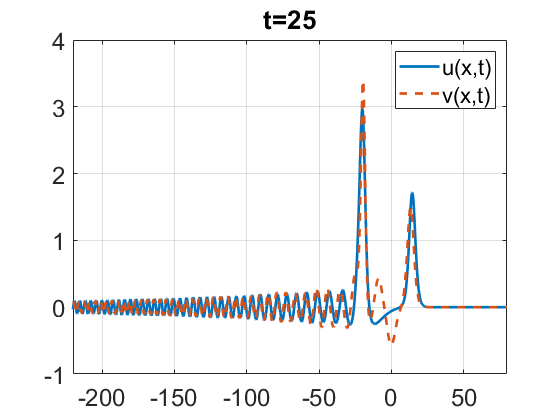}
\includegraphics[width=0.32\textwidth]{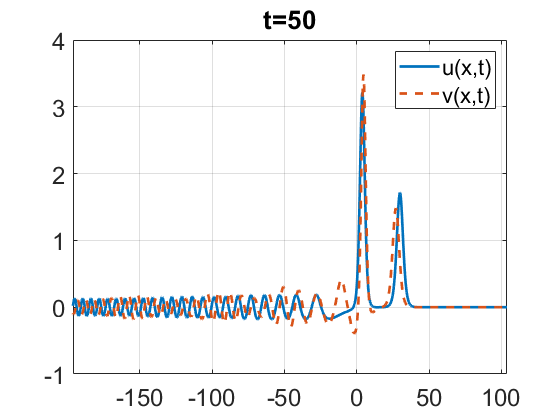}
\includegraphics[width=0.32\textwidth]{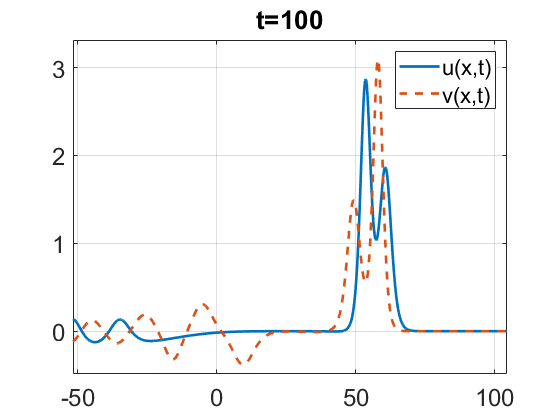}
\includegraphics[width=0.32\textwidth]{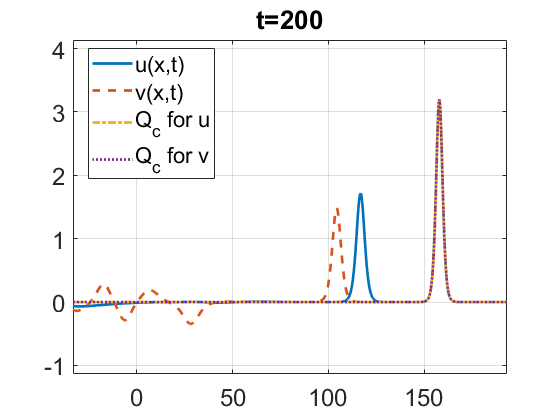}
\caption{\footnotesize Time evolution for $\alpha=\frac{7}{9}$, $u_0=v_0=Q(x+45)+3.5\,e^{-x^2}$.}
\label{F:profile PQPE p79}
\end{figure}

Figure \ref{F:profile PQPE infty} tracks the $L^{\infty}_x$ norm for $u$ and $v$. One can see the collision time (the deepest point) in the GKdV is later than in the gKdV evolution. 

\begin{figure}[ht]
\includegraphics[width=0.36\textwidth]{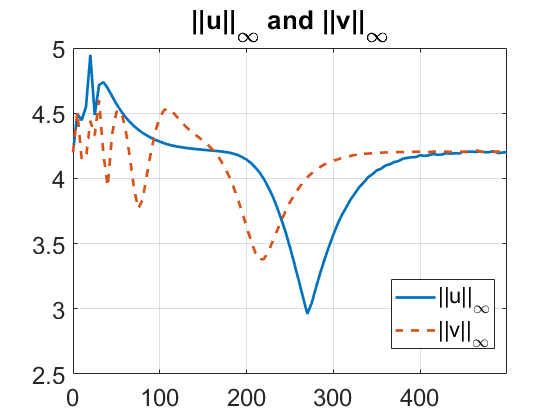}
\includegraphics[width=0.36\textwidth]{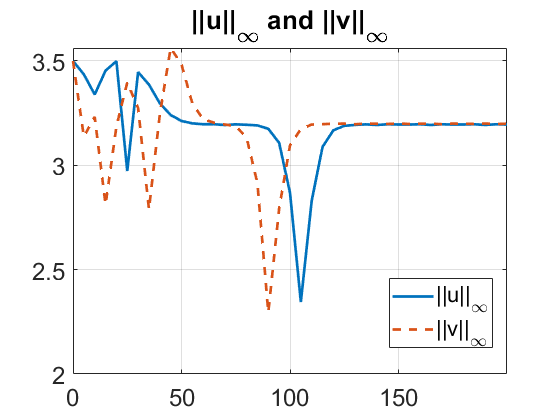}
\caption{\footnotesize Time dependence of $L^\infty_x$ norms for $u$ and $v$ from the data as in Figures \ref{F:profile PQPE p19} and \ref{F:profile PQPE p79} for $\alpha=\frac{1}{9}$ (left) and $\alpha=\frac79$ (right).}
\label{F:profile PQPE infty}
\end{figure}


\subsection{Interaction between two Gaussians}

When we consider the initial data with two Gaussian bumps of opposite signs, we place the negative soliton to the right of the positive one, since it will generate dispersion in both $u$ and $v$ solutions, going to the left, and thus, may affect the formation of a positive soliton from the positive Gaussian bump. We take 
$$
u_0=v_0= -2.5\, e^{-(x-50)^2}+1.5\,e^{-x^2}.
$$
Figures \ref{F:profile Ngaugau p19} and \ref{F:profile Ngaugau p79} show the solution at various times for $\alpha=\frac{1}{9}$ and $\alpha=\frac{7}{9}$, respectively.
\begin{figure}[ht]
\includegraphics[width=0.32\textwidth]{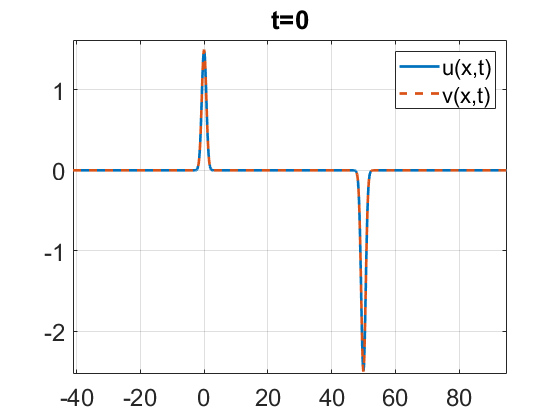}
\includegraphics[width=0.32\textwidth]{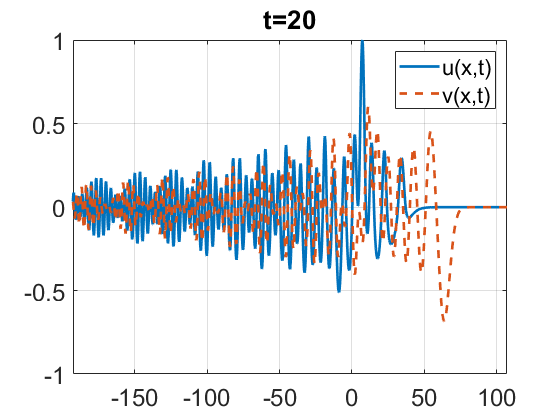}
\includegraphics[width=0.32\textwidth]{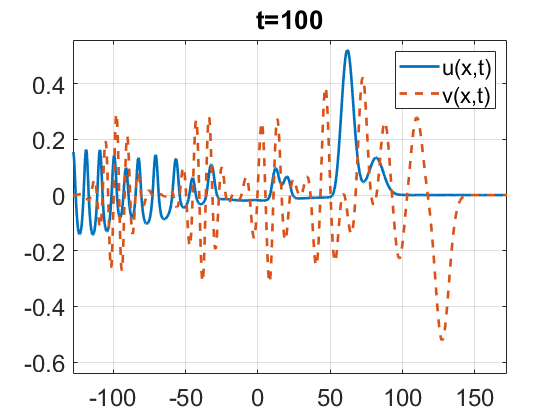}
\includegraphics[width=0.32\textwidth]{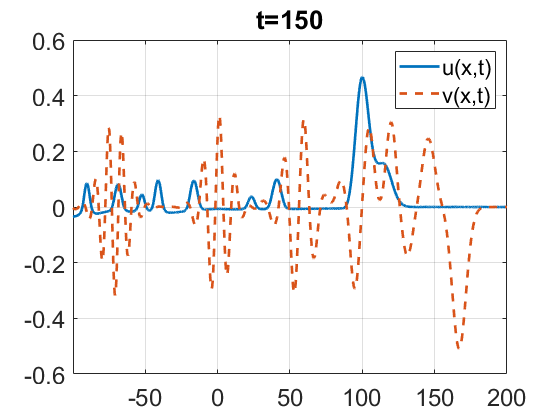}
\includegraphics[width=0.32\textwidth]{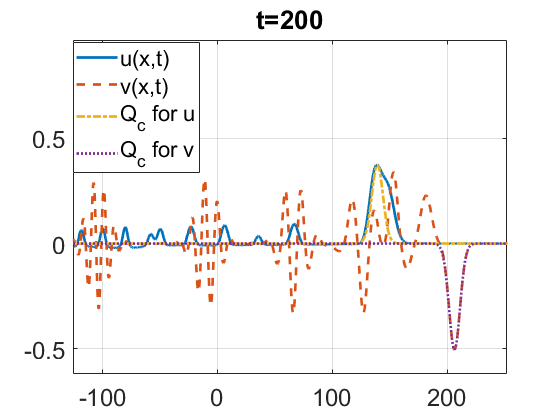}
\includegraphics[width=0.32\textwidth]{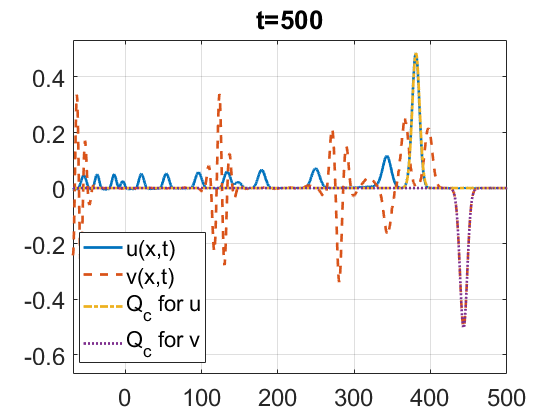}
\caption{\footnotesize Time evolution for $\alpha=\frac{1}{9}$, $u_0=v_0=-2.5\,e^{-(x-50)^2} + 1.5\, e^{-x^2}$.}
\label{F:profile Ngaugau p19}
\includegraphics[width=0.32\textwidth]{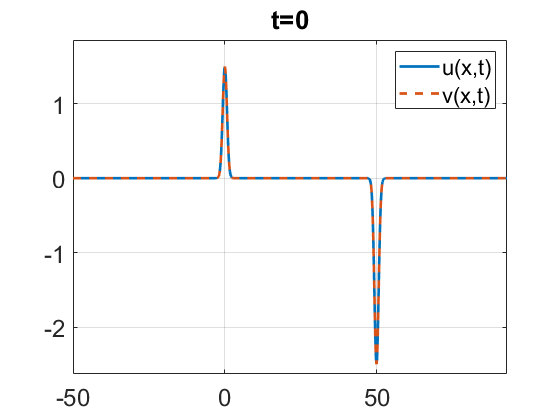}
\includegraphics[width=0.32\textwidth]{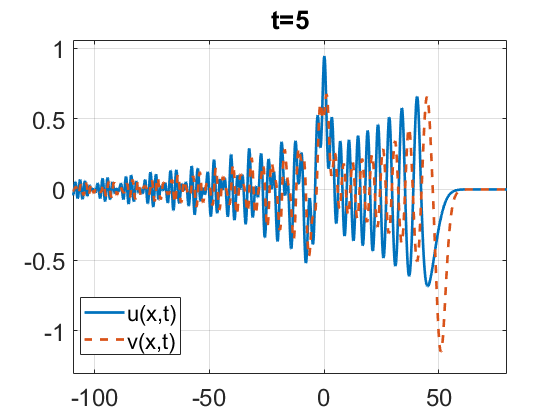}
\includegraphics[width=0.32\textwidth]{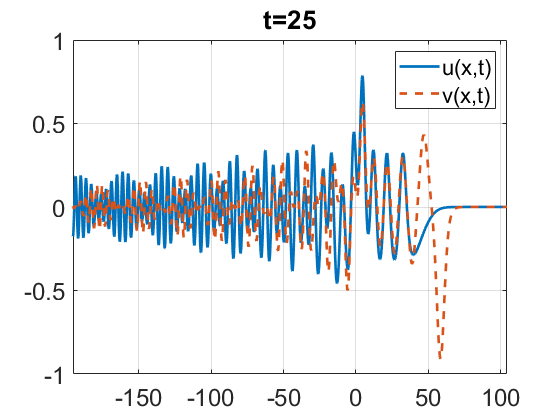}
\includegraphics[width=0.32\textwidth]{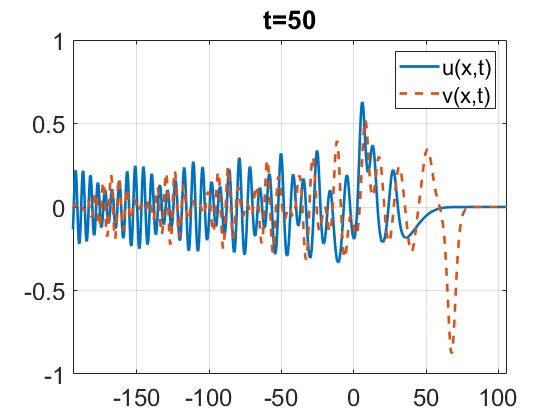}
\includegraphics[width=0.32\textwidth]{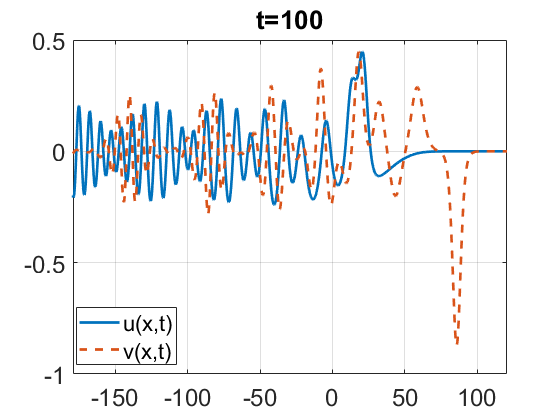}
\includegraphics[width=0.32\textwidth]{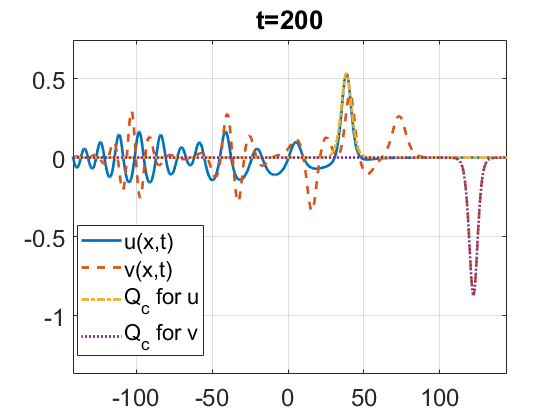}
\caption{\footnotesize Time evolution for $\alpha=\frac{7}{9}$, $u_0=v_0=-2.5\,e^{-(x-50)^2}+1.5\, e^{-x^2}$.}
\label{F:profile Ngaugau p79}
\end{figure}

In the gKdV case, as we have seen in one-bump examples, the negative bump would go into a dispersive radiation and the positive bump would converge to a rescaled soliton. In their combination, the gKdV (solid blue) negative bump produces quite a severe disturbance of the positive part, taking some of its mass into the radiation, and after that interaction the largest positive soliton emerges (see bottom right plot in Figures \ref{F:profile Ngaugau p19} and \ref{F:profile Ngaugau p79}, where we also matched (dotted yellow) with the rescaled soliton). The negative radiation also forms a smaller positive soliton to the right of the largest soliton, this can be seen more clearly in the top right plot of Figure \ref{F:profile Ngaugau p19}, and there maybe even more interactions of the largest soliton with other ones as can be seen when tracking the $L^\infty_x$ norm, see Figure \ref{F:profile Ngaugau infty} 
(two dips in $\alpha=\frac19$ case and at least three in $\alpha=\frac79$ case). 
In the $\alpha=\frac19$ case one can also see a train of solitons (solid blue bumps decreasing in height) form at $t=500$, see the bottom right plot in Figure \ref{F:profile Ngaugau p19}. The same will happen in the $\alpha=\frac79$ case, thought at much longer time. 
\begin{figure}[ht]
\includegraphics[width=0.32\textwidth]{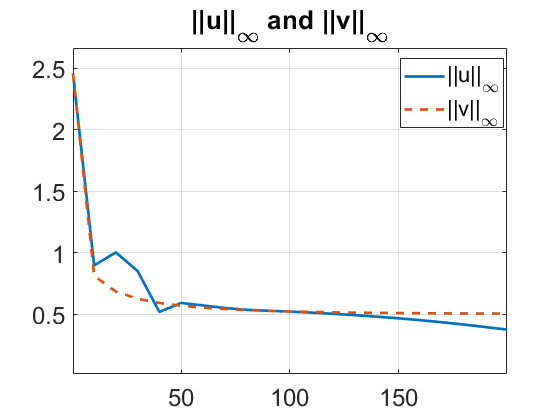}
\includegraphics[width=0.32\textwidth]{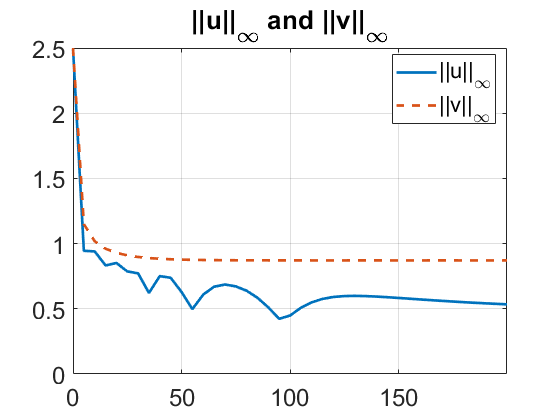}
\caption{\footnotesize 
$L^\infty_x$ norm for solutions  
in Figures \ref{F:profile Ngaugau p19}, \ref{F:profile Ngaugau p79} 
for $\alpha=\frac{1}{9}$ (left) and $\alpha=\frac79$ (right).}
\label{F:profile Ngaugau infty}
\end{figure}
In the GKdV case (dash red), the negative bump forms several solitons alternating in sign (starting from the largest negative one and then positive), that can be seen right away in the top row of Figures \ref{F:profile Ngaugau p19} and \ref{F:profile Ngaugau p79}. 
The right plot in Figure \ref{F:profile Ngaugau infty} indicates convergence to the largest soliton almost immediately in the GKdV case (dash red). The positive bump also forms a positive soliton, however, it gets disturbed by the radiation coming from the negative bump. Eventually several solitons (some negative, some positive) will be formed from the interaction and radiation as can be seen in the last plots of Figures \ref{F:profile Ngaugau p19} and \ref{F:profile Ngaugau p79}. 
\begin{figure}[ht]
\includegraphics[width=0.32\textwidth]{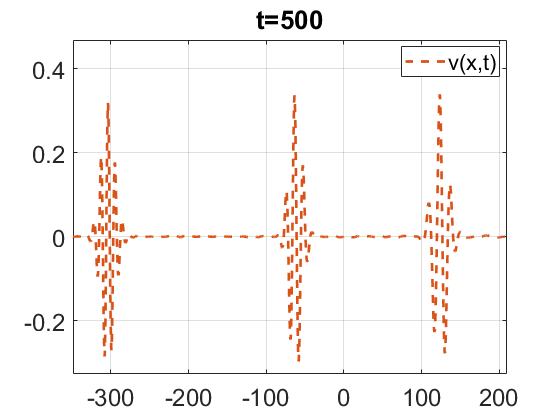}
\includegraphics[width=0.32\textwidth]{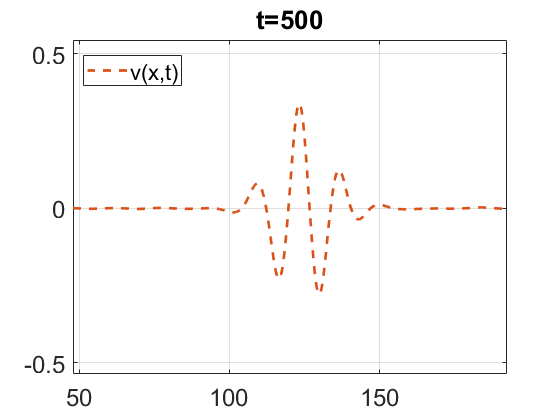}
\includegraphics[width=0.32\textwidth]{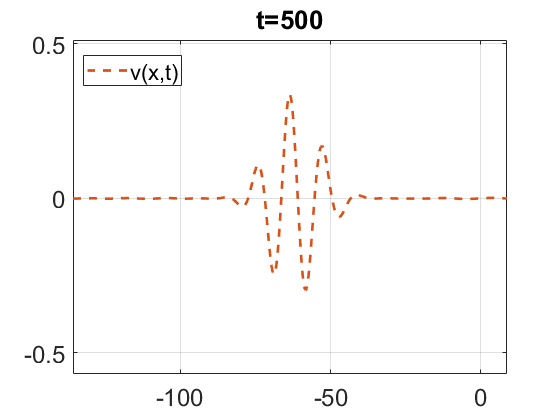}
\includegraphics[width=0.32\textwidth]{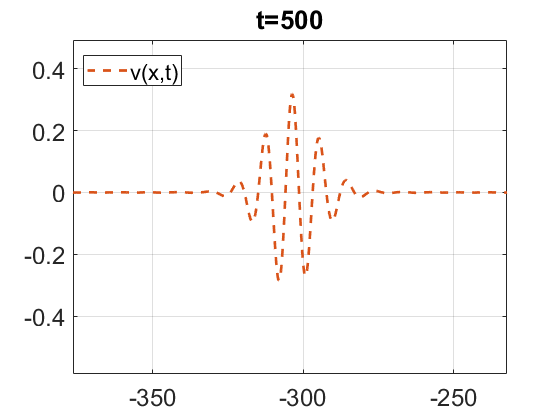}
\includegraphics[width=0.32\textwidth]{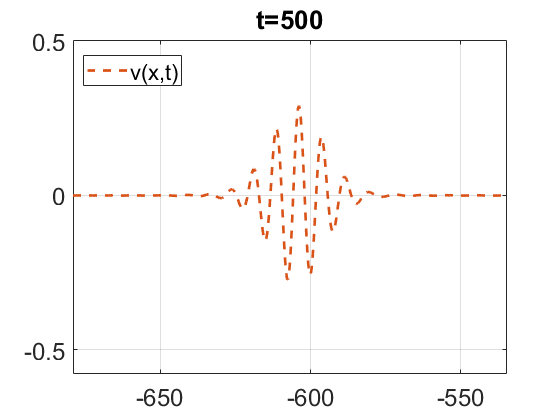}
\includegraphics[width=0.32\textwidth]{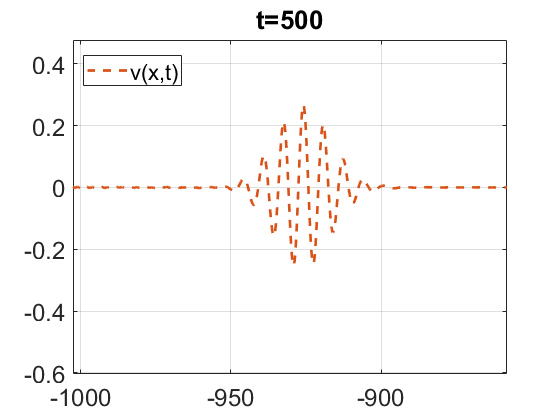}
\caption{\footnotesize Breathers in GKdV \eqref{GK} for $\alpha=\frac{1}{9}$ with $v_0=-2.5\,e^{-(x-50)^2} + 1.5\, e^{-x^2}$ at $t=500$.}
\label{F:profile Ngaugau p19 breather}
\end{figure}
Examining the GKdV solution's radiation part a bit more closely, one can observe the formation of {\it breathers}. This is easier to see for the lower nonlinearity case, $\alpha=\frac{1}{9}$. We zoom-in into the range $[-350,200]$ at $t=500$, see the first top left plot in Figure \ref{F:profile Ngaugau p19 breather}, where we see the periodic structures propagating to the left; then we examine several breathers separately, enlarging even further to see their shape in the rest of the plots of Figure \ref{F:profile Ngaugau p19 breather}. We study this phenomenon in more details elsewhere. \\


\section{Conclusion}\label{S:C}
In this work, we investigated solutions to the generalized KdV equation, with the emphasis on the low powers of nonlinearity, and also comparing the affect of having an absolute value in the nonlinearity, i.e., $|u|^\alpha \, u_x$ vs. $u^{\alpha} u_x$. First, in the analytical part, we proved the local well-posedness for a certain subclass of $H^1$ space that includes polynomially decaying initial data:  Theorem \ref{mainTHM} rigorously establishes the existence of local solutions for the Cauchy problems \eqref{gKdV} and \eqref{GK}  for any $\alpha>0$. A main ingredient is the deduction of Lemma \ref{derivexp2}, which relates solutions of the linear KdV equation and fractional weights.
We point out that our well-posedness result extends the conclusions in \cite{LinaresMiyaGus} to the weights of fractional order. Having fractional powers offer more flexibility in applications as they provide a wider class of resolution spaces for the equations \eqref{gKdV} and \eqref{GK}. 

Next, in our numerical study, we investigated the long term behavior of solutions that are given by the local well-posedness of Theorem \ref{mainTHM} and also extended to those, which are not covered by the theorem (e.g., data with exponential decay). Via numerical simulations we were able to study the dynamics of solutions to the equations \eqref{GK} and \eqref{gKdV}, which turns out to be quite different, unless the initial data is positive. The difference in the {\it positive} case is that the solitons formed by the GKdV ($|u|^{\alpha} u_x$) evolution have smaller height than in the gKdV ($u^{\alpha} u_x$) evolution. Given {\it negative} data, solutions to the gKdV equation \eqref{gKdV} initially  disperse completely into the radiation, but then after some time start generating positive solitons. Solutions to the GKdV always form solitons regardless of the initial data sign, in fact, symmetrically with respect to the sign of data. We also observe that the lower power of nonlinearity is, the shorter time it takes for the solution to form multiple solitons such as train of solitons in gKdV or alternating solitons in GKdV (especially, taking into account the computational restrictions such as range, time and errors), therefore, it might be useful to consider low powers of $\alpha$ for future studies of long term behavior of solutions. We also studied the interaction of two-bump initial data that further helps understanding the formation of solitons, radiation, and even breathers in radiation, confirming the soliton resolution conjecture.

\section{Appendix}
We briefly describe the numerical method used in this paper. For the spatial discretization, we use the Fourier spectral method, see e.g., \cite{GWWC2017}, \cite{GCW2014} and \cite{CWJ2021}. We take large enough periodic domain $[-L,L]$ to approximate the solution on $\mathbb{R}$. The equation in the frequency space becomes the nonlinear ODE system:
\begin{align}\label{E:FFT}
(\hat{U}_k)_t+\left( \frac{i\pi k}{L} \right)^3 \hat{U}_k + \frac{i\pi k}{L} \widehat{U^3}_k=0,
\end{align}
where $k=-N/2, \cdots, N/2-1$, and $\hat{U}$ is the Fourier coefficients vector of the solution $ U_k  \approx  u(x_k)  $ obtained from Fast Fourier transform.

The resulting ODE system \eqref{E:FFT} can be solved by two different approaches. The first natural approach is to use the standard implicit mid-point rule, i.e.,
\begin{align}\label{E:IRK2}
\frac{(\widehat{U^{n+1}})_k-(\widehat{U^n})_k }{\Delta t}+\left( \frac{i\pi k}{L} \right)^3 \widehat{(U^{n+\frac{1}{2}}) }_k+ \frac{i\pi k}{L} \widehat{(U^{n+\frac{1}{2}} )^3}_k=0,
\end{align}
where $U^{n} \approx u(x,t_n)$ is the value of $u(x,t)$ at the $n$th time step, and $\Delta t$ is the step size.
We denote $U^{n+\frac{1}{2}}=\frac{U^{n+1}+U^n }{2}$ for simplicity. The system \eqref{E:IRK2} can be solved by the fixed point iteration. We refer to \cite{KP2015} and \cite{Yang2021} for more details about the implementation of the fixed point iteration solver.

The other choice is to use explicit methods to avoid solving the nonlinear system. The Exponential time differencing integrator can alleviate the stiffness. In this paper, we use the 4th order Exponential Time Differencing Runge-Kutta (ETDRK4) method from \cite{CM02}, \cite{KT05}.
We consider the ODE of the form
 \begin{align}\label{E:ODE ETD}
U_t=\mathbf{L}U+\mathbf{N}(U),
\end{align}
where $\mathbf{L}$ is the matrix that comes from the discretization of the spatial derivatives, and $\mathbf{N}(U)$ is the nonlinear term. For example, $\mathbf{L}$ is the diagonal matrix with the diagonal term ranging from $-N/2, \cdots, N/2-1$, if the gKdV equation system is discretized by the Fourier spectral method.
Then, rewriting the equation \eqref{E:ODE ETD} in the Duhamel form, we get,
\begin{align}\label{E:ODE Du}
U^{n+1}=e^{\mathbf{L}\Delta t}U^{n}+  \int_0^{\Delta t} e^{\mathbf{L} (\Delta t-s)} \mathbf{N}(U(t_n+s)) ds.
\end{align}
Therefore, it is only needed to find a quadrature to approximate the integral term $$\int_0^{\Delta t} e^{\mathbf{L} (\Delta t-s)} \mathbf{N}(U(t_n+s)) ds.$$
One possibility is the fourth order scheme from Cox and Matthews (known as ETDRK4) \cite{CM02}, which gives the following formulas:
\begin{align}\label{E:ETDRK}
&a_n=e^{\mathbf{L} \Delta t/2}U^n+ \mathbf{L}^{-1} (e^{\mathbf{L} \Delta t/2} - \mathbf{I})\mathbf{N}(U^n); \nonumber \\
&b_n=e^{\mathbf{L} \Delta t/2}U^n+ \mathbf{L}^{-1} (e^{\mathbf{L} \Delta t/2} - \mathbf{I})\mathbf{N}(a_n); \nonumber \\
&c_n=e^{\mathbf{L} \Delta t/2}a_n+ \mathbf{L}^{-1} (e^{\mathbf{L} \Delta t/2} - \mathbf{I})(2\mathbf{N}(b_n)-\mathbf{N}(U^n)); \nonumber \\
&\mathbf{f}_1=\Delta t^{-2} \mathbf{L}^{-3} [-4-\Delta t \mathbf{L} + e^{\Delta t \mathbf{L}}(4-3\Delta t \mathbf{L} +(\Delta t \mathbf{L})^2 ) ]; \nonumber \\
&\mathbf{f}_2=2\Delta t^{-2} \mathbf{L}^{-3} [2+\Delta t \mathbf{L} + e^{\Delta t \mathbf{L}}(-2+\Delta t \mathbf{L} ) ];  \nonumber \\
&\mathbf{f}_3=\Delta t^{-2} \mathbf{L}^{-3} [-4-3 \Delta t \mathbf{L} - (\Delta t \mathbf{L})^2+ e^{\Delta t \mathbf{L}}(4-\Delta t \mathbf{L}) ]; \nonumber \\
&U^{n+1}=e^{\Delta t \mathbf{L}} U^n +\mathbf{f}_1\mathbf{N}(U^n)+\mathbf{f}_2[\mathbf{N}(a_n)+\mathbf{N}(b_n)]+\mathbf{f}_3\mathbf{N}(c_n).
\end{align}
We note that the potential singularity in the term $\mathbf{L}^{-1} (e^{\mathbf{L} \Delta t/2} - \mathbf{I})$ at the $0th$ Fourier node can be resolved by using the contour integrals from \cite{KT05}.

In our simulation, these two methods produce almost the same solution profiles. We also mention that there are other higher order conservative schemes, which can be applied in this study  as well, e.g., \cite{Yang2021}. Since the numerical stability and accuracy are not our main concern in this paper, we only use the most efficient scheme as described above.


\bibliographystyle{acm}
\bibliography{references}

\end{document}